\documentclass{siamltex}
\pdfoutput=1 

\usepackage{amssymb}
\usepackage{amsmath}
\usepackage[final]{graphicx}
\usepackage{psfrag}
\usepackage{epsfig}
\usepackage{latexsym}
\usepackage{exscale}
\usepackage{color}

\newcommand{\bN}{\mathbb{N}}
\newcommand{\bR}{\mathbb{R}}
\newcommand{\bC}{\mathbb{C}}
\newcommand{\eps}{\epsilon}
\newcommand{\veps}{\varepsilon}
\newcommand{\vfi}{\Phi}%{\varphi}

\newcommand{\bw}{\bar{w}}

\newcommand{\cL}{\mathcal{L}}

\newcommand{\real}{{\rm Re}\,}
\newcommand{\imag}{{\rm Im}\,}
\newcommand{\sgn}{\,{\rm sgn}\,}
\newcommand{\cC}{\,{\mathcal C}\,}

\newcommand{\la}{\lambda}
\newcommand{\cs}{c^*}

\newcommand{\cP}{\,{\mathcal P}\,}
\newcommand{\cQ}{\,{\mathcal Q}\,}
\newcommand{\cPQ}{\,{\mathcal PQ}\,}
\newcommand{\Pcon}{\stackrel{\cP}{\rightarrow}}
\newcommand{\Qcon}{\stackrel{\cQ}{\rightarrow}}
\newcommand{\PQcon}{\stackrel{\cPQ}{\longrightarrow}}

\newtheorem{ilustracion}{Example}
\newtheorem{nota}{Remark}

\title{Approximating travelling waves by equilibria of non local equations}
\author{Jose M. Arrieta\thanks{Departamento de Matem\'atica Aplicada.
Universidad Complutense de Madrid, 28040 Madrid, Spain.  E-mail: ({\tt arrieta@mat.ucm.es})\-.
Partially supported by Grant MTM2009-07540 MICINN,  Grant GR58/08 UCM-BSCH Grupo 920894 and PHB2006-003 PC, MICINN, Spain} \and Mar\'ia
L\'opez-Fern\'andez\thanks{Institut f\"ur Mathematik. Universit\"at Z\"urich.
Winterthurerst. 190, CH-8057 Zurich, Switzertland
E-mail:({\tt maria.lopez@math.uzh.ch}). Partially supported by grants MTM 2008-03541 and  MTM 2010-19510, MICINN, Spain.}
\and{Enrique Zuazua
\thanks{Address 1: BCAM - Basque Center for Applied Mathematics, Bizkaia Technology Park,
B.500 E48160, Derio, Basque Country, Spain. E-mail: ({\tt
zuazua@bcamath.org}),  Address 2:  Ikerbasque Research Professor, Ikerbasque -
Basque Foundation for Science, E48011, Bilbao, Basque Country,
Spain. Partially supported by
Grant MTM2008-03541 of the MICINN, Spain,  the ERC Advanced Grant FP7-246775 NUMERIWAVES, the ESF Research Networking Programme OPTPDE
and the Grant PI2010-04 of the Basque Government.}}}

\begin{document}

\maketitle

\renewcommand{\thefootnote}{\fnsymbol{footnote}}

%\footnotetext[2]{Departamento de Matem\'aticas, Universidad
%Aut\'onoma de Madrid, Madrid, Spain. ~E-mail: {\tt \{maria.lopez,
%enrique.zuazua\}@uam.es}. Both authors have been partially supported
%by DGI-MCYT under project MTM 2005-00714, cofinanced by FEDER funds,
%and the SIMUMAT project  S-0505/ESP/0158 of the Council of Education
%of the Regional Government of Madrid (Spain). The first author has
%been also partially supported by DGI-MCYT under project MTM
%2007-63257 and the second author by the DOMINO Project
%CIT-370200-2005-10 in the PROFIT program.}

%%----------------------------------------------------------------

\begin{abstract}
We consider an evolution equation of parabolic type in $\bR$ having a
travelling wave solution. We perform an appropriate change of variables which
transforms the equation into a non local evolution one having a travelling wave
solution with zero speed of propagation with exactly the same profile as the
original one. We analyze the relation of the new equation with the original one
in the entire real line. We also analyze the behavior of the non local problem
in a bounded interval with appropriate boundary conditions and show that it has
a unique stationary solution which is asymptotically stable for large enough
intervals and that converges to the travelling wave as the interval approaches
the entire real line. This procedure allows to compute simultaneously the
travelling wave profile and its propagation speed avoiding moving meshes, as we
illustrate with several numerical examples.
\end{abstract}

%

%\begin{abstract}
%We introduce a change of coordinates allowing to capture in a fixed
%reference frame the profile of travelling wave solutions for
%nonlinear parabolic equations. For nonlinearities of bistable type
%the asymptotic travelling wave profile becomes an equilibrium state
%for the augmented reaction-diffusion equation. In the new equation,
%the profile of the asymptotic travelling front and its propagation
%speed emerge simultaneously as time evolves. Several numerical
%experiments illustrate the efficiency of the method.
%\end{abstract}

\begin{keywords}
travelling waves, reaction--diffusion equations, implicit coordinate-change,  non-local equation, asymptotic stability, numerical approximation.
\end{keywords}
\begin{AMS}
35K55, 35K57, 35C07
\end{AMS}

%%----------------------------------------------------------------

\section{Introduction}\label{sec_intro}

We address the problem of the analysis and effective computation of travelling wave
solutions e\-mer\-ging from parabolic semilinear equations on the
real line:
\begin{equation}\label{originalpb}
\left\{
\begin{array}{l}
u_t(x,t) = \displaystyle u_{xx}(x,t) + f(u(x,t)), \qquad -\infty < x < +\infty, \quad t > 0, \\[5pt]
u(x,0) = u_0(x).
\end{array}
\right.
\end{equation}
We assume that $f \in C^1$ with $f(0)=f(1)=0$, so that $u=0$ and
$u=1$ are stationary solutions of \eqref{originalpb}. Under these assumptions,
if the initial data $u_0$ is piecewise
continuous and $0 \le u_0 \le 1$, there exists a unique bounded
classical solution $u(x,t)$ defined for all $t>0$ and,
due to the maximum principle, $0\le
u(x,t) \le 1$ for all $x,t$.

%{\color{red}
%This paper is devoted to develop a systematic analytical method
%allowing to approximate travelling wave solutions of the form
%$u(x,t)=U(x-ct)$ as time evolves. }

A travelling wave is a solution of the type $u(x,t)=\Phi(x-ct)$ where the function $\Phi$ is
the profile of the travelling wave and $c$ is the speed of propagation of the wave. For instance, if $c>0$ (resp. $c<0$) the
solution will consist of the profile $x\to \Phi(x)$ travelling in space to the right (resp. left) with speed $|c|$. Proofs of the existence of this
kind of solutions can be found in \cite{AroWei78, FiMc77, Henry} among others.
%But here we are interested in developing methods that could be of
%easy adaptation at the numerical level. Our main goal is
%to build algorithms allowing to obtain, simultaneously, the profile
%of the travelling waves and their velocity of propagation.

The asymptotic profile $\Phi$, when it exists, will have finite limits
at $\pm \infty$, either $\Phi(-\infty)=0$, $\Phi(\infty)=1$ or
$\Phi(-\infty)=1$, $\Phi(\infty)=0$. In the first case $\Phi$ will be a
solution to
\begin{equation}\label{ode}
\left\{\begin{array}{l}
\Phi''(\xi)+c\Phi'(\xi)+f(\Phi(\xi)) = 0, \qquad -\infty < \xi < +\infty, \\[5pt]
0\le \Phi \le 1, \qquad \Phi(-\infty)=0,\ \Phi(+\infty)=1,\qquad  \Phi'>0,
\end{array} \right.
\end{equation}
that is called a $[0,1]$-{\em wave front}. The monotonicity
condition on the profile $\Phi$ is not a restriction but rather an
intrinsic property of the travelling wave profiles, as it is shown
in \cite[Lemma 2.1.]{FiMc77}.

Note also that, by simply making the change of coordinates $\xi \to
-\xi$, to every pair $(\Phi,c)$ with $\Phi$ a monotone increasing
$[0,1]$-{\em wave front} corresponds a monotone decreasing
$[1,0]$-{\em wave front} with propagation speed $-c$. Thus, all what
follows applies to monotone decreasing solutions as well.

These profiles are well known to have the property of attracting,
as $t \to \infty$, the dynamics of a significant class
of solutions of the Cauchy problem \eqref{originalpb}, see for
instance Theorem~\ref{th:FiMc} below (from \cite{FiMc77}). In
\cite{FiMc77} it is proven that if $f$ is of bistable type satisfying
\begin{equation}\label{hipof_intro}
\left\{
\begin{array}{l}
f(0) = f(1) = 0, \\
f'(0) < 0, f'(1) < 0, \\
\exists  \alpha\in (0,1),\hbox{ s.t } f(u) < 0, \quad \mbox{for }\ u\in(0,\alpha), \quad f(u) > 0, \quad
\mbox{for }\ u\in(\alpha,1).
\end{array}\right.
\end{equation}
then, for a certain set of initial data $u_0$, the solution $u(x,t)$
of \eqref{originalpb} evolves into a travelling wave $\Phi(x-x_0-ct)$,
for a certain $x_0\in \bR$ depending on the initial datum $u_0$,
i.e.,
\begin{equation}\label{utoU_intro}
|u(x,t)-\Phi(x-x_0-ct)| \to 0, \qquad \mbox{as }\ t \to \infty.
\end{equation}
Further, the convergence in \eqref{utoU_intro} is shown to be
uniform in $x$ and exponentially fast in $t$.

From the numerical point of view, one of the main difficulties in
the approximation of these asymptotic solutions $\Phi$ and their
propagation speed $c$ is the need of setting a finite computational
domain. While the solution $u$ evolves into $\Phi$, it also moves left
or right at velocity $c$ and it eventually leaves the chosen finite
computational domain. A natural approach is then to perform the
change of variables $u(x,t)=v(x-ct,t)$, so that the resulting
initial value problem for $v$ converges to a stationary solution,
i.e, $v_t \to 0$ as $t\to\infty$. However, in general, the value of
$c$ is not known {\em a priori}.

The issue of having a priori characterizations of the velocity of
propagation $c$ then plays an important role from a computational
viewpoint. It has been addressed in a number of articles. For
instance, in \cite{VVV} explicit mini-max representations for the speed of
propagation $c$
are provided.
% More precisely, denoting by $C_{0}$ the set of
%absolutely continuous mappings on $[0,1]$ satisfying
%$$
%P(0)=0,\quad P(U)>0 \quad \mbox{for } \quad U\in (0,1],
%$$
%and $C_{1}$ the set of absolutely continuous mappings on $[0,1]$
%satisfying
%$$
%P(1)=0,\quad P(U)>0 \quad \mbox{for } \quad U\in [0,1),
%$$
%it is shown that
%\begin{equation}\label{c_Volperts}
%\begin{array}{l}
%c = \displaystyle \inf_{P\in C_{0}} \left( \sup \left ( P'(U) +
%\frac{f(U)}{P(U)} : U\in[0,1] \mbox{ such that } P'(U) \mbox{ exists
%} \right) \right) \\[1em]
%\ \  =\displaystyle  \sup_{P\in C_{1}} \left( \inf \left( P'(U) +
%\frac{f(U)}{P(U)} : U\in[0,1] \mbox{ such that } P'(U) \mbox{ exists
%} \right) \right).
%\end{array}
%\end{equation}
Unfortunately, the expressions in \cite{VVV}  are difficult
to handle in practice for the effective computation of $c$ and $\Phi$.

In the present paper, we follow and further develop the approach
introduced in \cite{BeTh04}, where a new unknown $\gamma(t)$ is
added to the problem, to perform the change of variables
\begin{equation}\label{cambiov}
u(x,t) = v(x-\gamma(t),t).
\end{equation}
Then, $v$ satisfies the equation
$$
v_t(x,t) = v_{xx}(x,t) + \gamma'(t)v_x(x,t) + f(v).
$$
Our goal is to determine $\gamma=\gamma (t)$ a priori so as to ensure that, as
time evolves, it converges to the asymptotic speed $c$ of the travelling wave.
Of course, in order to compensate for the additional unknown $\gamma$, one has
to add a so called ``phase condition", that is, an additional equation linking
$v$ and $\lambda:=\gamma'$. Two different possibilities were proposed in
\cite{BeTh04}. One of them consists on minimizing the $L^2$-distance of the
solution $v$ to a given template function $\hat{v}(x)$, which must satisfy
$\hat{v}-\Phi \in H^1(\bR)$. This approach leads to a Partial Differential
Algebraic Equation of the form
\begin{equation}\label{template_modifpb}
\left\{
\begin{array}{l}
v_t = \displaystyle v_{xx} + \lambda v_x+ f(v), \quad -\infty < x <
\infty, \quad t>0,
\\[1em]
0 = \langle \hat{v}',v-\hat{v} \rangle,
\\[1em]
v(x,0) = u_0(x),
\end{array}\right.
\end{equation}
where $\langle \cdot, \cdot \rangle$ denotes the inner product in
$L^2(\bR)$. This system is locally (in time) equivalent to the
original one \eqref{originalpb}, provided that the implicit function
theorem can be applied to the equation
$$
\vfi(\gamma,t):=\langle \hat{v}',u(\cdot+\gamma(t))-\hat{v}
\rangle=0,
$$
to obtain $\gamma(t)$ and $v(x,t)=u(x+\gamma(t),t)$, once $u$ is given, see
\cite[Theorem 2.9]{BeTh04}. The approximation properties of
\eqref{template_modifpb} and its numerical discretizations have been studied in
detail in \cite{Th05, Th08,Th08-II}. %In these references, the truncation to a
%finite interval and the semidiscretisation by using finite differences of
%\eqref{template_modifpb} are considered. For this, the existence of an
%asymptotically stable equilibrium is proven and rates of convergence to the
%equilibrium of the original problem \eqref{originalpb} in terms of the
%discretization parameters are provided.
This approach is useful as long as a suitable template mapping $\hat{v}$, close
enough to $\Phi$, is available and the initial data $u_0$ belongs to
$\Phi+H^1(\bR)$, too. From the computational point of view, the
semidiscretization in space of \eqref{template_modifpb} leads to a Partial
Differential Algebraic Equation of index 2. This means that it is necessary to
differentiate twice the algebraic constraint in order to eliminate it and get
the underlying Partial Differential Equation, see for instance \cite[Chap.
VII]{HaiWan}.

A more global approach, also proposed in \cite{BeTh04}, is obtained
by minimizing $\|v_t\|_2$. This yields the augmented system
\begin{equation}\label{vt_modifpb}
\left\{
\begin{array}{l}
v_t = \displaystyle v_{xx} + \lambda v_x+ f(v), \quad -\infty < x <
\infty, \quad t>0,
\\[1em]
0 = \lambda \langle v_x, v_x \rangle + \langle f(v), v_x \rangle,
\\[1em]
v(x,0) = u_0(x),
\end{array}\right.
\end{equation}
which is a Partial Differential Algebraic Equation of index 1.
Assuming that the products above are well defined and
$\|v_x(\cdot,t)\|_2 \neq 0$, $t>0$, the phase condition
$$
0 = \lambda \langle v_x, v_x \rangle + \langle f(v), v_x \rangle
$$
yields directly
\begin{equation}\label{gamma_intro}
\gamma(t) = -\int_0^t \frac{\langle f(u(\cdot,s)),u_x(\cdot,s)
\rangle} {\langle u_x(\cdot,s),u_x(\cdot,s) \rangle} \, ds =
-\int_0^t \frac{\langle f(v(\cdot,s)),v_x(\cdot,s) \rangle} {\langle
v_x(\cdot,s),v_x(\cdot,s) \rangle} \, ds,
\end{equation}
and one gets the following nonlocal semilinear equation
\begin{equation}\label{modifpb}
\left\{
\begin{array}{l}
v_t = \displaystyle v_{xx} - \frac{\langle f(v),v_x \rangle}{\langle v_x,v_x \rangle}\,
v_x+ f(v), \quad -\infty < x < \infty, \quad t>0,
\\[1em]
v(x,0) = u_0(x).
\end{array}\right.
\end{equation}
Definition \eqref{gamma_intro} is in fact quite natural if we notice
that, after multiplicating by $\Phi'$ in \eqref{ode} and integrating
along the real line, we obtain
\begin{equation} \label{cquotient}
c = - \frac{\langle f(\Phi),\Phi' \rangle}{\langle \Phi',\Phi' \rangle}.
\end{equation}
Thus, the equation for $\Phi$ can be written in the form
\begin{equation}\label{ode_nonlocal}
\Phi'' - \frac{\langle f(\Phi),\Phi' \rangle}{\langle \Phi',\Phi' \rangle} \Phi' +
f(\Phi) = 0.
\end{equation}
Equation \eqref{modifpb} is a time evolving version of
\eqref{ode_nonlocal}. It is natural to expect that, as $t\to
\infty$, system \eqref{modifpb} will yield both the speed of
propagation and the profile of the travelling wave.

We notice that \eqref{cquotient} is equivalent to
\begin{equation} \label{cquotient2}
c = - \frac{F(1)}{\langle \Phi',\Phi' \rangle},
\end{equation}
with
\begin{equation}\label{potencial}
F(u)=-\int_0^u f(s)\, ds.
\end{equation}
This observation leads in a natural way, to
the following alternative definition of $\gamma(t)$ in
\eqref{cambiov}
\begin{equation}\label{gamma_intro2}
\gamma(t) = -\int_0^t \frac{F(1)} {\langle v_x(\cdot,s),v_x(\cdot,s)
\rangle} \, ds,
\end{equation}
which yields the nonlocal semilinear problem
\begin{equation}\label{modifpb2}
\left\{
\begin{array}{l}
v_t = \displaystyle v_{xx} - \frac{F(1)}{\langle v_x,v_x \rangle}\,
v_x+ f(v), \quad -\infty < x < \infty, \quad t>0,
\\[1em]
v(x,0) = u_0(x).
\end{array}\right.
\end{equation}
We will show that \eqref{modifpb2} enjoys in fact similar properties
to those of \eqref{modifpb} and, to our knowledge, it has never been
used in practice to approximate travelling waves and their
propagation speed.

In \cite{BeTh04}, the use of \eqref{gamma_intro} or, in other words,
\eqref{modifpb} or \eqref{modifpb2}, is regarded to be particularly useful near relative
equilibria and good numerical results are reported.  Moreover, observe that analyzing these two equations
 requires less {\em a priori}
knowledge of the asymptotic state than analyzing  \eqref{template_modifpb} (no
template function in $\Phi+H^1(\bR)$ is required), and it allows to consider
more general initial data. Moreover, equation \eqref{modifpb} is simpler to
approximate numerically than \eqref{template_modifpb}. However, to our
knowledge, no rigorous asymptotic analysis of the modified equation
\eqref{modifpb}, \eqref{modifpb2} seems to be available. In fact, in later
works by the authors \cite{BeTh07,Th08,Th08-II}, where the effects of
discretizations are also taken into account, this kind of phase condition is
not analyzed, focusing only on the extended system \eqref{template_modifpb}.

In this work, we set necessary conditions for the well-posedness of these two
new initial value problems \eqref{modifpb}, \eqref{modifpb2} and analyze the
relation between these two problems and the original one \eqref{originalpb}. We
will prove that both problems  \eqref{modifpb}, \eqref{modifpb2} have the one
parameter family  of  travelling waves with $c=0$ speed of propagation
$\Phi(x-a)$,  $a\in \bR$, (``standing wave'') with the same profile $\Phi$ of the
original equation  \eqref{originalpb}.  Moreover,  under appropriate
assumptions on the initial data $u_0$,  we prove that $\gamma'(t)\to c$, as
$t\to\infty$ (and we recover the speed of propagation of the travelling wave
of the original problem), and the solutions of \eqref{modifpb},
\eqref{modifpb2} converge exponentially fast to one of these standing waves.
%{\color{red}
%To our knowledge this is the first rigorous result in this setting, obtained in part by using and extending
%the PDE approach and the results in \cite{FiMc77}, instead of the dynamical system techniques used in
%\cite{BeTh04,Th08,Th08-II}, where the modified problems are formulated as partial differential algebraic
%equations. As we shall see, at least for the
%theoretical asymptotic analysis of \eqref{modifpb}, the PDE approach
%turns out to be very useful. }

Once the modified problems \eqref{modifpb}, \eqref{modifpb2} are understood and
shown to converge to an equilibrium state with the same profile $\Phi$ as the
travelling wave for \eqref{originalpb}, the problem of its numerical
approximation arises naturally, but can be addressed quite more easily because,
now, one does not need to address the issue of moving the frames as time
evolves. To this end it is necessary to truncate the spatial domain and add
some reasonable artificial boundary conditions.  This motivates the analysis of
problems  \eqref{modifpb},  \eqref{modifpb2} in a bounded spatial interval
$(a,b)$ with certain ``artificial'' boundary conditions.  We have chosen non homogeneous boundary conditions of Dirichlet type which emulate
the behavior of the travelling wave in the complete real line, that is,
\begin{equation}\label{modifpb_interval-intro}
\left\{
\begin{array}{l}
\displaystyle  v_t = v_{xx} - \frac{F(1)}{\|v_x(\cdot)\|_{L^2(a,b)}^2} v_x +
f(v),
\qquad x \in (a,b) , \quad t>0, \\[1em]
%\displaystyle \la(z(\cdot)) = - \frac{F(1)}{\|z(\cdot)\|_{L^2(a,b)}^2}\in \bR, \\[1em]
v(a,t) = 0 ; \ v(b,t) = 1, \qquad t>0, \\
v(x,0) = u_0(x), \qquad x \in [a,b].
\end{array}
\right.
\end{equation}
Observe that when restricting both equations  \eqref{modifpb},  \eqref{modifpb2}
to a bounded interval and impossing $v(a)=0$, $v(b)=1$ we obtain in both cases
the very same
equation, which is the one given above in \eqref{modifpb_interval-intro}.

We analyze equation \eqref{modifpb_interval-intro} and show that with the nonlinearity $f$ satisfying
\eqref{hipof_intro} we have a unique stationary state $\Phi_{(a,b)}$ with  $0\leq \Phi_{(a,b)}\leq 1$. Moreover, this
stationary state, when normalized so that $\Phi_{(a,b)}(0)=1/2$,  will converge to the profile of the travelling wave of equation \eqref{modifpb},  \eqref{modifpb2}
as $(a,b)\to (-\infty,+\infty)$ (see the details in Section \ref{subsec-convergence-stationary}).
We also analyze the stability properties of this stationary state. In order to
accomplish this, we will need to analyze the spectral properties of the linearization of \eqref{modifpb_interval-intro}
around the stationary state, which means to analyze the spectra of the ``nonlocal operator''
$$
\begin{array}{l}
\displaystyle Lw=w_{xx}-\frac{F(1)}{\|\Phi'_{(a,b)}\|_{L^2(a,b)}}w_x+f'(\Phi_{(a,b)}(x))w \\
\displaystyle\qquad\qquad\qquad\qquad- \frac{2F(1)}{\|\Phi'_{(a,b)}\|_{L^2(a,b)}^2}\Phi'_{(a,b)} \int_a^b w'(x)\Phi'_{(a,b)}(x)dx.
\end{array}
$$

This task is not a simple one. There are in the literature several works which
analyze the spectra of operators of the type above, see
\cite{DaDo06a,DaDo06b,Do08,Frei94,Frei99}, but none of them are conclusive
enough to characterize it completely in our case.

Nevertheless, we will be able to show that $\sigma(L)\subset \{z\in\bC,
\real(z)<-\kappa(a,b)\}$ for certain $\kappa(a,b)>0$ when the length of the
interval is large enough (that is, for $b-a\to +\infty$). We will obtain this
result via a perturbative argument, viewing the operator $L_{(a,b)}$ as a
perturbation of the operator on $(a,b)=\bR$, that is
$$
\begin{array}{l}
\displaystyle Lw=w_{xx}-\frac{F(1)}{\|\Phi'\|_{L^2(\bR)}}w_x+f'(\Phi(x))w \\
\displaystyle\qquad\qquad\qquad\qquad- \frac{2F(1)}{\|\Phi'\|_{L^2(\bR)}^2}\Phi' \int_\bR w'(x)\Phi'(x)dx,
\end{array}
$$for which the spectra is easier to characterize, since $\Phi'$ is the eigenfunction associated to the eigenvalue $0$ or
the operator $L_\infty$.

We will conclude in this way the asymptotic stability of the stationary
solution $\Phi_{(a,b)}$ for large enough intervals $(a,b)$.

The analysis in the present work is performed for non-homogeneous Dirichlet
boundary conditions. This choice is justified since, somehow, it imitates the
behavior of the travelling wave for large enough intervals. Nevertheless, other
boundary conditions may be suitable to approximate the travelling wave although
the dynamics of system \eqref{modifpb_interval-intro} with these other boundary
conditions may differ from from the case treated in this paper. As a matter of
fact an analysis of the differences and similarities for different boundary
conditions will be important and will be carried out in a future work.

Let us notice that the idea of performing a change of coordinates so that the
front of the asymptotic profile $\Phi$ remains eventually fixed in space
appears also in \cite{TaSlee}, where a different change of variables is
considered. This alternative change of variables is based on the fact that it
also holds
\begin{equation} \label{cintegral}
c = - \int_{-\infty}^{\infty} f(\Phi(\xi))\,d\xi.
\end{equation}
Formula \eqref{cintegral} is readily obtained after integration along the real
line in \eqref{ode} and leads to the change of coordinates $u(x,t) =
v(x-\tilde{\gamma}(t))$ with
$$
\tilde{\gamma}(t) = -\int_0^t \int_{-\infty}^{\infty} f(u(x,s))\,dx
\,ds.
$$
The convergence of the solutions of the resulting equation to an equilibrium is not proved in
\cite{TaSlee}.  In this paper, we do not study this particular change of variables although we
 expect that by adjusting the techniques we develop here we will be able to obtain similar
 results.  As a mattter of fact, we regard the analysis included in this paper as a general
 technique that, with the appropriate adjusments to the different possible changes of variables
 which make the velocity of propagation implicit in the equation, will yield in an effective
 way both, the speed of propagation and the profile of the travelling wave.

The paper is organized as follows. Sections \ref{section-estimates} and
\ref{subsec_changev} are devoted to the deduction and analysis of the problem
in the whole real line.  In Section~\ref{section-estimates} besides recalling
the result on existence of travelling waves from \cite{FiMc77}, we also obtain
several important estimates of the solutions of the original problem
\eqref{originalpb} when the initial condition $u_0$ satisfies $u_0\in
L^\infty(\bR)$ and $\partial_x u_0\in L^p(\bR)$, $1\leq p<\infty$.

 The change of variables that leads to the modified problema
\eqref{modifpb}, \eqref{modifpb2} is considered in Section~\ref{subsec_changev}, where
we establish a fundamental relation between these new problems and the
original one \eqref{originalpb}.  The estimates obtained in Section \ref{section-estimates} are used in
a crucial way in this section. We show the asymptotic stability, with asymptotic phase,
 of the family of travelling wave solutions of the nonlocal problems \eqref{modifpb}, \eqref{modifpb2},
 see Theorem \ref{th:v}.

The next two sections, Section \ref{existence-bounded-interval} and Section
\ref {stability-bounded-interval} are devoted to the nonlocal problem in a
bounded interval. In Section \ref{existence-bounded-interval} we obtain the
existence and uniqueness of an stationary solution of problem
\eqref{modifpb_interval-intro} and  show that the stationary solution converges
to the profile of the travelling wave solution in the entire real line. In
order to accomplish this, we will need to perform a careful analysis of the
behavior of the associated local problems in a bounded domain, see Subsection
\ref{subsec-local-bounded}, and then to relate the results obtained for the
local and nonlocal problem, see Subsection \ref{subsec-nonlocal-bounded}. The
convergence of the stationary states to the travelling wave is obtained in
Subsection \ref{subsec-convergence-stationary}. In Section
\ref{stability-bounded-interval} we show the asymptotic stability of the
stationary state of the non local problem in a bounded interval. We analyze
first the properties of the spectra of the linearized non local equation in a
bounded interval, see Subsection \ref{subsec-spectral-bounded}, and in
the complete real line, see Subsection \ref{subsec-spectral-nonlocal-bounded}.
In Subsection \ref{subsec-spectral-convergence} we obtain the asymptotic
stability of the stationary states for large enough intervals.  This is
obtained through a spectral perturbation argument.

Finally in Section \ref{sec_numerical} we include several mumerical examples which
 illustrate the efficiency of the methods developed in this article to capture the asymptotic
 travelling wave profile and its velocity of propagation.

\par\bigskip{\bf Ackowledgements.} We would like to thank Erik Van Vleck, Pedro Freitas,
Wolf-J\"urgen Beyn and Nicholas Alikakos for the different discussions we had with them
and for pointing out several important aspects of the subject of this paper.

\section{Estimates for the original problem \eqref{originalpb}}
\label{section-estimates}

%\subsection{Known results and outline of useful proofs}\label{subsec_pre}
We start reviewing some of the results in \cite{FiMc77} and \cite{Henry} on the
existence and behavior of travelling wave solutions of \eqref{originalpb}.

The following theorem of \cite{FiMc77} ensures the existence and
uniqueness of an asymptotic travelling front for the original
problem \eqref{originalpb} under quite general assumptions on the
initial data $u_0$.

\begin{theorem}\label{th:FiMc}
Let $f\in C^1[0,1]$ satisfying \eqref{hipof_intro}. Then there exists a unique
(except for translations) monotone travelling front with range $[0,1]$, i.e.,
there exists a unique $\cs$ and a unique (except for translations) monotone
solution $\Phi$ of \eqref{ode}.
%, with the limits $\Phi(-\infty)=0$ and $\Phi(+\infty)=1$.

Suppose that $u_0$ is piecewise continuous, $0\le u_0(x)\le 1$ for
all $x\in \bR$, and
\begin{equation}\label{hipou0}
\liminf_{x\to +\infty} u_0(x) > \alpha, \qquad \limsup_{x\to
-\infty} u_0(x) < \alpha.
\end{equation}

Then there exist $x_0\in \bR$, $K,\omega >0$, such that the solution $u$ to
\eqref{originalpb} satisfies
\begin{equation}\label{fifemcl_expconv}
|u(x,t)-\Phi(x-\cs t-x_0)| < Ke^{-\omega t}, \qquad x\in\bR,\quad t>0.
\end{equation}
Furthermore, $\cs \ge 0 \ (resp. \le 0)$ when $F(1)=-\int_0^1 f(s)\, ds \ge 0 \
(resp. \le 0)$.

\end{theorem}

{\sc Summary of the Proof of Theorem~\ref{th:FiMc} in \cite{FiMc77}}. For
$\cs$ in the statement of Theorem~\ref{th:FiMc}, set
\begin{equation}\label{defw}
w(x,t)= u(x + \cs t,t),
\end{equation}
which fulfils
\begin{equation}\label{pbw}
\left\{
\begin{array}{l}
w_t(x,t) = \displaystyle w_{xx}(x,t) + \cs  w_x(x,t) + f(w(x,t)), \quad -\infty
< x < \infty, \ t > 0, \\[5pt]
w(x,0) = u_0(x).
\end{array}
\right.
\end{equation}
The proof is then based on the construction of a
Liapunov functional for equation \eqref{pbw}. The main tools are a
priori estimates and comparison principles for parabolic equations
\cite[Theorem 4 of Chapter 7 and Theorem 5 of Chapter 3]{Fried}.
The following two Lemmas are important intermediate steps in this
construction and will be used in Section~\ref{subsec_changev}.

\begin{lemma}\label{lema41FiMc}
Under the assumptions of Theorem~\ref{th:FiMc}, there exist
constants $x_1$, $x_2$, $q_0$ and $\mu$, with $q_0, \mu > 0$, such
that
\begin{equation}\label{lemFi}
\Phi(x-x_1) - q_0 e^{-\mu t} \le w(x,t) \le \Phi(x-x_2) + q_0 e^{-\mu t}.
\end{equation}
\end{lemma}

The following lemma provides asymptotic estimates for the
derivatives of $w$.

\begin{lemma}\label{lema43FiMc}
Under the assumptions of Theorem~\ref{th:FiMc}, there exist positive
constants $\sigma, \mu$ and $C$ with $\sigma > |\cs |/2$, such that
\begin{eqnarray}\label{boundsderw}
&& |1-w(x,t)|,\ |w_{x}(x,t)|, \ |w_{xx}(x,t)|,\ |w_t(x,t)| \nonumber \\
&&\hspace{12em}  < \, C(e^{-(\cs /2+\sigma)x} + e^{-\mu t}),\quad x > 0,\ t>0;
\\
&& |w(x,t)|,\ |w_{x}(x,t)|, \ |w_{xx}(x,t)|,\ |w_t(x,t)|  \nonumber
\\&& \hspace{12em}  < \, C(e^{(\sigma - \cs /2)x} + e^{-\mu t}),\quad x < 0,\
t>0.
\end{eqnarray}
\end{lemma}

\begin{nota}\label{nota_exprate}
Although the result stated in Theorem~\ref{th:FiMc} is very general in
terms of the initial data in \eqref{originalpb}, its proof yields little hint
about the exponent $\omega$ in the exponential estimate
\eqref{fifemcl_expconv}. In this sense the study accomplished in \cite{Henry}
is clearer. Following a different approach, the convergence result
\eqref{fifemcl_expconv} is also proven in \cite{Henry}, although for a less
general class of initial data. Once an equilibrium $\Phi$ for \eqref{pbw} is
shown to exist, the uniform convergence of $w$ to a shift of $\Phi$ is obtained
by analyzing the linearization about $\Phi$ of the equation in \eqref{pbw}.
More precisely, the spectrum of the operator
\begin{equation}\label{linop}
Lw := w''+\cs w'+f'(\Phi)w, \qquad \infty < x < \infty.
\end{equation}
is considered. By \cite[Theorem A.2 of Chapter 5]{Henry}, the
essential spectrum of $L$ lies in $\real z \le -\beta$
with
\begin{equation}\label{beta_esspec}
\beta=\min\{-f'(0),-f'(1)\} > 0.
\end{equation}
The rest of $\sigma(L)$, i.e., the set of isolated eigenvalues of $L$ of finite
multiplicity, is also shown to be negative but the eigenvalue 0, which turns
out to be simple. This, by \cite[Exercise 6 of Section 5.1]{Henry}, yields the
exponential rate of convergence in \eqref{fifemcl_expconv}. The rate of
convergence can be taken as any $\omega<\omega_0$ where
\begin{equation}\label{exprate}
\omega_0= \min\{ \beta, \gamma\},
\end{equation}
where $-\gamma < 0$ is the spectral abscissa, i. e. the largest real part of any non zero eigenvalue of $L$. In fact,
this analysis of $\sigma(L)$ yields the asymptotic stability with
asymptotic phase of the family of equilibria
$$
\{\Phi(\cdot-x_0) : x_0 \in \bR \}$$
of \eqref{pbw}.

We finally notice that from the proof of Lemma~\ref{lema41FiMc} accomplished in
\cite{FiMc77} it follows that the constant $\mu$ in Lemma \ref{lema41FiMc} and
Lemma \ref{lema43FiMc} can be chosen as close to $\beta$ as we wish.
This implies that we can choose any $\mu$ satisfying
%\eqref{exprate}, but close enough to $\min\{ \beta, \gamma\}$ and we may choose $\mu$  so that
\begin{equation}\label{cotamu}
\mu<\omega_0.
\end{equation}
The above bound will be used in Section~\ref{subsec_changev}.
\end{nota}

\par\bigskip

We next show an existence and uniqueness result  for the original Cauchy problem \eqref{originalpb} in the spaces
\begin{equation}\label{Wdot}
\dot{W}^{1,p}(\bR) = \{ u \in W^{1,p}_{loc}(\bR) : \partial_x u \in L^p(\bR) \},
\end{equation}
for $1\le p\le \infty$.
  This result is slightly more general than what we strictly need to ensure the well-posedness of \eqref{modifpb}.
%

%We next show an existence and uniqueness result of solutions for the original
%Cauchy problem \eqref{originalpb}. This result is slightly more general than
%what we strictly need to ensure the well-posedness of \eqref{modifpb}, since we
%prove it for $\partial_x u_0 \in L^p(\bR)$, for any $1\le p\le\infty$. We will
%work in the framework of the functional spaces
%\begin{equation}\label{Wdot}
%\dot{W}^{1,p}(\bR) = \{ u \in W^{1,p}_{loc}(\bR) : \partial_x u \in L^p(\bR) \},
%\end{equation}
%for $1\le p\le \infty$.

\begin{proposition}\label{prop:exun}
Let $f\in C^1(\bR,\bR)$ satisfying $f(0)=f(1)=0$.  Let $1\le
p\le\infty$ and $u_0 \in L^{\infty}(\bR)\bigcap \dot{W}^{1,p}(\bR)$
with $0\leq u_0\leq 1$ a.e. $x\in \bR$. Then,

\par\medskip\noindent (i) there exists a unique mild solution $u\in
L^{\infty}([0,\infty)\times\bR)\bigcap
C((0,\infty);\dot{W}^{1,p}(\bR))$ of the Cauchy problem
\eqref{originalpb}. Moreover, this solution satisfies $0\leq u\leq
1$, is a classical solution for $t>0$ and has the following
regularity $u\in C(0,\infty, C^{1,\eta}(\bR))$, for all $\eta<1$.
% Also, if $u_0$ is not a constant function then its solution is not a constant function for all $t>0$.

\par\medskip\noindent (ii) In case $p=1$ and when the function $f$
%and
%the initial condition $u_0$ satisfy also the conditions of Theorem
%\ref{th:FiMc}, that is $f$
satisfies (\ref{hipof_intro}) and the initial condition $u_0$
satisfies (\ref{hipou0}), then there exists $C>0$ such that
$\|u_x(\cdot,t)\|_1 \le C$, for all $t>0$.

\par\medskip\noindent (iii) In case $p=2$ and when the function $f$
%and
%the initial condition $u_0$ satisfy also the conditions of Theorem
%\ref{th:FiMc}, that is $f$
satisfies (\ref{hipof_intro}) and the initial condition $u_0$
satisfies (\ref{hipou0}), then there exists a $\beta>0$ such that
$\|u_x(\cdot, t)\|_{2}\geq \beta$ for all $t>0$.
\end{proposition}
\par\medskip
{\em Proof.}  $(i)$
For initial data $u_0 \in L^\infty(\bR)$  the existence and uniqueness of mild solutions in $L^\infty((0,\infty),L^\infty(\bR))$
holds by  the variation of constants formula below and  standard fixed point arguments,
%Observe that if we consider an initial condition
%$u_0\in L^\infty(\bR)$ we have the existence and uniqueness of mild
%solutions in $L^\infty((0,\infty),L^\infty(\bR))$ with the use of
%the variation of constants formula and using standard fixed point
%arguments. In fact, by the regularization properties of the heat
%kernel, we have existence and uniqueness of solutions in
%$C((0,\infty),W^{1,\infty}(\bR))$, see \cite{Henry}. More precisely,
%given $u_0 \in L^{\infty}(\bR)$ the solution is
\begin{equation}\label{soldebil}
u(t) = G(\cdot,t)
* u_0 + \int_0^t G(t-s) * f(u(\cdot,s)) \, ds, \qquad t>0,
\end{equation}
where $G = G(x,t) = (4\pi t)^{-1/2} \exp(-|x|^2/4t)$ is the heat
kernel and $*$ denotes the convolution in the space variable.
Then, by the regularization properties of the heat kernel, this solution belongs to $C((0,\infty),W^{s,\infty}(\bR))$, for all $s<2$, see \cite{Henry}.
The embedding $W^{s,\infty}(\bR)\hookrightarrow C^{1,\eta}(\bR)$ for $\eta <s-1$ implies the regularity result.

On the other hand let us consider the initial value problem
\begin{equation}\label{eq_ux}
\left\{
\begin{array}{l}
q_t=q_{xx}+f'(u(x,t))q, \qquad -\infty < x < \infty, \quad t>0, \\[5pt]
q(x,0)=\partial_x u_0 (x)\in L^p(\bR).
\end{array}
\right.
\end{equation}
formally solved by the space derivative $u_x$ of $u$. Since $q(x,0) \in L^p(\bR)$, we have a unique solution $q\in
C([0,\infty),L^p(\bR))$ of \eqref{eq_ux}. In fact, $q$ is given by
\eqref{soldebil} with $f'(u(\cdot,s))q(\cdot,s)$ instead of
$f(u(\cdot,s))$. The well-known $L^p\to L^q$ estimates for the heat
equation in $\bR^N$, namely
\begin{equation}\label{LpLq}
\|G(\cdot,t)*\varphi\|_q \le Ct^{-\frac N2 \big(\frac 1p-\frac 1q
\big)} \|\varphi\|_p, \qquad 1\le p < q \le \infty,
\end{equation}
imply that also $q\in C(0,\infty,L^{\infty}(\bR))$. But then it is
\begin{equation}\label{soldebil_ux}
q(x,t) = u_x (x,t) = G(\cdot,t)* \partial_x u_0 + \int_0^t G(t-s) *
(f'(u(\cdot,s))u_x(\cdot,s)) \, ds, \qquad t>0,
\end{equation}
by the uniqueness of solutions of \eqref{eq_ux} in
$C(0,\infty,L^{\infty}(\bR))$.

Since  $f(0)=0$ and $f(1)=0$, both functions $u\equiv 0$ and
$u\equiv 1$ are strong solutions of \eqref{originalpb}. Using
standard comparison arguments, we have that if $0\leq u_0\leq 1$
then any possible solution starting at $u_0$ will lie between this
two constants functions.

%Finally, the regularity properties stated in i) are inherited by
%those of the heat kernel and can be deduced in a straightforward manner from
%\eqref{soldebil}.

$(ii)$ From \eqref{fifemcl_expconv} and i) we have that the solution
$u$ to \eqref{originalpb} approaches a travelling wave solution and
$u_x(\cdot,t) \in L^1(\bR)$ for every $t$. We next show the uniform
boundedness in time of $\|u_x(\cdot,t)\|_1$. This is equivalent to
bound $\|w_x(\cdot,t)\|_1$, for $w$ in \eqref{defw}. To this end, we
consider $h = w_x$, which satisfies the equation
\begin{equation}\label{eq_wx}
h_t = h_{xx} + \cs h_x + f'(w)h, \qquad -\infty < x < \infty.
\end{equation}
Applying that $f'$ is continuous, that both $w, \Phi \in [0,1]$,  and
\eqref{fifemcl_expconv}, we can estimate
$$
|f'(w(x,t)) - f'(\Phi(x-x_0))| \le C |w(x,t)- \Phi(x-x_0)| \le K
e^{-\omega t}, \qquad x\in \bR,\quad  t>0.
$$
This, together with the hypotheses $f'(0),f'(1)<0$, imply the
existence of $L>0$ and $t_0>0$ large enough, and $\beta>0$, so that
$$
f'(w(x,t)) \le f'(\Phi(x-x_0)) + K e^{-\omega t} \le -\beta < 0, \quad
\mbox{ for } |x|\ge L ,\quad  t\ge t_0.
$$
Multiplying formally in \eqref{eq_wx} by the sign of $h$, $\sgn(h)$,
and integrating in $\{x\in \bR: |x| \ge L\}$ gives, for every $t\ge
t_0$,
$$
\frac{d}{dt} \int_{|x|\ge L} |h(x,t)|\, dx = \int_{|x|\ge L}
(h_{xx}\sgn(h) + \cs |h|_x) \,dx + \int_{|x|\ge L} f'(w(x,t))|h(x,t)|\,
dx.
$$

This yields, by Kato's inequality (see \cite{Kato72}) and applying
estimates in Lemma~\ref{lema43FiMc} to $h=w_x$ and $h_x=w_{xx}$,
%(he visto un argumento similar en
%unas notas de Enrique y se refer\'ia a la "desigualdad de Kato"??,
%pero era s\'olo con el laplaciano y en todo $\bR^{N}$. Yo lo he
%razonado a mano suponiendo que en cada punto de cambio de signo
%$x_j$ es $h(x_j,t)=0$, $h_x(x_j,t)\le 0$ si $h$ cambia de $>0$ a
%$<0$ y $h_x(x_j,t)\ge 0$ en caso contrario.)
\begin{eqnarray*}
&& \frac{d}{dt} \int_{|x|\ge L} |h(x,t)|\, dx \\
& & \le \,  \int_{|x|\ge L} |h|_{xx}\,dx + |\cs| (\limsup_{M\to \pm
\infty} |h(M,t)| + |h(\pm L,t)|) - \beta  \int_{|x|\ge L} |h(x,t)|\, dx \\
& & \le \, \limsup_{M\to \pm \infty} (|h_x(M,t)| + |h_x(\pm L,t)|) + \tilde{C} - \beta  \int_{|x|\ge L} |h(x,t)|\, dx \\
& & \le \,  -\beta \int_{|x|\ge L} |h(x,t)|\, dx + C, \qquad t\ge
t_0,
\end{eqnarray*}
where the constant $C$ is independent of $t$ and $L$. Thus, setting
$$
g(t) =\int_{|x|\ge L} |h(x,t)|\, dx,
$$
multiplying the above inequality by $e^{\beta t}$ and integrating
from $t_0$ to $t$, we obtain
$$
g(t) \le e^{-\beta (t-t_0)} g(t_0) + \frac{C}{\beta} (1-e^{-\beta
(t-t_0)}) \le A, \qquad t\ge t_0.
$$
Finally, we apply again Lemma~\ref{lema43FiMc} to estimate
$$
\int_{\bR} |h(x,t)|\,dx \le A + \int_{|x| < L} |h(x,t)|\,dx \le A +
2L \sup_{x\in[-L,L]} |h(x,t)| \le C, \quad \mbox{for all} \ t\ge
t_0.
$$
For $t\in (0,t_0]$ we can bound directly $\|u_x(\cdot,t)\|_1$ in
\eqref{soldebil_ux} and apply Gronwall's i\-ne\-qua\-li\-ty.

The above argument can be formalized by multiplying in \eqref{eq_wx} by
$h|h|^{p-2}$ with $p>1$. In this way we can get an estimate for $\|h\|_p$ which
turns out to be independent of $p$ and then take the limit as $p\to 1$. Another
possibility is to consider a Lipschitz regularization of $\sgn(h)$. %\cite{??}.

$(iii)$ From \eqref{fifemcl_expconv} we have that the solution
approaches a travelling wave solution and therefore,
$\liminf_{t\to+\infty}\|u_x(t,\cdot)\|_2>0$, which implies that
there exists a $T_1$ and $\beta_1$ with $\|u_x(\cdot, t)\|_2\geq
\beta_1$ for all $t\geq T_1$.

On the other hand, if there exists some time $0<T<T_1$ such that
$\|u_x(\cdot, T)\|_{2}=0$ then $u(\cdot, T)$ is a constant function
and therefore $u(\cdot, t)$ is a constant function for all $t\geq
T$. To see this we just use the uniqueness of solutions and the fact
that if the initial condition is a constant function, then the
solution is a constant function in space for all forward times.
Hence, $\|u_x(\cdot, t)\|_2>0$ for all $t\in [0,T_1]$ and since this
is a compact interval and the function $t\to \|u_x(\cdot,t)\|_2$ is
continuous, then there exists a $\beta_2>0$ such that $\|u_x(\cdot,
t)\|_2\geq \beta_2$ for all $t\in [0,T_1]$.    This shows the last
part of the proposition. \endproof

\begin{nota}\label{remark_prop}
{\rm Assuming further that $f'$ is Lipschitz continuous, it is
possible to prove i) of Proposition~\ref{prop:exun} by using
standard fixed point arguments in the space $L^\infty(\bR)\cap \dot
W^{1,p}(\bR)$.}
\end{nota}

\begin{nota}\label{remark1}
{\rm Under the assumptions of Theorem~\ref{th:FiMc} and $\partial_x
u_0 \in L^1(\bR)\bigcap L^2(\bR)$, the solution $u$ of
\eqref{originalpb} approaches a solution $\Phi$ of \eqref{ode}, as
$t\to \infty$, and, by Proposition~\ref{prop:exun}, $\Phi$ must fulfil
besides the added integrability condition $\Phi'\in L^1(\bR) \bigcap
L^2(\bR)$. We show below that this is consistent with the following
properties of such a travelling front $\Phi$:
\begin{itemize}
\item[(i)] $\Phi'\in L^1(\bR)$. This is a direct consequence of the
fact that $\Phi' > 0 $, so that $|\Phi'|=\Phi'$, and $\Phi(\pm\infty) < \infty$.

\item[(ii)] $\Phi'\in L^2(\bR)$. To see this, we multiply in
\eqref{ode} by $\Phi'$ and integrate along $\bR$, obtaining
\begin{eqnarray*}
&& \cs \int_{\bR}(\Phi'(\xi))^2 \, d\xi = - \int_{\bR} \Phi''(\xi) \Phi'(\xi)
\, d\xi - \int_{\bR} f(\Phi(\xi))\Phi'(\xi) \,d\xi \\
&& \hspace{1em}=\,  -\frac 12 \int_{\bR} \frac d{d\xi} (\Phi')^2(\xi)
\, d\xi - \int_{\bR} \frac d{d\xi} F(\Phi(\xi)) \,d\xi \\
&& \hspace{1em}=\, F(\Phi(\infty))-F(\Phi(-\infty))= F(1) <\infty,
\end{eqnarray*}
where $F$ is defined in \eqref{potencial} and we used that,
necessarily, $\Phi'(\pm \infty)=0$ (see \cite{FiMc77}).

%\item[(iii)] For later purpose we give here the following
%representation for $c$ in \eqref{ode} in terms of $f$ and $\Phi$
%\begin{equation}\label{cint}
%c = -\int_{-\infty}^{+\infty} f(\Phi(\xi)) \,d\xi < \infty,
%\end{equation}
%which follows in a straightforward way after integrating in
%\eqref{ode} from $-\infty$ to $\infty$.

\item[(iii)] We also notice that $\Phi''\in L^2(\bR)$, too.
This follows again from \eqref{ode}, now after multiplication by
$\Phi''$, which leads to
\begin{eqnarray*}
&& \int_{-\infty}^{+\infty} (\Phi'')^2(\xi)\, d\xi = -
\int_{-\infty}^{+\infty} \Big( \frac \cs 2 \frac d{d\xi} (\Phi')^2
+ f(\Phi)\Phi'' \Big) \, d\xi \\ && \hspace{1em} =\,
\int_{-\infty}^{+\infty} f'(\Phi)(\Phi')^2 \, d\xi < \infty,
\end{eqnarray*}
since, by hypothesis, $f\in C^1$, $\Phi$ is bounded and we have shown
that $\Phi'\in L^2(\bR)$.

\end{itemize}

}
\end{nota}

\section{The nonlocal problem in the entire real line}\label{subsec_changev} We now turn
to the non local Cauchy problems \eqref{modifpb} and
\eqref{modifpb2}.

The following result shows how the well-posedness of equation
\eqref{modifpb} depends on the properties of the solution $u$ to the
original problem \eqref{originalpb}.

\begin{proposition}\label{propo1}
Assume that the initial data $u_0$ in \eqref{originalpb} is
piecewise continuous and $0\le u_0 \le 1$. Let $u$ be the unique
classical solution to \eqref{originalpb}. Assume further that
\begin{itemize}
\item[(i)] $u_x(\cdot,t) \in L^1(\bR) \bigcap L^2(\bR)$, for every $t\ge 0$.

\item[(ii)] there exists $\beta > 0$ such that $\|u_x(\cdot,t)\|_2 \ge \beta$
for all $t\ge 0$.
\end{itemize}

Then
\begin{equation}\label{v}
v(x,t):=u(x+\gamma_u(t),t), \qquad x\in \bR, \quad t>0
\end{equation}
with
\begin{equation}\label{gamma}
\gamma_u(t) := -\int_0^t \frac{\langle f(u(\cdot,s)),u_x(\cdot,s)
\rangle}{\langle u_x(\cdot,s),u_x(\cdot,s) \rangle} \, ds, \qquad
t>0
\end{equation}
is well defined and is a classical solution of \eqref{modifpb}.
\end{proposition}
\begin{proof}
Due to assumption $(ii)$, $u(\cdot,t)$ is not constant in space for
all $t>0$ and
\begin{equation}\label{lambda}
\lambda_u(t) := -\frac{\langle f(u(\cdot,t)),u_x(\cdot,t)
\rangle}{\langle u_x(\cdot,t),u_x(\cdot,t) \rangle}.
\end{equation}
defines a bounded and continuous mapping of $t$. The integral in the
scalar product of the numerator in \eqref{lambda} is convergent, since $u\in L^{\infty}(\bR)$, then $f(u)\in L^{\infty}(\bR)$, and,
by hypothesis $(i)$, $u_x \in L^1(\bR)$. By hypotheses $(i)$ and $(ii)$, the
scalar product in the denominator in \eqref{lambda} is also finite and strictly
positive.
The continuity follows from the fact that $u$ is a classical
solution. Thus, $\gamma_u$ in \eqref{gamma} is well defined and so
is $v$ in \eqref{v}.

From the invariance with respect to translations of the integral in the whole real line, we easily get that for every $t\ge0$ fixed,
$\langle f(v(\cdot,t)),v_x(\cdot,t) \rangle=\langle f(u(\cdot,t)),u_x(\cdot,t) \rangle$.
%
%it holds
%\begin{eqnarray*}
%\langle f(v(\cdot,t)),v_x(\cdot,t) \rangle &=& \int_{-\infty}^{+\infty} f(v(\xi,t)) v_x(\xi,t)\, d\xi \\
%&=& \int_{-\infty}^{+\infty} f(u(\xi + \gamma_u(t),t))u_x(\xi+\gamma_u(t),t)\, d\xi \\
%&=& \int_{-\infty}^{+\infty} f(u(\xi,t)) u_x(\xi,t)\, d\xi = \langle
%f(u(\cdot,t)),u_x(\cdot,t) \rangle,
%\end{eqnarray*}
In a similar way, $\|v_x(\cdot,t)\|_2^2 = \|u_x(\cdot,t)\|_2^2 $. Thus,
$\lambda_u(t) = \lambda_v(t)$, $\gamma_u(t) = \gamma_v(t)$,
\begin{equation}\label{vv}
v(x,t) = u(x+\gamma_v(t),t),
\end{equation}
and clearly $v$ fulfils \eqref{modifpb}. \hfill \end{proof}

The analysis of \eqref{modifpb2} follows the same steps as the
analysis of \eqref{modifpb} and is, in fact, simpler. Thus, we only
state here the corresponding result for \eqref{modifpb2}.
\begin{proposition}\label{propo1-2}
Assume that the initial data $u_0$ in \eqref{originalpb} is
piecewise continuous and $0\le u_0 \le 1$. Let $u$ be the unique
classical solution to \eqref{originalpb}. Assume further that
\begin{itemize}
\item[(i)] $u_x(\cdot,t) \in L^2(\bR)$, for every $t\ge 0$.

\item[(ii)] There exists $\beta > 0$ such that $\|u_x(\cdot,t)\|_2 \ge \beta$
for all $t\ge 0$.
\end{itemize}

Then,
\begin{equation}\label{v2}
v(x,t):=u(x+\gamma_u(t),t), \qquad x\in \bR, \quad t>0
\end{equation}
with
\begin{equation}\label{gamma2}
\gamma_u(t) := -\int_0^t \frac{F(1)}{\langle
u_x(\cdot,s),u_x(\cdot,s) \rangle} \, ds, \qquad t>0,
\end{equation}
and $F$ in \eqref{potencial} is well defined and is a classical
solution of \eqref{modifpb2}.
\end{proposition}

Propositions~\ref{propo1} and \ref{propo1-2} imply that the study of
the well-posedness of \eqref{modifpb} and \eqref{modifpb2},
respectively, can be reduced to a further study of the original
Cauchy problem \eqref{originalpb}. In fact, all we need is to ensure
that the solution $u$ of \eqref{originalpb} fulfils assumptions
$(i)$ and $(ii)$ of Proposition~\ref{propo1} or
Proposition~\ref{propo1-2}. As summarized below, this is
provided by Proposition~\ref{prop:exun}.

\par\bigskip

%\subsection{Travelling waves as stationary solutions for the non local modified problem}\label{sec_modpb}
We are now in the position to prove the main result of this section and one of the main results of this paper.

\begin{theorem}\label{th:v}
Under the hypotheses of Theorem~\ref{th:FiMc} and assuming further
that $u_0 \in \dot{W}^{1,1}(\bR)\bigcap \dot{W}^{1,2}(\bR)$, the
augmented problem \eqref{modifpb} is well-posed and its solution $v$
is given by \eqref{v}-\eqref{gamma}.

Let any $\bar\omega<\omega_0$, $\omega_0$ being defined
as in \eqref{exprate}.
Then, there also exist $x^*\in \bR$ and
positive constants $C_1$, $C_2$, such that

\par\medskip\noindent (i) for $\cs $ the propagation speed in
Theorem~\ref{th:FiMc} and $\lambda_v$ in \eqref{lambda} it holds
\begin{equation}\label{conver_speed}
|\lambda_v(t)-\cs | \le C_1e^{-\bar\omega t}, \qquad  t>0.
\end{equation}

\noindent (ii) For $\Phi$ the unique (except for translations) solution
to \eqref{ode}, we can estimate
\begin{equation}\label{conver_v}
|v(x,t)-\Phi(x-x^*)| < C_2 e^{-\bar\omega t}, \qquad x\in \bR, \quad t>0.
\end{equation}

\end{theorem}

The above result validates the change of variables
\eqref{v}-\eqref{gamma}, shows that $\lambda_v$ in \eqref{lambda}
converges to the asymptotic speed $c$ at an exponential rate and
provides the analogue to Theorem~\ref{th:FiMc} for $v$, since
\eqref{conver_v} is equivalent to
\begin{equation}\label{conver_u}
|u(x,t)-\Phi(x-\gamma(t)-x^*)| < C_2 e^{-\omega t}\qquad x\in \bR,
\quad t>0.
\end{equation}
Moreover, the rate of the exponential convergence is the same as the
one derived in \cite{Henry}.

\begin{proof}
By applying Proposition~\ref{prop:exun} with $p=1$ and $p=2$, we
obtain that $u_x(\cdot,t) \in L^{1}(\bR) \bigcap L^{2}(\bR)$ for $u$
the solution to \eqref{originalpb}. Then, by $(iii)$ of
Proposition~\ref{prop:exun}, the assumptions of
Proposition~\ref{propo1} are fulfilled and the well-posedness of
\eqref{modifpb} follows in a straightforward way.

In order to prove the convergence results $(i)$ and $(ii)$, we also need
some integrability properties of $u_{xx}$. More precisely, we will
use that
\begin{equation}\label{uxx}
u_{xx}(\cdot, t)\in L^1(\bR) \quad \mbox{ and }\quad \|u_{xx}(\cdot,
t)\|_{L^1(\bR)}\leq C, \quad \mbox{for all } t>1.
\end{equation}
The estimates in \eqref{uxx} can be proved by considering the new
variable $q=u_x$, which satisfies the initial value problem
\eqref{eq_ux}. By the regularization properties of the equation in
\eqref{eq_ux}, see for instance \cite{Henry}, it is possible to show
a bound of the type $\|q_x(\cdot, t+1)\|_{L^1(\bR)}\leq
C\|q(\cdot,t)\|_{L^1(\bR)}$. This implies \eqref{uxx}, since by $(ii)$
of Proposition~\ref{prop:exun}, $\|q(\cdot,t)\|_1$ is uniformly
bounded in $t$.

$(i)$ By using that the inner product in $L^2(\bR)$ is invariant under
translations in the space variable and $(iii)$ of
Proposition~\ref{prop:exun}, we can estimate
$$
|\lambda_v(t)-\cs | = \frac{1}{\|u_x\|_2^2} \left| \langle f(v), v_x
\rangle + \cs \|v_x\|_2^2 \right| \le C \left| \langle f(v), v_x
\rangle + \cs \|v_x\|_2^2 \right|,
$$
for some $C>0$. Then, for $w$ defined in \eqref{defw}, $\Phi=\Phi(x-x_0)$,
with $x_0$ in \eqref{fifemcl_expconv}, and formula
\eqref{cquotient}, it follows
\begin{eqnarray*}
\langle f(v), v_x \rangle + \cs \|v_x\|_2^2 &=& \langle f(w), w_x
\rangle + \cs \|w_x\|_2^2 \\
&=& \langle f(w), w_x \rangle + \cs \|w_x\|_2^2 - \langle f(\Phi),\Phi'
\rangle - \cs \|\Phi'\|_2^2 \\
&=& \langle f(w)-f(\Phi), w_x \rangle +   \langle f(\Phi), w_x - \Phi'
\rangle + \cs \langle w_x - \Phi', w_x + \Phi' \rangle.
\end{eqnarray*}
On one hand, using that $f'$ is continuous, $w,\Phi \in(0,1)$,
and by $(ii)$ of Proposition~\ref{prop:exun} it is
$\|w_x(\cdot,t)\|_1 = \|u_x(\cdot,t)\|_1 \le C $, for all $t>0$, we
can estimate
$$
|\langle f(w)-f(\Phi), w_x \rangle| \le C \|w-\Phi\|_{\infty} \|w_x\|_1
\le \tilde{C} e^{-\omega t}.
$$
for $\omega$ in \eqref{fifemcl_expconv}.

 Moreover, with  estimates in Lemma~\ref{lema41FiMc},  Lemma~\ref{lema43FiMc}  and the fact that $\Phi$ approaches its limits at
$\pm\infty$ exponentially fast, there exists $\sigma>0$ such that
\begin{eqnarray*}
| \langle f(\Phi), w_x - \Phi' \rangle | &=& \lim _{L\to \infty} \left|
\int_{-L}^L f(\Phi)(w_x - \Phi')\, dx \right| \\
&\le& \limsup_{L\to \infty} \left| f(\Phi)(w(\cdot,t)-\Phi)|_{x=-L}^L + \int_{-L}^L f'(\Phi)\Phi'(w - \Phi) \right| \\
&\le& \lim_{L\to \infty} | C(e^{-\sigma L}+ e^{-\mu t}) | +
\left| \langle f'(\Phi)\Phi',w - \Phi \rangle \right| \\
&\le& Ce^{-\mu t} + \tilde{C}\|w-\Phi\|_{\infty} \le Ce^{-\mu t} +
\tilde{C}e^{-\omega t}.
\end{eqnarray*}

In a similar way,
\begin{eqnarray}\label{wxU'}
| \langle  w_x - \Phi', w_x + \Phi' \rangle | &\le& \lim_{L\to \infty} |
C(e^{-\sigma L}+ e^{-\mu t}) | +
\left| \langle w-\Phi, w_{xx} + \Phi'' \rangle \right| \nonumber \\
&\le& Ce^{-\mu t} + \|w-\Phi\|_{\infty} \|w_{xx} + \Phi''\|_1 \nonumber \\
&\le& Ce^{-\mu t} + \tilde{C}e^{-\omega t},
\end{eqnarray}
where we used that $w_{xx}$ and $\Phi''$ are in $L^1(\bR)$, that
$\|w_{xx}\|_1 = \|u_{xx}\|_1$ and \eqref{uxx}.

This implies that
$$
|\lambda_v(t)-\cs | \leq Ce^{-\mu t} + \tilde{C}e^{-\omega t}
$$

Observe now that from Remark \ref{nota_exprate} we can choose both $\mu, \omega<\omega_0$ but arbitrarily close to $\omega_0$ (defined in \eqref{exprate}).
Hence if $\omega^*<\omega_0$, we may choose $\omega^*<\mu,\omega<\omega_0$ which implies
$Ce^{-\mu t} + \tilde{C}e^{-\omega t}\leq C_1e^{-\omega^*t}$ and this  shows \eqref{conver_speed}.

$(ii)$ The estimate \eqref{conver_speed} implies the convergence of the
integral
$$
\int_0^{\infty} (\lambda_v(s)-\cs ) \,ds = A, \quad \mbox{ for some }
A\in \bR,
$$
since it is absolutely convergent. Moreover,
$$
|\gamma_v(t) - \cs t - A| = \left|\int_t^{\infty} (\lambda_v(s) - \cs )
\,ds \right| \le C_1 \int_t^{\infty} e^{-\omega^* s} \,ds =
\tilde{C}e^{-\omega^* t}.
$$
Then, setting $x^*=x_0-A$, for $x_0$ in \eqref{fifemcl_expconv}, and
applying Lemma~\ref{lema43FiMc}, Theorem~\ref{th:FiMc}, and
\eqref{cotamu} it follows
\begin{eqnarray*}
&& |v(x,t) - \Phi(x-x^*)| = |w(x+A + \gamma_v(t)-\cs t-A,t) - \Phi(x+A-x_0)|
\\
& & \qquad \le \,  |w(x+A + \gamma_v(t)-\cs t-A,t)-w(x+A,t)| +
|w(x+A,t)-
\Phi(x+A-x_0)| \\
& & \qquad  \le \, \sup_{x\in \bR} (|w_x(x,t)| )\,
|\gamma_v(t)-\cs t-A| + |w(x+A,t)- \Phi(x+A-x_0)| \\
& & \qquad \le \,  C_2 e^{-\omega^* t}.
\end{eqnarray*}
which concludes the proof of the theorem. \end{proof}

\par\medskip

The analogous result holds also for problem \eqref{modifpb2}.
\begin{theorem}\label{th:v2}
Under the hypotheses of Theorem~\ref{th:FiMc} and assuming further
that $u_0 \in \dot{W}^{1,2}(\bR)$, the augmented problem
\eqref{modifpb2} is well-posed and its solution $v$ is given by
\eqref{v2}-\eqref{gamma2}.

Let any $\bar\omega<\omega_0$, $\omega_0$ being defined as in \eqref{exprate}.
 Then, there also exist $x^*\in \bR$ and
positive constants $C_1$, $C_2$, such that
\par\medskip\noindent i) for $\cs $ the propagation speed in \eqref{ode}
and $\lambda_v(t) = \gamma_v'(t)$  for $\gamma_v(t)$ in
\eqref{gamma2} it holds
\begin{equation}\label{conver_speed2}
|\lambda_v(t)-\cs | \le C_1e^{-\bar \omega t}, \qquad  t>0.
\end{equation}

\noindent ii) For $\Phi$ the unique (except for translations) solution
to \eqref{ode}, we can estimate
\begin{equation}\label{conver_v2}
|v(x,t)-\Phi(x-x^*)| < C_2 e^{-\bar\omega t}, \qquad x\in \bR, \quad t>0.
\end{equation}
\end{theorem}
The proof of Theorem~\ref{th:v2} is a simplified version of the one
given for Theorem~\ref{th:v}. Only estimate \eqref{wxU'} is needed
in order to prove $(i)$.

\section{Stationary solutions of the non local problem in a bounded interval}
\label{existence-bounded-interval}

In this section we show the existence and uniqueness of a stationary solution of
the non local problem  \eqref{modifpb_interval-intro} and analyze its relation
with the travelling wave solution found in the previous section. We have
divided the section in three subsections. In the first one we analyze the local
problem (see \eqref{pb_interval}) in a bounded domain. The results from this
subsection are used to obtain existence and uniqueness  of stationary solutions
for the non local problem \eqref{modifpb_interval-intro}. Finally, we show that
these stationary solutions approach the travelling wave solution.

\subsection{The local problem in a bounded interval}\label{subsec-local-bounded}
In this subsection we study the existence, uniqueness and properties of stationary solutions of
problem \eqref{pbw} when the domain is truncated to a finite interval. We impose non homogeneous Dirichlet
boundary conditions, i.e., we consider the evolution problem
\begin{equation}\label{pb_interval}
\left\{
\begin{array}{l}
u_t = u_{xx} + cu_x + f(u), \qquad x \in (a,b) , \quad t>0, \\
u(a,t) = 0 ; \ u(b,t) = 1, \qquad t>0, \\
u(x,0) = u_0(x), \qquad x \in [a,b],
\end{array}
\right.
\end{equation}
with $0\le u_0 \le 1$ and analyze its set of equilibria lying between 0 and 1.
Observe that these equilibria are
solutions of the boundary value problem,
\begin{equation}\label{ode_interval}
\left\{\begin{array}{l}
U''(\xi)+cU'(\xi)+f(U(\xi)) = 0, \qquad a < \xi < b,\\
U(a) = 0; \quad U(b)=1, \qquad 0\le U\le 1,
\end{array}\right.
\end{equation}
which can be interpreted as the first coordinate of the solution $(U,V)$ of the
$2\times 2$ ODE
\begin{equation}\label{my-ode-01}
\left\{
\begin{array}{l}
U'=V,     \\
V'=-cV-f(U).
\end{array}
\right.
\end{equation}
satisfying  $U(a)=0$, $U(b)=1$, $0\leq U(\xi)\leq 1$ for all $\xi\in
(a,b)$, where we denote by $'=\frac{d}{d\xi}$.

We will consider now several important properties of the solutions of the
ODE \eqref{my-ode-01}. In the first place, we notice that two different
solutions of \eqref{my-ode-01} do not intersect, due to the uniqueness
of solutions of any initial value problem for this ODE. In what follows we will
make use of this property without further notice. Observe also that since this
ODE is independent of the ``time'' variable $\xi$, we have that if $(U(\xi),
V(\xi))$ is a solution, then $(U(\xi+a), V(\xi+a))$ is also a solution for any
$a\in \bR$. Hence, we may consider solutions of \eqref{my-ode-01} defined for
$\xi\in [0, r]$. We will concentrate mainly on solutions starting at a point of
the form $(U,V)=(0,\theta)$ and will analyze the dependence and properties of
this solutions with respect to $\theta$, the constant $c$ and so on. Since we
are assuming that the nonlinearity $f$ satisfies conditions \eqref{hipof_intro},
then system \eqref{my-ode-01} has three distinguished solutions which are given
by the stationary states $(0,0)$, $(1,0)$ and $(\alpha, 0)$.

There is another special kind of solutions of \eqref{my-ode-01} which will play
an important role in the analysis below and that we will denote as ``simple
solution''.

\begin{definition} A ``simple solution''  is a solution
 $(U(\xi),V(\xi))$, $\xi\in [0,r]$ of  \eqref{my-ode-01}, joining $(0,\theta)$
and $(1,\Upsilon)$ with $\theta,\Upsilon>0$  (that is,  $(U(0),V(0))=(0,
\theta)$ and  $(U(r), V(r))=(0, \Upsilon)$) and with  $0\leq U(\xi)\leq 1$,
$\xi\in [0,r]$ (see Figure \ref{simplesolution}).
\end{definition}

We start proving some results about these ``simple solutions''.

\begin{lemma}\label{simple-solution-lemma}
If $(U,V)$ is a ``simple solution'' then there exists $v_0>0$ such that $V\geq v_0$. Moreover, the curve
$(U(\xi), V(\xi))$, $\xi\in [0,r]$,  can be represented as a function $V=P(U)$ for $0\leq U\leq 1$ and $P(U)>0$ for all $0\leq U\leq 1$.
\end{lemma}

\begin{proof}
Observe that by definition it is $V(0)=\theta >0$ and $V(r)=\Upsilon>0$.
%since $U(0)=0$ and $U(\xi)\geq 0$ for  $0\leq \xi\leq r$, we have
%$U'(0)=V(0)\geq 0$.  Moreover $V(0)>0$ since if $V(0)=0$ then
%$(U(0),V(0))=(0,0)$, which is a stationary solution of \eqref{my-ode-01} and by
%uniqueness of solutions we will have $(U(\xi),V(\xi))=(0,0)$ for all $\xi\in
%(0,r)$ and therefore $U(r)\ne 1$.

\begin{figure}
  \centering
   \includegraphics[width=0.35\textwidth]{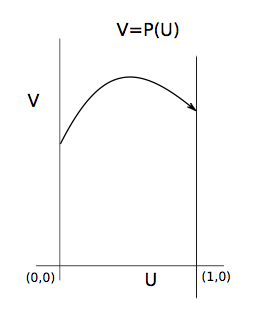} \qquad\qquad
   \includegraphics[width=0.35\textwidth]{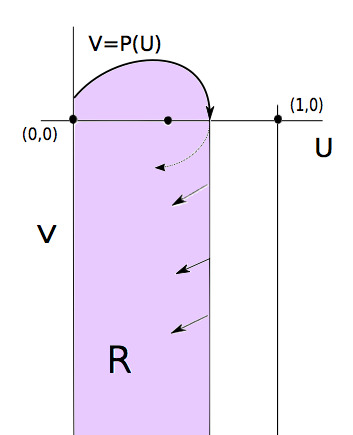}
  \caption{(left) A ``simple solution'' (right) Region R}
  \label{simplesolution}
\end{figure}

If there exists a point $\xi^-\in (0,r)$ with $V(\xi^-)\leq 0$, then let $\xi_0$
be the first point such that $V(\xi)>0$ for $0\leq \xi<\xi_0$ and $V(\xi_0)=0$.
By definition $\xi_0<r$ and $V'(\xi_0)\leq 0$. Again we cannot have
$V'(\xi_0)=0$ since if this is the case, then
$V'(\xi_0)=-cV(\xi_0)-f(U(\xi_0))=-f(U(\xi_0))=0$ and this implies that
$(U(\xi_0),V(\xi_0))$ is an equilibrium point of \eqref{my-ode-01} which is not
the case.  Therefore $V'(\xi_0)<0$ (hence $f(U(\xi_0))>0$) and in particular for
$\xi>\xi_0$ but $\xi$ near $\xi_0$ we have $U(\xi)<U(\xi_0)$ and $V(\xi)<0$.
Moreover, for $0\leq \xi \leq \xi_0$ we have $0\leq U(\xi)\leq U(\xi_0)<1$  (if
$U(\xi_0)=1$, since $V(\xi_0)=0$ we will have again another equilibrium point of
\eqref{my-ode-01}). Hence, the curve $(U(\xi),V(\xi))$ for $0\leq \xi\leq
\xi_0$ joins the point $(0, \theta)$ with $(U(\xi_0),0)$ with $U(\xi)$
increasing and $V(\xi)>0$, that is this curve can be expressed as $V=P(U)$ for
$0\leq U\leq U(\xi_0)$ for some function $P$.

Consider now the region $R=\{ (U,V):  0\leq U\leq U(\xi_0),  \quad  -\infty
<V<P(U)\} $. Using the direction field of \eqref{my-ode-01}, we have that any
initial condition starting in the set $R$ either remains for all times in $R$ or
it leaves the region $R$ through the boundary $\{ (0,V): -\infty <V<0\}$. Since
our solution satisfies $U(b)=1$, necessarily has to leave the region $R$ but
then there will exist a point $\xi_1\in (\xi_0,b)$ with $U(\xi_1)=0$ and
$V(\xi_1)<0$, which implies that $U(\xi_1+\eps)<0$, for $\epsilon>0$ small
enough. This is a contradiction with the fact that $0\leq U(\xi)\leq 1$. Hence
$V(\xi)>0$ for all $\xi\in [0,r]$. The existence of the function $P$ follows
easily now from the fact that $U'(\xi)=V(\xi)>0$ for all $\xi\in [0,r]$.
\end{proof}
\par\medskip

\begin{lemma}\label{basic-ode-lemma}
  The following properties hold:

(i)  If  $V=P(U)$ is a simple solution then the time it takes to go from the point
$(U(\xi_0), V(\xi_0))$ to $(U(\xi_1), V(\xi_1))$, that is $\xi_1-\xi_0$,  is
given by
$$\xi_1-\xi_0= \int_{U(\xi_0)}^{U(\xi_1)}\frac{dU}{P(U)}.
$$
In particular the time it takes to go from $(0, \theta)$ to $(1, \Upsilon)$ is
given by
$$
r= \int_{0}^{1}\frac{dU}{P(U)}.
$$

(ii) If we have two ``simple solutions'' $V=P_0(U)$ and $V=P_1(U)$ satisfying
$P_0(U^*)<P_1(U^*)$  for some $U^*\in [0,1]$, then $P_0(U)<P_1(U)$ for all
$0\leq U\leq 1$.

(iii) If we consider \eqref{my-ode-01} for two constants $c_0<c_1$
and simple solutions $V=P_i(U)$ for the system with $c_i$, $i=0,1$,
respectively, then if $P_0(U^*)=P_1(U^*)$ for some $0<U^*<1$, it is
$P'_0(U^*)>P'_1(U^*)$. In particular the orbits can only cross once (see Figure
\ref{2-simple-solutions}).

(iv) If $\theta>|c|+|f|_{\infty}+1$ where $|f|_\infty= \max\{ |f(s)|, 0\leq
s\leq 1\}$ then the solution starting at $(0,\theta)$ is a simple solution
joining $(0,Ê\theta)$ with $(1, \Upsilon)$ for some $\Upsilon\geq \theta
-|c|-|f|_\infty$. Moreover this simple solution $V=P(U)$ satisfies $P(U)\geq
\theta -(|c|+|f|_\infty)U$ for $0\leq U\leq 1$. In particular
$$
\int_0^1\frac{dU}{P(U)}\to 0, \quad \hbox{ as } \theta\to +\infty.
$$

(v) There exists $\theta_0\geq 0$ such that the solution starting at $(0,
\theta)$ for each $\theta>\theta_0$ is a simple solution  $V=P_\theta(U)$
with
$$
\int_0^1\frac{du}{P_\theta(U)}\to +\infty, \quad \hbox{ as } \theta\to
\theta_0^+.
$$

(vi) For any $\theta_0>0$ there exists a $c_0=c_0(\theta_0)>0$ such that for all
$c>c_0$ any ``simple solution'' starting at $(0, \theta)$ with $\theta>\theta_0$
ends at $(1,\Upsilon)$ with $\Upsilon<\theta$.  Similarly, there exists a
$c_1=c_1(\theta_0)<0$ such that for all $c<c_1$ any ``simple solution'' starting
at $(0, \theta)$ with $\theta>\theta_0$ ends at $(1,\Upsilon)$ with
$\Upsilon>\theta$

\end{lemma}
\begin{proof} $(i)$ It follows from standard results. Notice that the first
equation of  \eqref{my-ode-01} reads $\frac{dU}{d\xi}=V=P(U)$ which implies
$d\xi=\frac{dU}{P(U)}$ from where we obtain the result via integration.

\par $(ii)$ This follows from the uniqueness of solutions of the ODE
\eqref{my-ode-01}.
\par $(iii)$ Note that $\frac{dP(U)}{dU}=-c+f(U)/P(U)$. Hence, if
$P_0(U^*)=P_1(U^*)$, then $P'_0(U^*)=-c_0+f(U^*)/P(U^*)>
-c_1+f(U^*)/P(U^*)=P'_1(U^*)$

\begin{figure}
  \centering
  \includegraphics[width=0.45\textwidth]{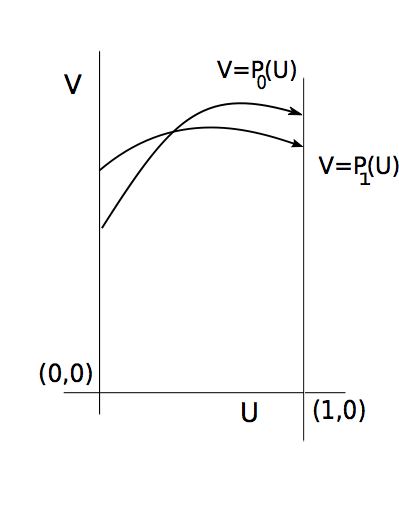}
      \caption{Two simple solutions $V=P_0(U)$, $V=P_1(U)$ for $c_0<c_1$ which crossing.}
  \label{2-simple-solutions}
\end{figure}

\par $(iv)$ In the $(U,V)$ plane, the segment joining $(0,\theta)$ with $(1,
\Upsilon^*)$ is given by $V=(\Upsilon^*-\theta)U+\theta$, for $0\leq U\leq 1$.
Hence, if $\theta>|c|+|f|_{\infty}+1$ and $\Upsilon^*=
\theta-|c|-|f|_{\infty}$, the straight line is $V= - (|c|+|f|_{\infty})U
+\theta$ and notice that we have $V>1$ for all $(U,V)$ in the segment. But, in
this segment, the direction field of the ODE \eqref{my-ode-01} points ``toward
the right''  of the straight line, since the scalar product of its normal
vector, $(|c|+|f|_\infty,1)$,  with the direction field $(V, -cV+f(U))$  is
$$(|c|+|f|_\infty) V-cV+ f(U)=V(|c|-c)+|f|_\infty V+f(U)>0$$
since if $(U,V)$ lies in the straightline then $V> 1$, see Figure \ref{prueba-lemma}.

\begin{figure}
  \centering
  \includegraphics[width=0.60\textwidth]{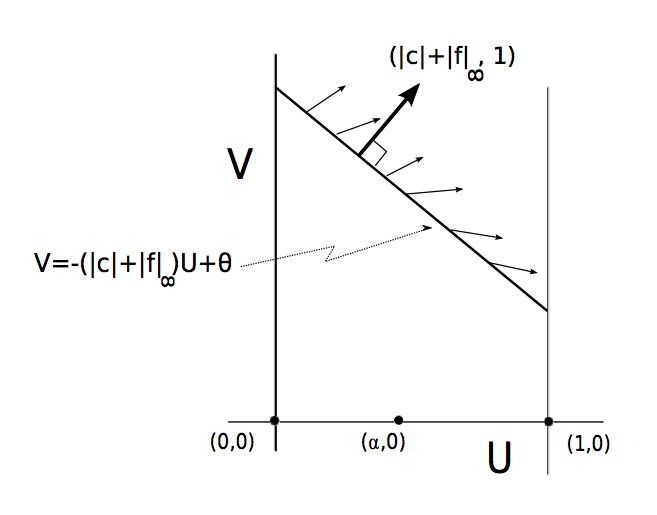}
      \caption{ The segment and the direction field.}
  \label{prueba-lemma}
\end{figure}

This implies that if we start at $(0, \theta)$ then the orbit of the ODE cannot
cross this line segment and therefore we will have
$P(U)>\theta-(|c|+|f|_\infty)U$ for all $0\leq U\leq 1$ and in particular it
will arrive at a point $(1,\Upsilon)$ with $\Upsilon>\Upsilon^*$.  The last
statement of this item is obvious from the bounds of $P(U)$ and taking
$\theta\to +\infty$.

\par $(v)$ It is clear from $(iv)$ that there exists $\theta_1>0$ such that the
orbit starting at $(0, \theta)$ for $\theta\geq \theta_1$ are all simple
solutions. Moreover, from Lemma \ref{simple-solution-lemma} and the continuous
dependence of the solutions of the ODE with respect to the initial conditions,
we have that if the solution starting at a point $(0, \theta^*)$ is a simple
solution, then there exists a small $\eps>0$ such that for each $\theta\in
(\theta^*-\eps, \theta^*+\eps)$ the solution is also a simple solution. Let us
define
$$\theta_0=\inf\{ \theta:  \hbox{ the orbit starting at }(0,\theta^+)  \forall
\theta^+>\theta \hbox{ is a simple solution }\}$$
Then, by definition  the orbit starting at $(0, \theta_0)$ is not a simple
solution. But, necessarily this orbit satisfies $0\leq U\leq 1$ and, if it
reaches a point of the form $(1, \Upsilon)$ in  finite time, by Lemma
\ref{simple-solution-lemma} it will also be a simple solution. Therefore by the
continuity of the flow, we necessarily have either $\theta_0=0$ or
$\Upsilon=0$, which leads to
$$
\int_0^1\frac{du}{P_\theta(U)}\to +\infty, \quad \hbox{ as } \theta\to
\theta_0^+.
$$

\par $(vi)$ Observe that if $\theta_0>0$ and $\theta\geq \theta_0$ then, on the
horizontal line $V=\theta$, the vector field of the ODE is $(\theta,
-c\theta+f(U))$. Therefore, if we choose $c_0=|f|_{\infty}/\theta_0>0$, then for
any $c>c_0$ the vector field ``points downwards'' in this line.
This implies that if we start with a solution at $(0,\theta)$ and the solution
travels to a point $(1, \Upsilon)$ then the whole orbit lies below the straight
line $V=\theta$. Hence, $\Upsilon<\theta$. Similarly for the other case.
\end{proof}

\par\medskip  We also have another two distinguised solutions with $0\leq U\leq
1$ and $V\geq 0$, which are usually denoted as the ``unstable orbit'' from the
equilibrium $(0,0)$ and the ``stable orbit'' to the equilibrium $(1,0)$.  These
are special solutions since they are defined either in an interval of the form
$\xi\in (-\infty, r_0)$  or $\xi \in (r_1, +\infty)$.

\begin{lemma}\label{unstable-manifolds}
There exists a unique value $c^*\in \bR$ such that the following holds

\par (i) If $c<c^*$ the unique ``unstable orbit'' emanating from $(0,0)$ with
$0\leq U\leq 1$ ends at a point $(U_0,0)$ for some value $0<U_0<1$ and the
unique ``stable orbit'' converging to $(1,0)$ with $0\leq U\leq 1$ starts at a
point of the form $(0,\theta_0)$, with $\theta_0>0$.

\par (ii) If $c>c^*$ the unique ``unstable orbit'' emanating from $(0,0)$ with
$0\leq U\leq 1$ ends at a point $(1,\Upsilon_1)$ for some value $\Upsilon_1>0$
and the unique ``stable orbit'' converging to $(1,0)$ with $0\leq U\leq 1$
starts at a point of the form $(U_0, 0)$.

\par (iii) If $c=c^*$, both orbits coincide and we have a unique orbit with
$0\leq U\leq 1$ and $V\geq 0$ which comes out from $(0,0)$ (as $\xi\to -\infty$)
and comes in to $(1,0)$  (as $\xi\to +\infty$).

\end{lemma}

\begin{figure}
  \centering
  \includegraphics[width=0.30\textwidth]{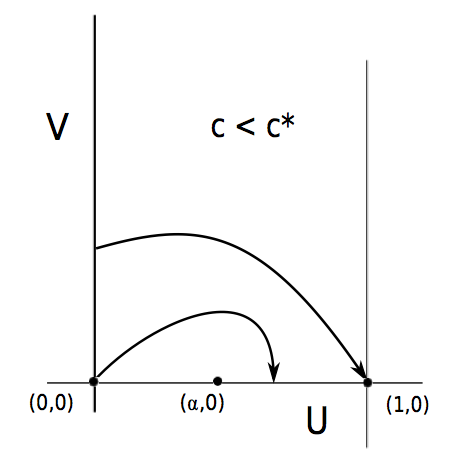}
  \includegraphics[width=0.30\textwidth]{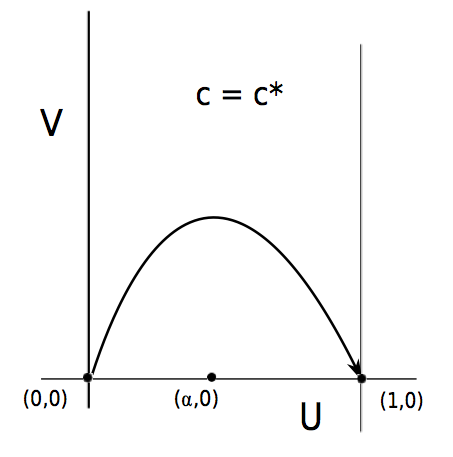}
  \includegraphics[width=0.30\textwidth]{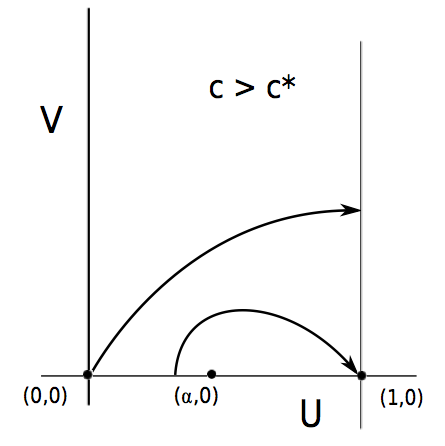}
  \caption{Phase planes for Lemma~\ref{unstable-manifolds}.}
  \label{Dibujo-unstable-manifolds}
\end{figure}

\begin{proof} The proof is a simple exercise of phase plane techniques and we
leave if for the reader. See Figure \ref{Dibujo-unstable-manifolds} and the
book \cite{Henry}  for details. \end{proof}

\begin{nota}
Observe that the value $c^*$ from the previous Lemma is the speed of
propragation of the travelling wave of problem \eqref{originalpb}.
\end{nota}

\par\medskip
With the results above we can prove the first result on existence of solution of
\eqref{ode_interval} with $c$ fixed.

\begin{proposition}\label{lema:equi_cab}
For every $c\in \bR$, there exists a unique solution $\vfi_c$ of
\eqref{ode_interval}. Furthermore, $\vfi_c$ is strictly monotone, that is,
$\vfi'(x)>0$ for all $x\in [a,b]$.
\end{proposition}
\begin{proof}
Observe that $\vfi_c$ is a solution of
\eqref{ode_interval} if and only if $(U,V)=(\vfi_c,\vfi_c')$ is a simple
solution as defined above. But if we consider the value $\theta_0$ from Lemma
\ref{basic-ode-lemma} $(v)$ and the simple solution $V=P_\theta(U)$
which starts at $(0,\theta)$ then the time it takes to go from $(0, P(0))$ to
$(1, P(1))$ is given by
$$
r(\theta)=\int_0^1\frac{dU}{P_\theta(U)},
$$
which by Lemma \ref{basic-ode-lemma} $(iv)$ and $(v)$ satisfy that there exists
at least one $\theta$ for which $r(\theta)=b-a$.
From  Lemma \ref{basic-ode-lemma} $(ii)$ we have the uniqueness.
\hfill\end{proof}

\par\medskip

As we show below, the monotonicity of the equilibrium, yields its
asymptotic stability.

\begin{proposition}\label{propo:stabinter}
Let $c\in \bR$ and the equilibrium solution $\vfi_c$ of
\eqref{pb_interval}. Then, there exists $K,a > 0$ such that for
$\|u_0-\vfi_c\|_{\infty}$ small enough it holds
\begin{equation}\label{asymstab_inter}
\|u(\cdot,t)- \vfi_c\|_{\infty} \le Ke^{-a t},
\end{equation}
for $u$ the solution of \eqref{pb_interval}.
\end{proposition}
\begin{proof}
%{\color{red}{\em He dejado de momento la prueba que ten\'iamos, pero parece que es un resultado conocido e inmediato
%seg\'un \cite{BeNiVa}}.}
We consider the linearization about $\vfi_c$ of \eqref{modifpb_interval}
\begin{equation}\label{lin_inter}
\left\{\begin{array}{l}
u_t = u_{xx} + c u_x + f'(\vfi_c)u, \\
u(a,t) = u(b,t) = 0.
\end{array}\right.
\end{equation}
We need information about the spectrum of the operator
\begin{equation}\label{opelin}
L_0 q := q'' + c q' + f'(\vfi_c)q,
\end{equation}
in $D(L_0) = \{q\in \cC^2[a,b] : q(a) = q(b) = 0 \}$. Deriving the equation in
\eqref{ode_interval} we obtain that $\phi=\vfi'_c
> 0$ satisfies the boundary value problem
\begin{equation}\label{robantirob}
\left\{\begin{array}{l}
\phi'' + c \phi'_c + f'(\vfi_c) \phi = 0, \\
\phi'(a) + c \phi(a) = 0, \qquad \phi'(b) + c \phi(b) = 0,
\end{array}\right.
\end{equation}
where the boundary conditions are obtained by evaluating
\eqref{ode_interval} at $\xi = a$ and $\xi=b$ and using that
$f(0)=f(1)=0$. The change of variables $z(\xi) = e^{\frac c2 \xi}q(\xi)$
in \eqref{robantirob} leads to the self-adjoint problem in
$L^2(a,b)$
\begin{equation}\label{selfad_robantirob}
\left\{\begin{array}{l} \displaystyle -z'' + \left( \frac{c^2}4 -
f'(\vfi_c)\right) z = 0,
\\[1em]
\displaystyle z'(a) + \frac c2 z(a) = 0, \qquad z'(b) + \frac c2
z(b) = 0,
\end{array}\right.
\end{equation}
For \eqref{selfad_robantirob}, the positive mapping $\xi\to e^{\frac
c2\xi} \phi(\xi) > 0$ is an eigenfunction associated to the eigenvalue
0. Then, by Krein-Rutman's Theorem, 0 is the smallest eigenvalue
of the operator
\begin{equation}\label{opself}
Bq := -z'' + \left( \frac{c^2}4 - f'(\vfi_c)\right) z,
\end{equation}
for the boundary conditions in \eqref{selfad_robantirob}. Then,
\begin{eqnarray*}
0 &=& \min_{H^1(a,b)} \frac{\int_a^b
\left[(z')^2 + \left( \frac {c^4}4 - f'(\vfi_c) \right) z^2
\right]\, d\xi + \frac {|c|}2
\left(z^2(a)-z^2(b) \right)}{\int_a^b z^2 \,d\xi } \\
&<& \min_{H^1_0(a,b)} \frac{\int_a^b  \left[(z')^2 + \left( \frac
{c^4}4 - f'(\vfi_c) \right) z^2 \right]\, d\xi }{\int_a^b z^2
\,d\xi},
\end{eqnarray*}
so that the smallest eigenvalue of $B$ with homogeneous Dirichlet boundary
condition is strictly positive. This implies $\sigma(L_0) \subset \{z\in \bC:
\real z < 0\}$, yielding the asymptotic stability of $\vfi_c$. \hfill
\end{proof}

\begin{corollary}\label{negative-spectrum}
With the notations above, we have $\sigma(L_0)\subset \{z\in \bC:
\real z < 0\}$.
\end{corollary}

\subsection{The nonlocal problem in a bounded interval}\label{subsec-nonlocal-bounded}
In this section we study the existence of equilibria of the non local equations
\eqref{modifpb} and \eqref{modifpb2} restricted to a finite interval with
nonhomogeneus Dirichlet boundary conditions. It is immediate to see that the resulting
Initial and Boundary Value Problems (IBVP) for both equations are the same,
namely:
\begin{equation}\label{modifpb_interval}
\left\{
\begin{array}{l}
\displaystyle  u_t = u_{xx} +\la(u_x) u_x + f(u),
\qquad x \in (a,b) , \quad t>0, \\[.5em]
\displaystyle \la(z(\cdot)) = - \frac{F(1)}{\|z(\cdot)\|_{L^2(a,b)}^2}\in \bR, \\[1em]
u(a,t) = 0 ; \ u(b,t) = 1, \qquad t>0, \\
u(x,0) = u_0(x), \qquad x \in [a,b],
\end{array}
\right.
\end{equation}

We have the following

\begin{proposition}\label{well-posednessÁ}
Problem \eqref{modifpb_interval} is locally well posed, in the sense that for
any initial data $u_0\in H^1(a,b)$ with
$u_0(a)=0$, $u_0(b)=1$ there exists a $T=T(u_0)$ and a unique classical solution
$u(x,t)$ defined for the time interval $[0,T(u_0))$.

Moreover, if $0\leq u_0\leq 1$, then the solution $u(x,t)$ is globally defined
and it also satisfies $0\leq u(x,t)\leq 1$.
\end{proposition}

\begin{proof}
Observe that problem \eqref{modifpb_interval} can be rewritten as a more
standard problem with homogeneous Dirichlet boundary conditions with the change
of variables: $u(x)=w(x)+h(x)$, where the function
$h(x)=\frac{x-a}{b-a}$. Notice that if $u(a)=0$, $u(b)=1$, then
$w(a)=w(b)=0$ and problem \eqref{modifpb_interval} takes the form

\begin{equation}\label{modifpb_interval-h10}
\left\{
\begin{array}{l}
\displaystyle  w_t = w_{xx} +F(w)
\qquad x \in (a,b) , \quad t>0, \\
%\displaystyle \la(u_x) = - \frac{F(1)}{\|u_x(\cdot,t)\|_2^2}, \\[1em]
w(a,t)=w(b,t) = 0,\qquad t>0, \\
w(x,0) = w_0(x)\equiv u_0(x)+h(x), \qquad x \in [a,b],
\end{array}
\right.
\end{equation}
where $F:H^1_0(a,b)\to L^2(a,b)$ is the map defined by
$F(w)=\la(w_x+\frac{1}{b-a}) (w_x +\frac{1}{b-a})+ f(w+h(\cdot))$.
We can easily see that this map is well defined, since if $w\in H^1_0(a,b)$
evaluating in the following expression, we have
$$\left\| w_x+\frac{1}{b-a}
\right\|_{L^2(a,b)}^2=\|w_x\|^2_{L^2(a,b)}+\frac{1}{b-a} \geq\frac{1}{b-a}$$
and therefore the denominator in the function $\lambda$ is bounded away from 0
and $\lambda(w_x+\frac{1}{b-a})$ is well defined. Moreover, following standard
arguments, the map $F$ is Lipschitz on bounded sets of $H^1_0(a,b)$ which from
standard techniques, see \cite{Henry}, we obtain that problem
\eqref{modifpb_interval-h10} is locally well posed in $H^1_0(a,b)$ and therefore
problem \eqref{modifpb_interval} is locally well posed for any initial condition
$u_0 \in H^1(a,b)$ satisfying $u_0(a)=0$, $u_0(b)=1$. Standard regularity
results applied to \eqref{modifpb_interval}  (notice that $\lambda$ is
independent of $x$) show that the solution is classical for $t>0$

If we consider now that $0\leq u_0\leq 1$, then we may argue by comparison
with the constants to show that as long as the solution exists, it will also
satisfy $0\leq u\leq 1$.  Let us argue by contradiction. Assume the solution is
negative at some time $T>0$. Then, for $\eps>0$ small enough there exists a
$0<t_\eps\leq T$ such that $u(x,t)>-\eps$ for all $x\in [a,b]$, $0\leq t<t_\eps$
and
there exists $x_\eps\in (a,b)$ such that $u(x_\eps,t_\eps)=-\eps$. This implies
that $u_t(x_\eps,t_\eps)\leq 0$, $u_{xx}(x_\eps,t_\eps)\geq 0$,
$u_x(x_\eps,t_\eps)=0$ and $f(u(x_\eps,t_\eps))=f(-\eps)>0$. Which is a
contradiction. In a similar way we may proceed with the upper bound $u\leq 1$.
This shows that, as long as the solution exists we have $0\leq u(x,t)\leq 1$.
With standard continuity arguments we show that the solution is globally defined
and satisfy the bounds.
This concludes the proof of the proposition.
\end{proof}

\begin{nota}
Notice that although we have been able to obtain a comparison result with the
constants 0 and 1, we do not have a general comparison argument for equation
\eqref{modifpb_interval}. That is, if the initial conditions are ordered
$u_0\leq v_0$ we cannot conclude that the solutions satisfy the same ordering
for positive times. The lack of these comparison arguments for this equation is
very much related to the lack of maximum principles for the associated linear non local
operators.  This lack is  a
serious drawback, specially when analyzing the stability properties of the equilibrium of
this equation, see Remark \ref{Remark-no-maximum} below.
\end{nota}

\begin{figure}
  \centering
  \includegraphics[width=0.35\textwidth]{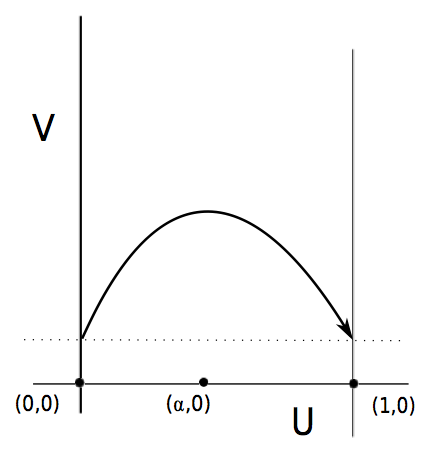}\qquad
  \includegraphics[width=0.5\textwidth]{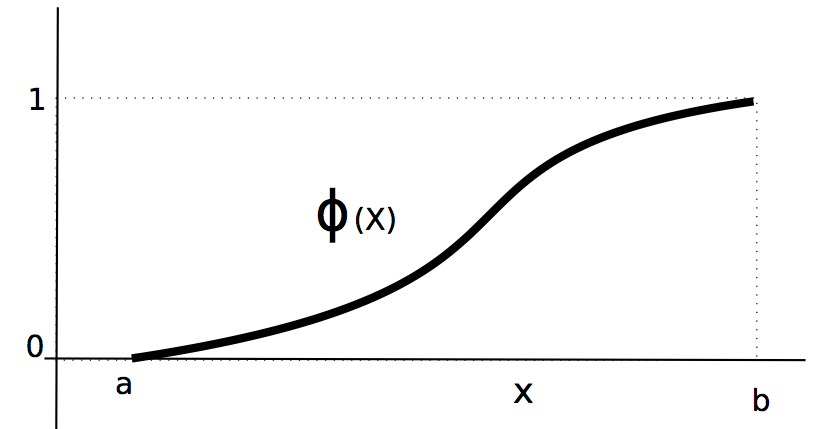}
  \caption{The equilibrium in a bounded domain: (Left) In the phase space $(U,V)$; (Right) As a function of $x$.}
  \label{equilibrium}
\end{figure}

The following result provides a characterization of the stationary solutions of
\eqref{modifpb_interval}.

\begin{lemma}\label{lema:equi_modipb_ab}
A function $\vfi(\cdot)\in H^1(a,b)$ with $\vfi(a)=0$, $\vfi(b)=1$ and $0\leq
\vfi(x)\leq 1$ is an stationary solution of \eqref{modifpb_interval} if and only
if  $\vfi(\cdot)$ is a solution of \eqref{ode_interval} with $c=\la(\vfi')$ and
satisfies
\begin{equation}\label{extrabc}
\vfi'(a) = \vfi'(b).
\end{equation}
\end{lemma}
\begin{proof}
It is clear that any equilibrium solution $\vfi$ of \eqref{modifpb_interval}
satisfies \eqref{ode_interval} for $c=\la(\vfi_x)$. Multiplication by $\vfi'$ in
the ODE in \eqref{ode_interval} and integration along $(a,b)$ yields
$$
\frac 12 [(\vfi'(b))^2-(\vfi'(a))^2] + \la(\vfi_x)\|\vfi'\|_2^2 +
\langle f(\vfi), \vfi' \rangle = 0.
$$
Using that $\langle f(\vfi), \vfi' \rangle=F(1)$ and the definition of $\lambda$
in \eqref{modifpb_interval} we obtain that  $\vfi$ satisfies
\eqref{extrabc}. \hfill \end{proof}

We can proceed now to prove an existence and uniqueness result for stationary
solutions of \eqref{modifpb_interval}.

\begin{theorem}\label{th:nonloc_interval}
There exists one and only one stationary solution $\vfi$ of
\eqref{modifpb_interval} with $0\leq \vfi(x)\leq 1$. Moreover, this solution is
strictly monotone increasing in $x$.
\end{theorem}
\begin{proof}
From Lemma \ref{lema:equi_modipb_ab}, a stationary solution $\vfi$ of
\eqref{modifpb_interval} with $0\leq \vfi(x)\leq 1$ is the first coordinate of a
``simple solution'' of the ODE
\begin{equation}
\left\{
\begin{array}{l}
U'=V,     \\
V'=-cV-f(U)
\end{array}
\right.
\end{equation}
where $c=\lambda(\vfi')=- \frac{F(1)}{\|\vfi'\|_{L^2(a,b)}^2}\in \bR$.
Lemma \ref{simple-solution-lemma} proves that $\vfi'(x)>0$ for all
$x\in [a,b]$ and therefore $\vfi$ is strictly monotone increasing.

Uniqueness is obtained as follows. Assume that there exist two solutions
$\vfi_1$, $\vfi_2$ and denote by $c_1=\lambda(\vfi_1')$,
$c_2=\lambda(\vfi'_2)$. From the uniqueness of solutions  given in Proposition
\ref{lema:equi_cab}, we get that $c_1\neq c_2$. Then, if we denote by
$(U_1,P(U_1))$, $(U_2,P(U_2))$ the two simple solutions associated to $\vfi_1$
and $\vfi_2$ respectively, from Lemma \ref{basic-ode-lemma} $(iii)$ and from
$P_1(0)=P_1(1)$, $P_2(0)=P_2(1)$ we must have that either $P_1(U)>P_2(U)$ or
$P_1(U)<P_2(U)$ for all $0\leq U\leq 1$. In both cases we have
$$b-a=\int_0^1\frac{dU}{P_1(U)}\neq \int_0^1\frac{dU}{P_2(U)}=b-a$$
which is a contradiction. This shows uniqueness.

Existence is shown as follows. We know from Proposition \ref{lema:equi_cab}
that for every fixed $c_0$ and $r=b-a$ there exists a
unique solution $\vfi_{c_0}$ of \eqref{ode_interval}, which actually is given by
$\vfi_{c_0}=U_0$ where $(U_0,V_0)$ is a ``simple solution'' joining
$(0,\theta_0)$ with $(1,\Upsilon_0)$ for some $\theta_0,\Upsilon_0>0$.  If
it happens that $\theta_0=\Upsilon_0$, that is $\vfi'_{c_0}(a)=\vfi'_{c_0}(b)$,
then Lemma \ref{lema:equi_modipb_ab} shows that this function $\vfi_{c_0}$
is the stationary solution we are looking for.
If $\theta_0\ne \Upsilon_0$, let us assume that $\theta_0<\Upsilon_0$  (the
other case is treated similarly). For $c>c_0$ and the same $r=b-a$, again by
Proposition \ref{lema:equi_cab} we have the existence of a solution $\vfi_c$
which again is given by $\vfi_{c}=U$ where $(U,V)$ is a ``simple solution''
joining $(0, \theta)$ with $(1,\Upsilon)$. But since $r$ is the same for both
solutions and we have $r=\int_0^1\frac{dU}{P_0(U)}=\int_0^1\frac{dU}{P(U)}$ then
necessarily, both solutions must cross at least at some point and by
Lemma \ref{basic-ode-lemma} they can only cross at one point and it must be
satisfied $\theta>\theta_0$, $\Upsilon<\Upsilon_0$. Moreover, from Lemma
\ref{basic-ode-lemma} we can choose $c_1>c_0$ large enough  such that for this
value $c_1$ the unique simple solution joining a point of the form $(0,
\theta_1)$ with $(1, \Upsilon_1)$ in a time $r=b-a$ satisfies
$\theta_1>\Upsilon_1$.  By the continuous dependence of the solutions $\vfi_c$
with respect to the parameter $c$, we will have that there will exist a value
$c^*\in (c_0, c_1)$ such
that the unique solution $\vfi_{c^*}$ travelling for a time $r=b-a$ joins a
point of the form $(0,\theta^*)$ with $(0, \Upsilon^*)$ with
$\theta^*=\Upsilon^*>0$, that is $\vfi'(a)=\vfi'(b)$. This is the desired
solution.  \end{proof}

\subsection{Convergence of the stationary solutions to the travelling wave as
the length of the interval goes to $+\infty$}
\label{subsec-convergence-stationary}
In this section we will pass to the limit as the interval grows to cover the
whole line and we analyze how the solution encountered in Theorem \ref{th:nonloc_interval} behaves as
the length of the interval goes to infinity.  The first step is to prove the convergence of the wave speed to the
one of the travelling wave. More precisely,

\begin{lemma}\label{lema:conv_la}
Let $\la_{r}$ be the unique value given by Theorem~\ref{th:nonloc_interval}
for which an equilibrium of \eqref{modifpb_interval}-\eqref{extrabc} exists on
the interval $(a,b)$ with $r=b-a$. Then,
\begin{equation}\label{converlambda}
|\la_{r} - \la_{\infty}| \to 0, \qquad \mbox{as } \quad r\to +\infty.
\end{equation}
where $\la_{\infty}=c^*$ from Lemma \ref{unstable-manifolds}, that is,  the speed of propagation of the travelling
wave of equation \eqref{originalpb}.
\end{lemma}

\begin{proof} %{\em Plano de fases. Argumento fundamental: el orden inducido
%en el campo por los valores de $\la$, que impide que se crucen dos
%\'orbitas asociadas a dos valores de $\la$ distintos.}
Observe first that the value of $\lambda_r$ really depends only on $r=b-a$ and not
on $a$ or $b$.

Assume that the result is not true. Then, there is a sequence
$\{r_n\}_{n\in\bN}$ with $r_n \to
\infty$, as $n\to \infty$, and $\veps>0$ so that if we denote by
$\la_n := \la_{r_n}$, then either $\lambda_n>\lambda_\infty+\eps$ or $\lambda_n<\lambda_\infty-\eps$. So let us assume that  $\la_n > \la_{\infty}
+ \veps$, for all $n$, the other case is treated similarly.

Observe that from Lemma \ref{unstable-manifolds}, in the phase plane associated
to the equation \eqref{modifpb_interval} for $c=\la_{\infty} + \veps$ there is
an orbit $(\vfi_{*},\vfi'_{*})$ arriving at $(1,0)$ from
$(0,\Upsilon^*)$, for a certain $\Upsilon^*>0$. This orbit is also represented
as $V=P^*(U)$ for $0\leq U\leq 1$. By $(iii)$ of Lemma~\ref{basic-ode-lemma}
and \eqref{extrabc}, none of the orbits $(\vfi_n,\vfi'_n)$, which
are given by $V=P_n(U)$, can cross $V=P^*(U)$ and it has to be $P_n(0) =
\Upsilon_n > \Upsilon^*$, for all $n$. It follows then that
$P_n(U)>P^*(U)$ for all $0\leq U\leq 1$. Furthermore, the graph of the
function $V=P_n(U)$ is also above the straight line passing through
$(1,\Upsilon^*)$ with slope $(1,-\la_{\infty} + \veps)$. This comes from the
fact that for $\alpha$ in \eqref{hipof_intro} and $\alpha < u < 1$, it is
$f(u)>0$ and the field in the phase plane is proportional to $(1,-\la-f(u)/v)$
with $-\la-f(u)/v < -\la$. It follows that the orbit $V=P_n(U)$ has to arrive at
$(1,\Upsilon_n)$ from
above this line. But then, the time it takes to travel from $(0, \Upsilon_n)$ to
$(1, \Upsilon_n)$  remains bounded, i.e., by $(i)$ of
Lemma~\ref{basic-ode-lemma} it holds
$$
r_n= b_n-a_n = \int_0^1 \frac{du}{P_n(u)} \le M(\veps), \qquad \mbox{for all }
\quad n\in \bN.
$$
This is in contradiction with the fact that $r_n\to +\infty$. \hfill
\end{proof}

\begin{lemma}\label{lema:conv_orbit}
Let $\vfi_{r}$ be the equilibrium obtained in Theorem \ref{th:nonloc_interval}
in the interval $(0,r)$. Then the orbit $(\vfi_{r},\vfi'_{r})$ in the phase
plane converges to the orbit associated to the travelling wave on the whole
line,
$(\vfi_{\infty},\vfi'_{\infty})$ as $r\to \infty$.
\end{lemma}
\begin{proof} Observe that the orbit $(\vfi_r,\vfi'_r)$ is a simple solution and it is given
as $V=P_r(U)$, $0\leq U\leq 1$. Moreover, we know that the travelling wave
$(\vfi_{\infty},\vfi'_{\infty})$ is given as the function $V=P_\infty(U)$ for
$0\leq U\leq 1$. We will show that $P_r\to P_\infty$ as $r\to +\infty$.

\begin{figure}
  \centering
  \includegraphics[width=0.60\textwidth]{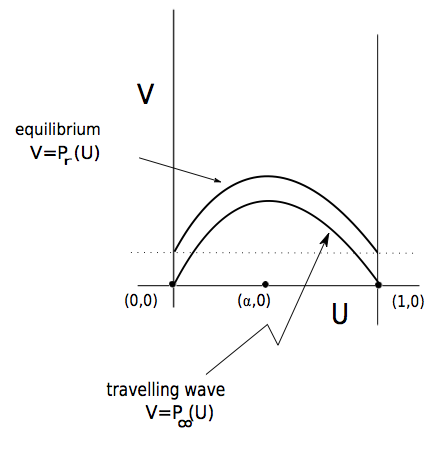}
  \caption{Convergence of the equilibrium to the travelling wave}
  \label{convergence}
\end{figure}

Assume the lemma is not true. Then we will have a sequence of $r_n\to +\infty$
and a $U_0\in [0,1]$ such that $P_{r_n}(U_0)\to V_0>P_\infty(U_0)+\delta$ for
some $\delta >0$. Notice that we have used the fact that $P_r>P_\infty$.
But we know from Lemma~\ref{lema:conv_la} that $\lambda_{r_n}\to
\lambda_\infty$. Hence, by continuous dependence with respect to the initial
conditions and with respect to the parameters appearing in the equation, the
orbit $(\vfi_{r_n}, \vfi_{r_n}')$ converges to the orbit of the ODE with
$\lambda=\lambda_{\infty}$ passing by $(U_0,V_0)$. Since
$V_0>P_\infty(U_0)+\delta$ we have that this orbit takes a finite time to go
from the line $U=0$ to the line $U=1$.  This is a contradiction with the
fact that $r_n\to +\infty$.  \end{proof}

We will normalize the orbit $(\vfi_r,\vfi_r')$ so that the time $\xi=0$ will
correspond to
the unique point for which $\vfi_r(0)=1/2$. Hence, we will denote by $a(r)<0<b(r)$
so that $b(r)-a(r)=r$ and $(\vfi_r(a(r)),\vfi_r'(a(r)))=(0,\theta)$ and
$(\vfi_r(b(r)),\vfi_r'(b(r)))=(1,\theta)$.  In a similar way we may normalize the
travelling wave solution so thar $\vfi_\infty(0)=1/2$.

We have the following

\begin{proposition}\label{a-b-to-infinity}
With the notations above, we have both,
$$
a(r)\to -\infty, \qquad \hbox{ and } \qquad b(r)\to +\infty.
$$
\end{proposition}
\begin{proof}
Assume one of them is not true. For instance, let us consider that there exists a sequence
$r_n\to +\infty$ such that $b(r_n)\to b_0<\infty$. This implies that the finite
interval $[0,b(r_n))$ approaches the finite interval $[0,b_0)$ and therefore by
the continuous dependence of the solutions of the ODE with respect to the
parameters and the initial conditions in a finite time interval \cite{Hale}, we
will have that $(\vfi_{r_n}(b(r_n)), \vfi_{r_n}'(b(r_n)))=(1,\Upsilon_n)\to
(\vfi_\infty(b_0), \vfi'_\infty(b_0))$ and this implies
that $\vfi_\infty(b_0)=1$, which is impossible for any $b_0<\infty$, since
$\vfi_\infty$ is the travelling wave solution.

A similar proof shows that $a(r)\to -\infty$. \end{proof}

We may also prove

\begin{lemma}\label{lema:conver_H1_vfi}
With the notations above, if we extend the function $\{\vfi_r\}$ by 0 to the
left of $a(r)$ and by 1 to the right of $b(r)$ (and we still denote this
function by $\vfi_r$) then
$$
\|\vfi_r-\vfi_\infty\|_{W^{1,\infty}(\bR)}+\|\vfi_r'-\vfi_\infty'\|_{L^2(\bR)}
\to 0, \quad \hbox{  as  } r\to \infty.
$$
\end{lemma}

\begin{proof}
The convergence in $W^{1,\infty}(\bR)$ follows directly from Lemma
\ref{lema:conv_orbit}. Moreover, notice that since $\lambda_r\to \lambda_\infty$
and using that $\lambda_r=-F(1)/\|\vfi_r'\|_{L^2(\bR)}^2$
and $\lambda_\infty=-F(1)/\|\vfi_\infty'\|_{L^2(\bR)}^2$ we have that
$\|\vfi_r'\|_{L^2(\bR)}^2\to \|\vfi_\infty'\|_{L^2(\bR)}^2$.

Hence, consider a small enough parameter $\epsilon>0$		 and let us fix a large
enough interval $[-T,T]$ such that $\|\vfi_\infty'\|^2_{L^2(\bR\setminus
(-T,T))}\leq \eps$.
%, which implies that $\|\vfi_\infty'\|^2_{L^2(-T,T)}\geq
%\|\vfi_\infty'\|^2_{L^2(\bR)}-\epsilon$.
Then, from the convergence of the
orbits given by Lemma~\ref{lema:conv_orbit}, we have that
$\lim_{r\to\infty}\|\vfi'_r-\vfi'_\infty\|_{L^2(-T,T)}^2 = 0$, which implies that $\lim_{r\to\infty}\|\vfi'_r\|_{L^2(-T,T)}^2=\|\vfi'_\infty\|_{L^2(-T,T)}^2$. Hence,
\begin{eqnarray*}
&& \lim_{r\to \infty} \|\vfi'_r\|_{L^2(\bR\setminus(-T,T))}^2 = \lim_{r\to
\infty} \|\vfi'_r\|_{L^2(\bR)}^2-\lim_{r\to\infty}\|\vfi'_r\|_{L^2(-T,T)}^2 \\
&& \qquad \qquad = \|\vfi'_{\infty}\|_{L^2(\bR)}^2-\|\vfi'_{\infty}\|_{L^2(-T,T)}^2=\|\vfi'_{\infty}\|_{L^2(\bR\setminus (-T,T))}^2\leq \epsilon
%\\
%&& \qquad \qquad \le \|\vfi'_{\infty}\|_{L^2(\bR)}^2- \lim_{r\to\infty}( \left| \|\vfi'_r
%-\vfi'_{\infty} \|_{L^2(-T,T)}  - \| \vfi'_{\infty}\|_{L^2(-T,T)} \right|^2 )
%\leq \epsilon,
\end{eqnarray*}
and therefore,
\begin{eqnarray*}
\lim_{r\to \infty} \|\vfi'_r-\vfi'_\infty\|_{L^2(\bR)}^2 &\leq& \lim_{r\to
\infty} \|\vfi'_r-\vfi'_\infty\|_{L^2(-T,T)}^2
\\ && + 2\lim_{r\to\infty}\|\vfi'_r\|_{L^2(\bR\setminus
(-T,T))}^2+2\|\vfi'_\infty\|_{L^2(\bR\setminus (-T,T))}^2\leq 4\epsilon.
\end{eqnarray*}

Since $\eps$ is arbitrarily small, we show the Lemma. \end{proof}

\section{Asymptotic stability of the stationary solutions of the nonlocal problem}
\label{stability-bounded-interval}

We analyze in this section the stability properties of $\vfi_r$, the unique stationary solution of the nonlocal
problem \eqref{modifpb_interval} in  the bounded domain $(a(r),b(r))$. We consider the normalization of this equilibrium explained in the previous subsection, that is $\vfi_r(0)=1/2$ and to simplify the notation we will denote the interval by $(a,b)$ instead of $(a(r),b(r))$, unless it is necessary to specify the dependence of the domain in $r$.
% In many places in this subsection we will  drop the subindex $T$ from the notation and we will denote by $\vfi$ the
%stationary solution in the interval $(a,b)$.

The linearization of \eqref{modifpb_interval} around $\vfi_r$ is given by,
\begin{equation}\label{lineq}
\left\{\begin{array}{l}
 w_t = w_{xx} +  \la(\vfi_r) w_x + f'(\vfi_r) w + \pi_r(w)\vfi'. \qquad x\in
(a,b),
\ t>0, \\
w(a,t) = w(b,t) = 0,
\end{array} \right.
\end{equation}
with $\pi_r$ the linear nonlocal operator
\begin{equation}\label{Pw}
\pi_r (w) = -2\la(\vfi_r) \frac{\langle w_x, \vfi'_r\rangle}{\|\vfi'_r\|_2^2}.
\end{equation}

%In what follows we will drop the subindex $\la$ of the notation and will write
%simply $\vfi$.

The equilibrium $\vfi_r$ will be asymptotically stable if the spectrum
of the linear operator $L^r:D(L^r)\subset L^2(a,b)\to L^2(a,b)$ with $D(L^r)=H^2(a,b)\cap H^1_0(a,b)$, given by
\begin{equation}\label{L}
L^rw := w_{xx} +  \la(\vfi_r) w_x + f'(\vfi_r) w + \pi_r(w)\vfi'_r
\end{equation}
is contained in the left half of the complex plane. We recall that, by
Proposition~\ref{propo:stabinter}, this is the case for the local operator
\begin{equation}\label{L0}
L_0^rw := w_{xx} +  \la(\vfi_r) w_x + f'(\vfi_r) w, \qquad w\in H^1_0(a,b).
\end{equation}
Observe that $L^rw=L_0^rw+\pi_r(w)\vfi'_r$ and the operator $w\to \pi_r(w)\vfi'_r$ has 1-dimensional rank and can be expressed as
$$w\to \frac{2\la(\vfi_r)}{\|\vfi'_r\|_2^2} \vfi'_r\int w\vfi''_r$$

This operator is of the form  $w\to A\langle w, B\rangle$ with
$A(\cdot)=\frac{2\la(\vfi_r)}{\|\vfi'_r\|_2^2} \vfi'_r(\cdot)$ and
$B(\cdot)=\vfi''_r(\cdot)=-\lambda(\vfi_r)\vfi'_r(\cdot) - f(\vfi_r(\cdot))$
and it is a bounded operator from $L^2$ to $L^2$ with finite rank. Several
properties of the operator $L^r$ are inherited from the operator $L_0^r$:  both
operators have the same domain, both operators have compact resolvent and
therefore the spectrum is only discrete, formed by eigenvalues with finite
multiplicity. Nevertheless, all the eigenvalues of operator $L_0^r$ are real
(there is a standard change of variables transforming $L_0^r$ to a selfadjoint
operator) but the operator $L^r$ may not have this property. Actually, unless
$A\equiv B$ operator $L^r$ is not selfadjoint.  There are several studies of
the spectrum of operators of the form $w\to L_0^r(w)+A\langle w, B\rangle$  but
none of them guarantee us that for our particular case, the spectrum lies in
the half complex plane with negative real part. Actually, with the known
results in the literature we are not even able to show that the spectrum of $L$
is real.  See \cite{DaDo06a,DaDo06b,Do08,Frei94,Frei99} for results in this
direction.  One important observation is that in the case that the interval is
the complete real line, that is $r=\infty$, then $\vfi'_\infty$ is the
eigenfunction associated to the eigenvalue $0$ for the operator $L_0^\infty $
and therefore the operator $L^\infty=L_0^\infty+\pi_\infty(w)\vfi'_\infty$ has
an special structure that will allow us to show that
$\sigma(L^\infty)=\sigma(L_0^\infty)$ and that $0\in \sigma(L^\infty)$ with
multiplicity 1.  As a matter of fact this will give us an alternative proof of
the asymptotic stability (with asymptotic phase) of the travelling wave
solution of the nonlocal equation in the whole real line (see Theorem
\ref{th:v}). The fact that $\vfi'_r$ is not an eigenfunction of $L_0^r$, for
finite $r$ (as a matter of fact $\vfi'_r$ does not even satisfy homogeneous
Dirichlet boundary conditions) will not permit us to perform a similar argument
in a bounded interval.  Paradoxically, the analysis in the whole real line is
``simpler'' than the analysis in a bounded interval.

Nevertheless we will be able to prove the asymptotic stability of the
stationary solution of the non local problem \eqref{modifpb_interval} for large
enough intervals using a perturbative method. The proof is divided into three
parts. In the first one we prove some properties of the spectrum of the non
local operator \eqref{L} on the finite interval $(a,b)$. We next fully analyze
the spectrum of the limit operator on the whole line $\bR$. Finally, we prove
the convergence of the spectrum of $L^r$ to the spectrum of $L^\infty$ as $r\to
+\infty$.

\subsection{Spectral properties for any finite interval}\label{subsec-spectral-bounded}
The results in this section apply to the stationary solution $\vfi_r$ obtained
in Theorem \ref{th:nonloc_interval} in the finite interval $(a,b)$.

Let us start with a general and rough estimate of the spectrum of $L^r$ but
which is uniform for all $r\geq 1$.

\begin{proposition}\label{uniform-sector}
There exist $\rho_0\in \bR^+$ and $\phi\in (\pi,2\pi)$ such that if we define
the sector $\Sigma_{\rho_0,\phi}=\{ z\in \bC, |Arg(z-\rho_0)|>\phi\}$, then
$\sigma(L^r)\subset \Sigma_{\rho_0,\phi}$ for all $r\geq 1$.
\end{proposition}

\begin{proof}
Note that $\mu\in \sigma(L^r)$ if and only if there exists $u\in H^2(a,b)\cap
H^1_0(a,b)$ such that $L^r u=\mu u$. But the operator $L^r$ can be written as
$L^r u= \Delta u + N(u)$ where $N:H^1_0(a,b)\to L^2(a,b)$  is defined as
$N_r(u)=\lambda(\vfi_r) u_x+f'(\vfi_r)u+\pi_r(u)\vfi'_r$  and as usual $\Delta
u= u_{xx}$. Observe that from Lemmas \ref{lema:conv_la} and
\ref{lema:conver_H1_vfi} we have that the operator $N_r$ is bounded uniformly
in $r$ for $r\geq 1$, that is, there exists a constant $C_0$ independent of
$r=b-a$ such that $\|N_r u\|_{L^2(a,b)}\leq C_0\|u\|_{H^1_0(a,b)}$, for all
$r\geq 1$.

On the other hand, standard estimates using the spectral decomposition of
$-\Delta$ with Dirichlet boundary conditions in $(a,b)$ show that for $\mu \not
\in \bR^-$
$$
\|(-\Delta +\mu I)^{-1}\|^2_{\mathcal{L}(L^2(a,b), H^1_0(a,b))}\leq
\frac{1}{\hbox{dist}(\mu, \bR^-)}+\frac{|\mu|}{(\mu+|\mu|)^2}.
$$

Hence, fixing  $\phi\in (\pi,2\pi)$  we can choose $\rho_0>0$ large enough so that we have
$$
\|(-\Delta +\mu I)^{-1}\|^2_{\mathcal{L}(L^2(a,b), H^1_0(a,b))}\leq
\frac{1}{(2C_0)^2},\quad \forall \mu\in \bC\setminus \Sigma_{\rho_0,\phi}.
$$
Therefore, if $\mu\in \bC\setminus \Sigma_{\rho_0,\phi}$ and if there exists
$u\in H^2(a,b)\cap H^1_0(a,b)$ such that $L^r(u)=\mu u$, then,  $u=N\circ
(-\Delta +\mu I)^{-1} u$ which implies that $\|u\|_{L^2}\leq
\|N\|_{\mathcal{L}(H^1_0,L^2)}\,  \|(-\Delta +\mu
I)^{-1}\|_{\mathcal{L}(L^2,H^1_0)}\|u\|_{L^2}\leq
C_0\frac{1}{2C_0}\|u\|_{L^2}\leq 1/2\|u\|_{L^2}$ and therefore $u\equiv 0$,
which implies that $\mu\not\in \sigma(L^r)$. \hfill\end{proof}

This rough estimate of the spectrum of $L^r$ allows us to prove that if there
is an eigenvalue of $L^r$ with positive real part, then necessarily we will
have that it is uniformly bounded in $r$, that is

\begin{corollary}\label{bound-spectrum-Lr}
With the notations of the previous proposition, we have that  for any value $\nu>0$, we have
$$
\{ z\in \sigma(L^r), \real z\geq - \nu\}\subset
\{ z\in \bC, -\nu\leq \real z\leq \rho_0, |Im(z)|\leq (\rho_0+\nu)\sin(\phi)\}.
$$
\end{corollary}

\begin{lemma}\label{lema:autoL}
Let $\mu \in \sigma(L^r)\bigcap\rho(L_0^r)$. Then, $\mu$ is at most a geometrically simple
eigenvalue of $L^r$, that is, $Ker(L^r-\mu I)$ is one dimensional. Moreover,
the associated eigenspace is generated by $y$, the unique
solution of
\begin{equation}\label{defy}
(L_0^r -\mu I)y = \vfi'_r,
\end{equation}
and
\begin{equation}\label{py-1}
\pi_r(y)= -2\la \frac{\langle y',\vfi'_r \rangle}{\|\vfi'_r\|_2^2}= -1.
\end{equation}
\end{lemma}
\begin{proof}
Let $w\ne 0$ be such that $L^rw=\mu w$. Then,
$$
0  = (L^r -\mu I)w =  (L_0^r -\mu I)w + \pi_r(w)\vfi'_r =(L_0^r -\mu I)(w+\pi_r(w)y),
$$
so that
\begin{equation}\label{wy}
w = -\pi_r(w)y.
\end{equation}
The above implies that $\mu$ is at most a simple eigenvalue of $L^r$ with
eigenspace generated by $y$. Identity \eqref{py-1} follows after applying the
linear operator $\pi_r$ in \eqref{wy}. \hfill \end{proof}
\par\medskip

However, it will be still useful for the last part of our argument.

\begin{proposition}\label{propo:autopositivo}
There is no real eigenvalue $\mu \ge 0$ in $\sigma(L^r)$.
\end{proposition}
\begin{proof}
Let us assume that there exists an eigenvalue $\mu \ge 0$ of $L^r$. Since
$\sigma(L_0^r)\subset\{ z\in \bC: \real z < 0\}$ (see Proposition
\ref{propo:stabinter} and Corollary \ref{negative-spectrum}) then $\mu\in
\rho(L_0^r)$.  Let $y$ be as in \eqref{defy} for this value of $\mu$. Then, by
definition, $y$ satisfies
\begin{equation}\label{eqy}
y'' + \la y' + (f'(\vfi_r) - \mu)y  = \vfi'_r.
\end{equation}
Multiplying in \eqref{eqy} by $\vfi'_r$ and integrating in $(a,b)$ we have
$$
\langle y'',\vfi' \rangle + \la \langle y', \vfi'_r \rangle + \langle f'(\vfi_r)y, \vfi'_r \rangle - \mu
\langle y,\vfi'_r\rangle = \|\vfi'_r\|^2.
$$
But
$$
\langle f'(\vfi_r)y, \vfi'_r \rangle = -\langle y', f(\vfi_r) \rangle =
\langle y', \vfi''_r + \la \vfi'_r \rangle,
$$
so that
$$
\langle y'',\vfi_r' \rangle + \langle y',\vfi_r'' \rangle + 2\la \langle y', \vfi_r' \rangle  - \mu
\langle y,\vfi_r'\rangle = \|\vfi_r'\|^2.
$$
Then, by \eqref{py-1}, it holds
\begin{equation}\label{less-or-equal-0}
\langle y'',\vfi_r' \rangle + \langle y', \vfi_r'' \rangle =  \mu
\langle y,\vfi_r'\rangle \le 0,
\end{equation}
where the inequality follows from the maximum principle applied to  $-L^r_0 y +
\mu y=-\vfi'_r$ and taking into account that $\vfi'_r>0$, so that $y<0$ in
$(a,b)$.  But, from Lemma~\ref{lema:equi_modipb_ab}, we know that $\theta :=
\vfi'_r(a)=\vfi'_r(b) > 0$, which implies, together with
\eqref{less-or-equal-0}, that
$$
\langle y'',\vfi'_r \rangle + \langle y', \vfi''_r \rangle =
y'\vfi'_r|_a^b = \theta(y'(b)-y'(a)) \le 0,
$$
and therefore $y'(b) \le y'(a)$. But, on the other hand, the fact that $y<0$ in
$(a,b)$ together with $y=0$ in $x=a,b$, imply that $y'(a)\leq 0\leq y'(b)$.
Therefore $y'(a)=y'(b)=0$. But this is impossible, since if, for instance,
$y'(a)=0$, then $y$ is a solution of the initial value problem $L^r_0
y=\vfi'_r$ in $(a,b)$ with $y(a)=y'(a)=0$ and this implies that
$y''(a)=\vfi'_r(a)>0$, so that with $x$ near $a$ we have $y>0$ which is not
true. \hfill \end{proof}

\begin{nota}\label{Remark-no-maximum}
i) This proposition would be enough to finish the proof of the asymptotic
stability if the non local operator $L^r$ had the property that the eigenvalue with the largest real part
were real. For instance, this could be obtained if $L^r$ satisfies the hypothesis for a Krein-Rutmann type
of theorem.  But for this theorem we need to have maximum principles and are unable to prove this
principles for this nonlocal operator.

ii) Observe that this proposition does not exclude the possibility of having
complex eigenvalues with positive real part. Actually, we will be able to
exclude this possibility only for large enough intervals by using a peturbative
argument. The fact that this operator may present complex eigenvalues with
positive real parts for some intervals $(a,b)$ is an open interesting question.

\end{nota}

\par\medskip
\subsection{Spectrum of the nonlocal problem in the whole line}
\label{subsec-spectral-nonlocal-bounded}
In this section we analyze in detail the spectrum of the corresponding nonlocal
operator in the entire real line. This operator is the one associated to the
linearization around the asymptotic equilibrium $\vfi_\infty$, that is,
\begin{equation}\label{lin_R}
L^\infty (w)= w_{xx} + \lambda(\vfi_\infty) w_x + f'(\vfi_\infty) w + \pi_\infty(w)\vfi_\infty',
\end{equation}
where now $\pi_\infty$ stands for the linear operator
\begin{equation}\label{pwR}
\pi_\infty(w) = \frac{-2\lambda(\vfi_\infty)}{\|U'\|_2^2}\langle U',w_x \rangle.
\end{equation}
%The product and the norm above are in $L^2(\bR)$ and $p$ is bounded for
%instance in $W^{1,1}(\bR)$.

%Thus, we analyze the spectrum of $L$ in \eqref{L} for $\vfi=U$, $(a,b)=\bR$, in
%the underlying space $X=C_{unif}(\bR)$, $X=L_p(\bR)$, $1\le p < \infty $, or
%$X=C_0(\bR)$.

We will use several important properties of the spectrum of the local operator
\begin{equation}\label{operatorA}
L_0^\infty w := w'' + \lambda(\vfi_\infty) w' + f'(\vfi_\infty) w,
\end{equation}
We have the following,
\begin{lemma}\label{local-operator-behavior}
With respect to the spectrum of $L_0^\infty$, defined by \eqref{operatorA}, we
have

\par\noindent (i) The essential spectrum  $\sigma_{ess}(L_0^\infty)\subset\{z\in \bC:  \real z\leq \max\{f'(0),f'(1)\}\}$.
\par\noindent (ii) There exists $0<\nu< -\max\{f'(0),f'(1)\}\}$ such that $\sigma(L_0^\infty)\cap\{z\in \bC: \real z \geq -\nu\}=\{ 0\}$ and
the eigenfunction associated to $\mu=0$ is $\vfi'_\infty$.
\par\noindent (iii) There is no solution $w\in H^2(\bR)$ of $L_0^\infty w=\vfi'_\infty$.
Therefore, $0$ is an algebraically simple eigenvalue of $L_0^\infty$, that is,
$Ker( (L_0^\infty)^2)=Ker(L_0^\infty)=$span$\{\vfi'_\infty\}$
\end{lemma}

We refer to Appendix \ref{proof-lemma} for a proof of this result.

\par\bigskip
Both $L_0^\infty$ and $L^\infty$ are sectorial operators and are related by
\begin{equation}\label{rel_AL}
L^\infty (u) = L_0^\infty (u) + \pi_\infty(u) \cdot \vfi_\infty'
\end{equation}

In the following Proposition we show that $L^\infty$ enjoys the same spectral
properties as $L_0^\infty$.

\begin{proposition}\label{propo:spec_L}
Let $L^\infty$ be the linear operator defined above in \eqref{rel_AL}. Then
$\sigma(L^\infty)=\sigma(L_0^\infty)$ and $0\in \sigma(L^\infty)$ is an
algebraically simple eigenvalue of $L^\infty$. In particular the three items
$(i)$, $(ii)$ and $(iii)$ from Lemma \ref{local-operator-behavior} also hold
for $L^\infty$.
\end{proposition}

\begin{proof}
By applying integration by parts it is easy to see that
$\pi_\infty(\vfi'_\infty) = 0$. Then, since  $L_0^\infty \vfi'_\infty=0$, we
also have that $L^\infty\vfi'_\infty=0$, so that $0\in \sigma(L^\infty)$ with
associated eigenfunction $\vfi'_\infty$. In order to see that 0 is a simple
eigenvalue of $L^\infty$, let us consider first $\phi \in D(L^\infty)$ with
$L^\infty\phi = 0$. In case $\pi_\infty(\phi) =0$, then it is $L_0^\infty
\phi=0$ and, since 0 is a simple eigenvalue for $L_0^\infty$, it must be $\phi
\sim \phi'_\infty$. Let us assume now that $\pi_\infty(\phi) \ne 0$. Then, it
follows $L_0^\infty \phi = -\pi_\infty(\phi)\vfi'_\infty$, which is impossible
by Lemma \ref{local-operator-behavior} $(iii)$. With a very similar argument it
is possible to show that there is no $w\in H^2$ satisfying $L^\infty
w=\phi'_\infty$.  Hence, 0 is an algebraically simple eigenvalue of $L^\infty$.

We show now that $\rho(L_0^\infty)\subset \rho(L^\infty)$. So, let $\mu \in
\rho(L_0^\infty)$, $f\in X$ and $w_f \in D(L_0^\infty)$ be the unique element
of $D(L_0^\infty)$ such that
\begin{equation}\label{wf}
L_0^\infty w_f - \mu w_f = f.
\end{equation}
If $\pi_\infty(w_f) = 0$, then $L^\infty w_f- \mu w_f = f$. In case
$\pi_\infty(w_f) \ne 0$, we can consider
\begin{equation}\label{w*}
w^* = w_f + \frac{1}{\mu}
\pi_\infty(w_f)\vfi'_\infty,
\end{equation}
since we already know that $\mu \ne 0$. Then, using that $L_0^\infty
\vfi'_\infty= \pi_\infty(\vfi'_\infty)= 0$ and \eqref{rel_AL}, one gets
\begin{equation} \label{L_onto}
L^\infty w^*- \mu w^* = f.
\end{equation}
This proves that $L^\infty-\mu I$ is onto.

Let us assume now that there exist two elements $w_1^*, w_2^* \in D(L^\infty)$
with
$$
L^\infty w_j^* - \mu w_j^* = f, \qquad \mbox{ for }\ j=1,2.
$$
Then
$$
L_0^\infty w_j^* - \mu w_j^* = -\pi _\infty(w_j^*) \vfi'_\infty + f, \qquad \mbox{ for }\ j=1,2.
$$
From the above it is clear that $\pi_\infty (w_1^*)=\pi_\infty (w_2^*)$ implies
$w_1^* = w_2^*$, since $L_0^\infty-\mu I$ is one to one. In case
$\pi_\infty(w_1^*) \ne \pi_\infty(w_2^*)$, we consider $\bw=w_1^*-w_2^*$ and we
get $ L_0^\infty \bw- \mu \bw = -\pi_ \infty(\bw)\vfi'_ \infty, $, which
implies $ \bw = -\pi_ \infty(\bw) (L_0^\infty -\mu I)^{-1} \vfi'_\infty. $ But
$(L_0^\infty-\mu I)^{-1}\vfi'_\infty=-\vfi'_\infty/ \mu$ and therefore $\bw\sim
\vfi'_\infty$ and $\pi_\infty(\bw)=0$, which is a contradiction.

The fact that $(L^\infty-\mu I)^{-1}$ is bounded is clear from the expression
$$
(L^\infty-\mu I)^{-1} f = (L_0^\infty-\mu I)^{-1}f + \frac 1{\mu} \pi_ \infty(w_f) \vfi'_\infty.
$$
which is obtained from \eqref{w*}-\eqref{L_onto}. This shows that
$\rho(L_0^\infty) \subset \rho(L^\infty)$.

The proof of the other inclusion $\rho(L^\infty) \subset \rho(L_0^\infty)$ is
completely symmetrical to this one,  once we know that $0$ is also a simple
eigenvalue of $L^\infty$. We just need to express $L_0^\infty=L-\pi_\infty
\vfi'_\infty$ and remake the proof we have just shown. \hfill \end{proof}

\par\bigskip

\begin{nota}\label{essential-spectrum}
(i) In particular, we have that the spectrum of $L^\infty$ apart from $0$ is
located in the left half plane, i.e., there exists $\nu>0$, such that
\begin{equation}\label{des_spec}
\sup \{ \real \mu : \mu \in \sigma(L^\infty),\ \mu \ne 0 \} = -\nu, \quad
\mbox{ for some } \ \nu >0.
\end{equation}

(ii) Observe also that since the operator $u\to \pi_\infty(u)\vfi'_\infty$ is a
compact operator then
$\sigma_{ess}(L_\infty)=\sigma_{ess}(L_\infty^0)\subset\{z\in \bC: \real z \leq
\max\{f'(0),f'(s)\}\,\}$, see \cite{Henry}.

\end{nota}

\subsection{Spectral convergence and asymptotic stability of the stationary solution}
\label{subsec-spectral-convergence} In this subsection we will end up proving
the asymptotic stability of the stationary solution $\vfi_r$ of the non local
problem. We will obtain this via a convergence of the spectrum of the operator
$L^r$ to $L^\infty$. In  order to prove this spectral convergence we will use
the theory of regular convergence developed in \cite{Vai,Vai77a,Vai79} and the
related results in \cite{BeRott07}. The necessary definitions are given below.

Let $E$ and $F$ denote separable Banach spaces and let $\{E_r\}_{r>0}$ and
$\{F_r\}_{r>0}$ be families of separable Banach spaces. Let $\{p_r\}_{r>0}$,
$p_r \in \cL(E,E_r)$ and $\{q_r\}_{r\in\bN}$, $q_r\in \cL(F,F_r)$ be linear
bounded operators such that
\begin{equation}\label{conver_proye}
\begin{array}{l}
\lim_{r\to \infty} \|p_r e \|_{E_r} \to \|e\|_E, \qquad \mbox{for every }\ e\in E \quad \mbox{and }\\
\lim_{r\to \infty} \|q_r f \|_{F_r} \to \|f\|_F, \qquad \mbox{for every }\ f\in F.
\end{array}
\end{equation}
A family $\{e_r\}_{r>0}$, $e_r \in E_r$, is said to be $\cP$-convergent to
$e\in E$, written $e_r \Pcon e$, if
\begin{equation}\label{defPcon}
\lim_{r\to \infty}\|e_r -p_r e\|_{E_r} = 0.
\end{equation}

A family $\{e_r\}_{r>0}$, $e_r \in E_r$, is said to be $\cP$-compact if every
infinite sequence contains a $\cP$-convergent subsequence. Analogous definitions
apply for $\cQ$-convergence and $\cQ$-compactness.

A family of bounded linear operators $\{A_r\}_{r>0}$, $A_r \in
\cL(E_r,F_r)$, is said to be $\cPQ$-convergent to $A\in \cL(E,F)$, written
$A_r\PQcon A$, as $r\to \infty$, if $e_r\Pcon e$ implies $A_r e_r \Qcon Ae$, as
$r \to \infty$. The $\cPQ$-convergence is said to be regular if for every
bounded sequence $\{e_r\}_{r>0}$, $e_r\in E_r$, such that the sequence $\{A_r
e_r\}_{r>0}$ is $\cQ$-compact, it turns out that  $\{e_r\}_{r>0}$ is
$\cP$-compact.

The relevance of the $\cPQ$ regular convergence is that we obtain the following
result, which is taken from \cite{Vai,Vai77a,Vai79} in a simplified version.

\begin{theorem} \label{eigenvalue-convergence}
Assume we have the family of operators $A(s)=A-sB\in \mathcal{L}(E,F)$ and
$A_r(s)=A_r-sB_r\in \mathcal{L}(E_r,F_r)$ where the parameter $s\in S$, a
bounded subset of the complex plane $\bC$, which satisfy the following
hypotheses:

\par\noindent (i) $A_r(s)$ $\cPQ$-converges regularly to $A(s)$ for all $s\in S$.

\par\noindent (ii) For each $s\in S$ the operators $A_r(s)$ and $A(s)$ are Fredholm with index 0.

\par\noindent (iii) There exists $s'\in S$ such that $Ker(A(s'))=\{0\}$.

\par\noindent (iv) There exists a constant $C=C(S)$ such that $\|A_r(s)\|_{\mathcal{L}(E_r,F_r)}\leq C$ for
all $r\geq 0$.

\par\medskip \noindent Then, if we denote by $W(s_0)$ the ``root subspace"
associated to $A(s_0)$, that is, the linear space  generated by the chain of
vectors $\{e_0, e_1,\ldots, e_k,\ldots\}$ defined as,
$$
(A-s_0B)e_0=0,\quad
(A-s_0B)e_1=B e_0,\ldots \quad (A-s_0B)e_k=Be_{k-1}, \ldots
$$
and if we denote by $W_r(s_0, \delta)$ the hull of all ``root subspaces''
associated to $A_r(s)$ for all $|s-s_0|\leq \delta$, $s\in S$, then we have
that for $\delta>0$ small enough
$$
\hbox{dist}_H(W_r(s_0,\delta), W(s_0))\to 0, \quad \mbox{ as } r\to +\infty,
$$
and therefore there exists a $\delta>0$ small such that
$$
\hbox{dim}(W_r(s_0,\delta))=dim(W(s_0)), \quad \mbox{ as } r\to+\infty.
$$

\end{theorem}

\begin{proof} See the proof in \cite{Vai, Vai77a,Vai79}. \end{proof}

\par\bigskip Let us write our operators in such a way that we can obtain the regular convergence.
Consider the following setting. Let $E=H^1(\bR, \bC^2)$ and $F=L^2(\bR,\bC^2)$,
the spaces in the whole real line. Also, $E_r=H^1(I_r, \bC^2)$ and
$F_r=L^2(I_r, \bC^2)\times \bC^2$, the spaces in the finite interval $I_r$ and
observe that the space $F_r$ has two extra coordinates.

Define the family of linear operators $p_r: E\to E_r$ and $q_r:F\to F_r$ as
$$
p_r\begin{pmatrix} u \\  v	
\end{pmatrix}=\begin{pmatrix} u_{| I_r} \\  v_{| I_r}	
\end{pmatrix}
$$
and
$$
q_r\begin{pmatrix} u \\  v	
\end{pmatrix}=\begin{pmatrix} u_{| I_r} \\  v_{| I_r} \\0 \\0 	
\end{pmatrix}.
$$
Consider the family of operators $A_{\infty}^s, A_{0,\infty}^s, \Pi_\infty:E\to
F$, defined as,
$$
A_{\infty}^s\begin{pmatrix} u \\  v	
\end{pmatrix}=
\begin{pmatrix}
u_x\\ v_x
\end{pmatrix} +  \begin{pmatrix}
0 & -I\\ f'(\vfi_\infty)-s & \lambda_\infty
\end{pmatrix}
\begin{pmatrix} u \\  v	
\end{pmatrix}
+
\begin{pmatrix} 0 \\  \pi_\infty(u)\vfi'_\infty	
\end{pmatrix},
$$

$$
A_{0, \infty}^s\begin{pmatrix} u \\  v	
\end{pmatrix}=
\begin{pmatrix}
u_x\\ v_x
\end{pmatrix} + \begin{pmatrix}
0 & -I\\ f'(\vfi_\infty)-s & \lambda_\infty
\end{pmatrix}
\begin{pmatrix} u \\  v	
\end{pmatrix},
$$
and
$$
\Pi_{\infty}\begin{pmatrix} u \\  v	
\end{pmatrix}=
\begin{pmatrix}
0\\  \pi_\infty(u)\vfi'_\infty
\end{pmatrix}
$$
and observe that $A_{\infty}^s=A_{0,\infty}^s+\Pi_\infty$. Moreover, the
operator $A^s_{0,\infty}$ is a local operator. The operator $A_{\infty}^s$ can
also be decomposed as
\begin{equation}\label{decomposition-Ainfinity}
A_\infty^s=A_\infty^0 -sB_\infty,
\end{equation}
where
\begin{equation}\label{B-infinity}
B_\infty\begin{pmatrix} u \\  v	
\end{pmatrix}=
\begin{pmatrix}
0\\ u
\end{pmatrix}.
\end{equation}

With respect to the operators in a bounded interval, we define, $A^s_{r},
A^s_{0,r}, \Pi_{r}: E_r\to F_r$ as
$$
A_{r}^s\begin{pmatrix} u \\  v	
\end{pmatrix}=
\begin{pmatrix}
u_x\\ v_x \\ 0 \\0
\end{pmatrix} + \begin{pmatrix}
0 & -I\\ f'(\vfi_r)-s & \lambda_r  \\ 0 & 0Ê\\0 & 0
\end{pmatrix}
\begin{pmatrix} u \\  v 	
\end{pmatrix}
+
\begin{pmatrix} 0 \\  \pi_r(u)\vfi'_r \\0 \\0
\end{pmatrix}
+
\begin{pmatrix} 0 \\  0 \\ u(a(r)) \\ u(b(r))
\end{pmatrix},
$$

$$
A_{0,r}^s\begin{pmatrix} u \\  v	
\end{pmatrix}=
\begin{pmatrix}
u_x\\ v_x \\ 0 \\0
\end{pmatrix} +  \begin{pmatrix}
0 & -I\\ f'(\vfi_r)-s & \lambda_r  \\ 0 & 0Ê\\0 & 0
\end{pmatrix}
\begin{pmatrix} u \\  v 	
\end{pmatrix}
+
\begin{pmatrix} 0 \\  0 \\ u(a(r)) \\ u(b(r))
\end{pmatrix}
$$
and
$$
\Pi_{r}\begin{pmatrix} u \\  v	
\end{pmatrix}=
\begin{pmatrix}
0\\  \pi_r(u)\vfi'_r \\0 \\0
\end{pmatrix}
$$

and observe that in a similar way, we have
\begin{equation}\label{decomposition-Ar}
A_r^s=A_r^0 -sB_r,
\end{equation}
with
$$
B_r\begin{pmatrix} u \\  v	
\end{pmatrix}=
\begin{pmatrix}
0\\ u\\ 0\\ 0
\end{pmatrix}.
$$

We have the following,

\begin{proposition}\label{pr:conver_regular}
With the notation above,  for any $s \in \{\real s > -\nu \}$, we have

\par (i) The sequence of operators  $A^s_{0,r}$  $\cPQ$-converges regularly to $A^s_{0,\infty}$ as $r\to \infty$.

\par (ii) The sequence of operators  $A^s_{r}$  $\cPQ$-converges regularly to $A^s_{\infty}$  as $r\to \infty$.

\par (iii)   The family of operators $A^s_\infty$, $A^s_r$ are Fredholm operators of index 0.

\end{proposition}
\par\medskip

\begin{proof}
$(i)$ Let us define the auxiliary operator $\tilde A_{0,r}^s:E_r\to F_r$ which
is given by,
$$
\tilde A_{0,r}^s\begin{pmatrix} u \\  v	
\end{pmatrix}=
\begin{pmatrix}
u_x\\ v_x \\ 0 \\0
\end{pmatrix} + \begin{pmatrix}
0 & -I\\ f'(\vfi_\infty)-s & \lambda_\infty \\ 0 & 0Ê\\0 & 0
\end{pmatrix}
\begin{pmatrix} u \\  v 	
\end{pmatrix}
+
\begin{pmatrix} 0 \\  0 \\ u(a(r)) \\ u(b(r))
\end{pmatrix}
$$
where we consider that $f'(\vfi_\infty)$ is restricted to the interval $I_r$.   Notice that
$A_{0,r}^s=\tilde A_{0,r}^s+ B_r$ where
$$
B_{r}\begin{pmatrix} u \\  v	
\end{pmatrix}=\begin{pmatrix}
0 & 0\\ -f'(\vfi_\infty)+f'(\vfi_r) & -\lambda_\infty+\lambda_r \\ 0 & 0Ê\\0 & 0
\end{pmatrix}
\begin{pmatrix} u \\  v 	
\end{pmatrix}
%=\begin{pmatrix}
%0 \\ (f'(\vfi_\infty)-f'(\vfi_r))u+(\lambda_\infty-\lambda_r)v \\ 0 \\0
%\end{pmatrix}
$$
and observe that $\|B_r\|_{\mathcal{L}(E_r,F_r)}\to 0$ as $r\to +\infty$, since
$\lambda_r\to \lambda_\infty$ and $\|f'(\vfi_\infty)-f'(\vfi_r)\|_{L^\infty}\to
0$, see Lemmas~\ref{lema:conv_la} and \ref{lema:conver_H1_vfi}.

But from \cite{BeRott07} we know that $\tilde A_{0,r}^s$ $\cPQ$-converges
regularly to $A_{0,\infty}^s$. This convergence is not trivial at all and it
uses deep techniques like exponential dichotomy. Also, it is worthwile to
mention that this result is implicit in the work of Beyn and Lorenz
\cite{BeLo99}.

The fact now that $\|B_r\|_{\mathcal{L}(E_r,F_r)}\to 0$ implies easily the
result.

\par\noindent $(ii)$ Once we have obtained the convergence for the ``local''
operators and recalling that $A_r^s=A_{0,r}^s+\Pi_r$,
$A_\infty^s=A_{0,\infty}^s+\Pi_\infty$ we obtain the $\cPQ$ regular convergence
from the $\cPQ$ regular convergence of $A_{0,r}^s$ to $A_{0,\infty}^s$, the
$\cPQ$ convergence of $\Pi_r$ to $\Pi_\infty$ and the fact that both $\Pi_r$
and $\Pi_\infty$ are operators with a 1-dimensional rank.

\par\medskip\noindent $(iii)$ Let us divide the proof in two parts.
\par $(iii$-1) {\bf $A_r^s$ is Fredholm with index 0}. Operator $A_r^s$ is defined
in the finite interval $I_r$ where we have the compact embedding
$H^1(I_r,\bC)\hookrightarrow L^2(I_r,\bC)$. This implies in particular that the
operator
$$
\begin{pmatrix} u \\  v	
\end{pmatrix}
\longrightarrow A_{r}^s\begin{pmatrix} u \\  v	
\end{pmatrix}
-\begin{pmatrix}
u_x\\ v_x \\ 0 \\0
\end{pmatrix}=
   \begin{pmatrix}
0 & -I\\ f'(\vfi_r)-s & \lambda_r \\ 0 & 0Ê\\0 & 0
\end{pmatrix}
\begin{pmatrix} u \\  v 	
\end{pmatrix}
+
\begin{pmatrix} 0 \\  0 \\ u(a(r)) \\ u(b(r))
\end{pmatrix}
$$
is a compact operator from $E_r$ to $F_r$, since it is a bounded operator from
$L^2(I_r,\bC)\times L^2(I_r,\bC)$ to $F_r$.

Hence, the operator $A^s_r:E_r\to F_r$ is a Fredholm operator of index 0 if and
only if the bounded operator  $D_r:E_r\to F_r$, given by
$$D_r\begin{pmatrix} u \\  v	
\end{pmatrix}
= \begin{pmatrix}
u_x\\ v_x \\ 0 \\0
\end{pmatrix}
$$
is a Fredholm operator of index 0. But this is very easy to show, since
$Ker(D_r)=\{ (u,v)\in E_r:  u=constant, v=constant\}\equiv \bC\times\bC$ and
therefore dim$(Ker(D_r))=2$. Moreover, the rank of $D_r$ is $L^2(I_r)\times
L^2(I_r)\times \{0\} \times\{ 0\}\subset F_r$ which has codimension 2.

\par $(iii$-2) {\bf  $A_\infty^s$ is Fredholm with index 0}. Observe that the operator
$A_\infty^s$ can be decomposed as $A^s_\infty=F^s_\infty+K_\infty+\Pi_\infty$
where $\Pi_\infty$ is given as above (and is a compact operator since it has
rank=1), and the other two operators are given as
$$
K_\infty \begin{pmatrix} u\\ v \end{pmatrix}=\begin{pmatrix} 0\\
[f'(\vfi_\infty)-V(\cdot)]u \end{pmatrix},
$$
and
$$
F_{\infty}^s\begin{pmatrix} u \\  v	
\end{pmatrix}=
\begin{pmatrix}
u_x\\ v_x
\end{pmatrix} + \begin{pmatrix}
0 & -I\\ V(\cdot)-s & \lambda_\infty
\end{pmatrix}
\begin{pmatrix} u \\  v	
\end{pmatrix},
$$
where the potential $V(x)$ is piecewise constant and it is defined as
$$
V(x)=
\left\{\begin{array}{l}
f'(0), \quad x\in (-\infty,0], \\
f'(1), \quad x\in (0,\infty).
\end{array}
\right.
$$

But the fact that $\vfi_\infty(x)\to  0$ as $x\to -\infty$ and
$\vfi_\infty(x)\to 1$ as $x\to +\infty$, implies that
$f'(\vfi_\infty(x))-V(x)\to 0$ as $x\to \pm\infty$ and therefore, the operator
$K_\infty:E_\infty\to F_\infty$ is a compact operator. Hence, $A^s_\infty$ is a
Fredholm operator of index 0 if and only if $F^s_\infty$ is a Fredholm operator
of index 0.

The operator  $F^s_\infty$ is written as
$$
F^s_\infty\begin{pmatrix}
u\\ v
\end{pmatrix}
=\begin{pmatrix}
u_x\\ v_x
\end{pmatrix} + M(x,s) \begin{pmatrix} u \\  v	
\end{pmatrix},
$$
with $M(x,s)$ the piecewise constant matrix function,
$$
M(x,s)=M_-(s)=\begin{pmatrix}
0 & -I\\ f'(0)-s & \lambda_\infty
\end{pmatrix} \qquad x<0,
$$
$$
M(x,s)=M_+(s)=\begin{pmatrix}
0 & -I\\ f'(1)-s & \lambda_\infty
\end{pmatrix} \qquad x>0,
$$
and recall that both $f'(0),f'(1)<0$.

To show that $F^s_\infty$ is Fredholm with index 0, we will show that
$Ker(F^s_\infty)=\{0\}$ and $R(F^s_\infty)=L^2(\bR,\bC^2)$. The fact that
$Ker(F^s_\infty)=\{0\}$ is proved as follows.   Let $(u,v)\in  H^1(R,\bC^2)$
such that
$$F^s_\infty\begin{pmatrix}
u\\ v
\end{pmatrix}
=\begin{pmatrix}
0\\ 0
\end{pmatrix}.
$$
Then, if we consider this equation in $x<0$  (reps. $x>0$), it is a linear
$2\times 2$ ODE with constant coefficient whose solution can be obtained
explicitly. Since $(u,v)\in H^1(\bR,\bC)$ we have that $(u,v)$ is a bounded
function as $|x|\to \infty$ and therefore, necessarily the behavior of the
solution as $x\to -\infty$ (resp. $x\to+\infty$) is completely determined by
the spectral decomposition of the matrix $M_-(s)$ (resp. $M_+(s)$).

Direct computations show that both matrices $M_-(s)$  and $M_+(s)$ are
hyperbolic matrices (no eigenvalues with 0 real part), each of them has one
eigenvalue with positive real part and the other with negative real part. If we
denote by $ \alpha_p(s)$ the eigenvalue with positive real part of $M_-(s)$
which has $\begin{pmatrix} 1\\ \alpha_p(s) \end{pmatrix}$ as its associated
eigenvector  (unstable manifold of 0 of $M_ -(s)$) and by $\omega_n(s)$
 the eigenvalue with negative real part of $M_ +(s)$ which has $\begin{pmatrix}
1\\ \omega_n(s) \end{pmatrix}$ as its associated eigenvector (stable manifold
of 0 of $M_+(s)$) then we necessarily have that
$$
\begin{pmatrix}
u(x)\\ v(x)
\end{pmatrix}
=c_p e^{ \alpha_p x} \begin{pmatrix}
1\\ \alpha_p(s) \end{pmatrix}, \hbox{  for } x<0, \quad
\begin{pmatrix}
u(x)\\ v(x)
\end{pmatrix}
=c_ne^{ \omega_n x} \begin{pmatrix}
1\\ \omega_n(s) \end{pmatrix},\hbox{  for } x>0,
$$
for some constants $c_p,c_n\in \bC$.
But since $(u,v)\in H^1(\bR,\bC)$, we necessarily mut have
$$
c_p\begin{pmatrix}
1\\ \alpha_p(s) \end{pmatrix}=c_n\begin{pmatrix} 1\\ \omega_n(s)
\end{pmatrix}
$$
and this is impossible unless $c_p=c_n=0$ since
$\real \alpha_p(s)>0$ and $\real \omega_n(s) <0$ and therefore both vectors are
linearly independent. This shows that $(u,v)=(0,0)$ and therefore,
$Ker(F^s_\infty)=\{0\}$.

To show that $R(F^s_\infty)=L^2(\bR,\bC^2)$, we apply \cite[(Lemma 1,
p.137)]{Henry}. Observe that again, the proof of this result uses that both
vectors $\begin{pmatrix} 1\\ \alpha_p(s) \end{pmatrix}$ and $\begin{pmatrix}
1\\ \omega_n(s) \end{pmatrix}$ are linearly independent and generate the
complete space $\bC^2$. \hfill \end{proof}

\par\bigskip
\begin{nota}
Behind the proof above we have implicitly used  the fact that the operator
$F^s_\infty$ has an ``exponential dichotomy'' in the whole real line $\bR$.  We
refer to \cite{Cop,Pal,San02} for literature relating exponential dichotomies
and Fredholm operators.
\end{nota}

\bigskip
\begin{nota}
Observe that the proof that $A^s_r$ is Fredholm of index 0 is valid for all
$s\in \bC$, while the proof that $A^s_\infty$ is Fredholm of index 0 uses in a
decisive way that $\real s >\max\{f'(0),f'(1)\}$. This is related to the fact
that the essential spectrum of $L^\infty$ is contained in $\{ z\in\bC : \real
z>\max\{f'(0),f'(1)\}\}$, see Remark \ref{essential-spectrum}.
\end{nota}

\par\bigskip
\begin{theorem}\label{th:conver_regular}
For every fixed $\veps>0$, there exists an $r_0>0$ such that for all $r\geq
r_0$ we have $\sigma(L_r)\cap
\{ \real z >-\nu+\veps \} =\{s(r)\}.$ Moreover, $s(r)<0$ is a simple eigenvalue
of
$L_r$ and $s(r)\to 0$ as $r\to +\infty$. In particular, the unique stationary
solution of \eqref{modifpb_interval} is asymptotically stable.
\end{theorem}
\begin{proof} Observe first that from Corollary \ref{bound-spectrum-Lr} we have that there exists  $R_0>0$ large enough and independent of $r$ such that
$\sigma(L_r)\cap \{ \real z >-\nu\}\subset  \{|z|\leq R_0\}$ and therefore, the
part of the spectrum of $L^r$ with $\real z >-\nu$  is uniformly bounded.
Hence, from now on in the proof of this theorem, we will only consider $s\in
\bC$ with $|s|\leq R_0$ and $\real s >-\nu$.

Observe that $s\in \sigma(L_r)$ if and only if $Ker(A_r^s)\ne 0$. Moreover,
notice that if $s$ is such that $\real s > -\nu$, then  $Ker(A_\infty^s)\ne
\{0\}$ if and only if $s$ is an eigenvalue of $L^\infty$. Hence, from the
spectral analysis performed above for $L^\infty$, we have that
$Ker(A_\infty^s)=\{0\}$ for all $s$ with $\real s>-\nu$ except for $s=0$, for
which $Ker(A_\infty^s)$ is one dimensional and it is generated by the vector
function $(\varphi_\infty',\varphi_\infty'')$.

Let us calculate the ``root subspace'' associated to $s_0=0$. Following Theorem
\ref{eigenvalue-convergence} we have that $e_0=(\vfi'_\infty,\vfi''_\infty)$.
To calculate $e_1$, we need to solve $A_\infty^0 e_1=B_\infty e_0$, where
$B_\infty$ was defined above. That is,
$$A^0_{\infty}\begin{pmatrix}u\\v\end{pmatrix}=
\begin{pmatrix}0\\ \vfi_\infty'\end{pmatrix}
$$
which can be written as
$$
\left\{
\begin{array}{l}
u_x-v=0 \\
v_x+ f'(\vfi_\infty)u+\lambda_\infty v+\pi_\infty(u)\vfi'_\infty=\vfi'_\infty
\end{array}
\right.
$$
or equivalently $L^\infty_0 (u)=u_{xx}+f'(\vfi)u+\lambda_\infty
u_x=\vfi'_\infty-\pi_\infty(u)\vfi'_\infty\in \hbox{span}(\vfi'_\infty)$ which
has no solution since $\vfi'_\infty$ is an algebraically simple eigenfunction of
$L^\infty_0$. Hence
$$
W(0)=\hbox{span}\left\{\begin{pmatrix}\vfi'_\infty \\
\vfi''_\infty\end{pmatrix}\right\}\qquad \hbox{ and }\qquad
\hbox{dim}(W(0))=1.
$$

Therefore, from the $\cPQ$ regular convergence of $A_r^s$ to $A_\infty^s$ given
by Proposition \ref{pr:conver_regular} and applying the definition of $W_r(s_0,
\delta)$ and the results from Theorem \ref{eigenvalue-convergence}, we have the
following:

\par\noindent i) All values $s(r)\in \bC$ with $\real s(r) >-\nu $ and
$|s(r)|\leq R_0$ for which $Ker(A_r^{s(r)})\ne \{0\}$ satisfy $s(r)\to 0$ as
$r\to +\infty$.

\par\noindent ii) There exists a $\delta>0$ small such that for $r$ large enough we have
dim$(W_r(0,\delta))=1$.

In particular $s(r)$ from i) is a real number since if it were a complex
number, then its complex conjugate $\bar s(r)$ would also satisfy
$Ker(A_r^{\bar s(r)})\ne \{ 0\}$, since if $A_r^{\bar s(r)}(u,v)=(0,0)$ and
$\imag \bar s(r) \ne 0$, then  $\imag (u,v)\ne
(0,0)$. Moreover, since all coefficients of $A_r^{s(r)}$ are real (except for
$s(r)$), we will have $A_r^{\bar s(r)} (\bar u,\bar v)=(0,0)$ and therefore we
will have at least two numbers, $s(r)$ and $\bar s(r)$ in the set $\{ s: \real
s>-\nu, \quad Ker(A_r^s)\ne \{ 0\} \} $. From i) we will have that both of them
have to aproach 0 and therefore, dim$(W_r(0,\delta))\geq 2$, which is a
contradiction
with ii). This shows that $s(r)\in \bR$. Hence, $s(r)$ is a real eigenvalue
of $L_r$, but from Proposition \ref{propo:autopositivo} we have that $s(r)<0$.
\hfill \end{proof}

\section{Numerical experiments and open problems}\label{sec_numerical}
In this section we propose numerical examples that show the efficiency of the
methods analyzed in this paper. Notice that the implementation of numerical
methods necessarily passes by the truncation of the interval and the use of
certain numerical schemes applied to the truncated equation. The stability
analysis carried out for the equation in a bounded domain in the previous
section shows that if the time is large enough but finite an appropriate
initial data will be close to the unique stationary solution. Moreover, for
this fixed time, if the numerical scheme is convergent then for an appropriate
refinement of the discretization mesh, we will obtain a numerical approximation
of the equilibria which, in turn, will also be an approximation of the
travelling wave. Which is the main goal of this paper.

Observe that at this moment, one step further may be given which consists on
the analysis of the dynamics of the equations obtained by the discretization,
maybe discretizing both space and time or just discretizing only the space. The
analysis of the dynamics of the resulting discrete equations and the comparison
with the dynamics of the continuous equations is an interesting subject that
will be analyzed in a forthcoming work.

We consider the prototypical Nagumo equation
\begin{equation}\label{nagumo}
\left\{\begin{array}{l}
u_t = \displaystyle u_{xx} + u(1-u)(u-\alpha), \quad x\in \bR, \ t>0,\quad  \alpha \in \Big(0,\frac 12\Big), \\[5pt]
u(x,0) = u_0(x), \quad x\in \bR,
\end{array}\right.
\end{equation}
for which an explicit travelling wave solution $u(x,t) = \Phi(x-ct)$ is known,
namely
\begin{equation}\label{exactsol}
\Phi(x) = \frac{1}{1+ e^{-x/\sqrt{2}}}, \qquad c =
\sqrt{2}\,\bigg(\alpha - \frac 12 \bigg), \qquad x\in \bR.
\end{equation}
%We notice that in case $\alpha = 1/2$, $c=0$ and $\bv$ is a asymptotic equilibrium of \eqref{nagumo}.
%This case is also equivalent to the Allen-Cahn equation, for which
%$$
%f(u)=u-u^3 = u(1-u)(1+u),
%$$
%if we look for travelling waves between -1 and 1 rather than between 0 and 1 (see the examples in \cite{KaTre05}).
As we have already mentioned in the Introduction, we only consider
here monotonic increasing waves, taking values 0 at $-\infty$ and 1
at $+\infty$, since the other case, that of monotonic decreasing
fronts, can be dealt with by simply changing variables from $x$ to
$-x$.

It is clear that $f(u)=u(1-u)(u-\alpha)$ fulfils the hypotheses of
Theorems~\ref{th:FiMc} and \ref{th:v} and we can use the change of
variables \eqref{v} to approximate both the asymptotic travelling
front $\Phi$ and its propagation speed. In this way, for fixed $J>0$, we
consider the numerical integration of \eqref{modifpb_interval-intro} in the
interval $[-J,J]$, this is
\begin{equation}\label{modif_nagumo}
\left\{\begin{array}{l}
v_t = \displaystyle v_{xx} - \frac{F(1)}
{\| v_x(\cdot,t)\|^2_{L^2(-J,J)}}\, v_x + v(1-v)(v-\alpha),  \quad x\in
[-J,J], \ t>0, \\[1em]
v(x,0) = v_0(x), \quad x\in[-J,J], \\[5pt]
v(-J,t) = 0, \quad v(J,t) = 1, \quad t > 0.
\end{array}\right.
\end{equation}

We applied the method of lines to integrate \eqref{modif_nagumo} up
to time $T=150$, for $\alpha= 1/4$. For the spatial discretization,
we use standard finite differences formulas, centered for the
approximation of $v_x$, on the uniform grid $x_j = -J + j \Delta x$,
$1\le j\le M-1$, for $\Delta x = 2J/M$, and different values of $J$
and $M$. The nonlocal term $\lambda_v$ is approximated by using the
scalar product of the vector with the values of $v_x$ at the grid points $x_j$.

For the time integration of the spatially semidiscrete problem we
use the MATLAB solver {\tt ode15s}. Since we are interested in
computing both $v$ and $\lambda_v$, it is convenient to reformulate
\eqref{modif_nagumo} as a Partial Differential Algebraic Equation
\begin{equation}\label{PDAE_nagumo}
\left\{\begin{array}{l}
v_t = v_{xx} + \lambda_v v_x + v(1-v)(v-\alpha),  \quad x\in [-J,J], \ t>0, \\[5pt]
0 = \lambda_v \langle v_x,v_x \rangle + F(1), \ t>0, \\[5pt]
v(x,0) = v_0(x), \quad x\in[-J,J], \\[5pt]
v(-J,t) = 0, \quad v(J,t) = 1, \quad t > 0,
\end{array}\right.
\end{equation}
The results shown in this section are
obtained with the following options for the solver:
\begin{equation}
\begin{array}{l}
\mbox{\tt e = ones(M-1,1);} \\
\mbox{\tt Massmatrix = spdiags([e;0],0,M,M);} \\
\mbox{\tt options = odeset(`Mass',Massmatrix,`MassSingular',`yes',...}\\
\hspace{2em} \mbox{\tt `MStateDependence',`none',...}\\
\hspace{2em} \mbox{\tt
`AbsTol',[1e-8*ones(M-1,1);1e-10],`RelTol',1e-8);}.
\end{array}
\end{equation}
In the Figures below we show the results obtained for different
possible choices of the initial data $u_0(x)$. These plots show how
for $J$ and $M$ large enough the solution to the discrete
problem seems to converge exponentially fast in time to a
stationary state close to the stationary state of \eqref{modifpb}. The
same happens with the value of the propagation speed. These numerical results
are in agreement with the ones reported in \cite{BeTh04,Th05} and, up to
a great extent, are to be expected from our theoretical analysis.
However, let us notice that while Theorem~\ref{th:v} guarantees convergence to
the equilibrium of the problem in the whole line for a wide class of initial
data, Theorem~\ref{th:nonloc_interval} guarantees convergence for the problem
in a bounded interval only for initial data close enough to the equilibrium
and in a large enough interval. How ``close'' and ``large'' are ``enough'', is
not really specified.

Several additional questions arise now related with the asymptotic behavior of
problem \eqref{PDAE_nagumo} as $t\to \infty$. For instance, we know that the
unique equilibrium of \eqref{modifpb_interval-intro} approaches the unique
equilibrium of \eqref{modifpb2} as the interval grows to $\bR$. A natural
question now is what the rate of convergence with repect to the length on the
interval is. With the notation in this section, this is the convergence with
respect to $J$.

Concerning the boundary conditions, other variants are meaningful and could in
principle lead to a faster convergence to equilibrium, such us homogeneous
boundary conditions of Neumann type or even more sophiscated conditions
like transparent boundary conditions. Notice that then problems \eqref{modifpb}
and \eqref{modifpb2} lead to different IVBP.

Another issue is the effect of the numerical approximation. %A natural question
%is whether the analysis performed in Sections~4 and 5 can be extended to the
%spatial discretisation of \eqref{modifpb_interval-intro} by using say finite
%differences. The existence of a unique stationary solution of the resulting
%semidiscrete problem and its asymptotic stability is an open question.
One could address for instance the study of the speed of convergence towards
the asymptotic profiles depending on the chosen numerical scheme, the mesh
size, etc. In particular, it would be natural to analyze the question of
whether upwinding yields better convergence rates. The same questions make
sense for fully discrete approximation schemes.

Finally, let us notice how the worst approximation results displayed
in Figure~\ref{errckt} illustrate the importance of capturing
properly the front of the asymptotic profile. In other words, the
importance of controlling the value of the phase $x_1$ in
\eqref{conver_v}. A careful study of the dependence of this location
on the initial data $u_0$ is in order.

All these questions are beyond the
scope of the present paper.

\begin{ilustracion}\label{ex_linini}
{\rm We consider the linear initial data
\begin{equation}
u_0(x) = \frac{x+J}{2J}, \qquad x \in [-J,J].
\end{equation}
The results plotted in Figures~\ref{vlin} and \ref{errclin} were
obtained with $J=40$, $\Delta x = 0.1$ and $\Delta x= 0.025$.

\begin{figure}
  \centering
    \includegraphics[viewport= 44    44   372   346,width=0.45\textwidth]{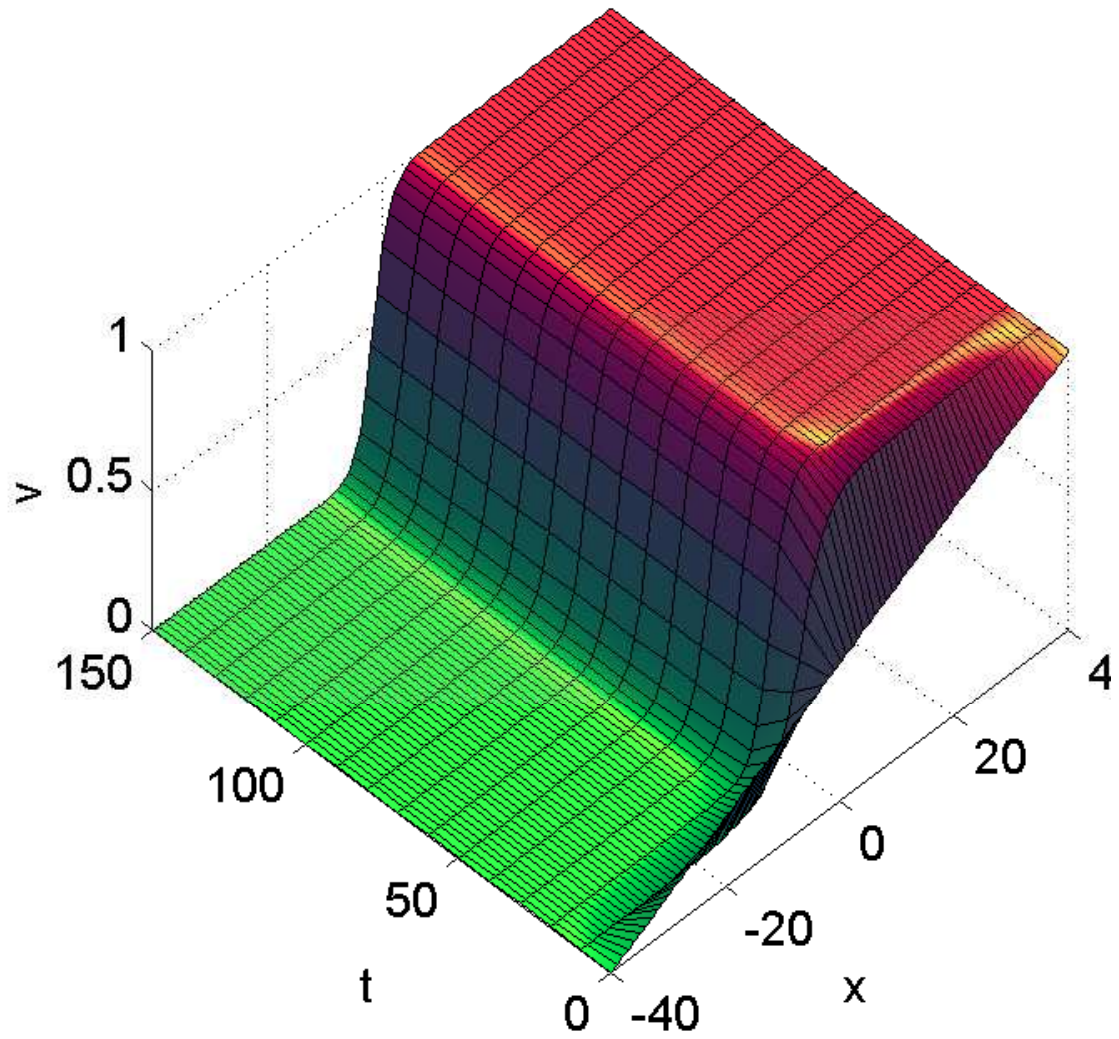}
    \includegraphics[viewport=32    29   372   354, width=0.40\textwidth]{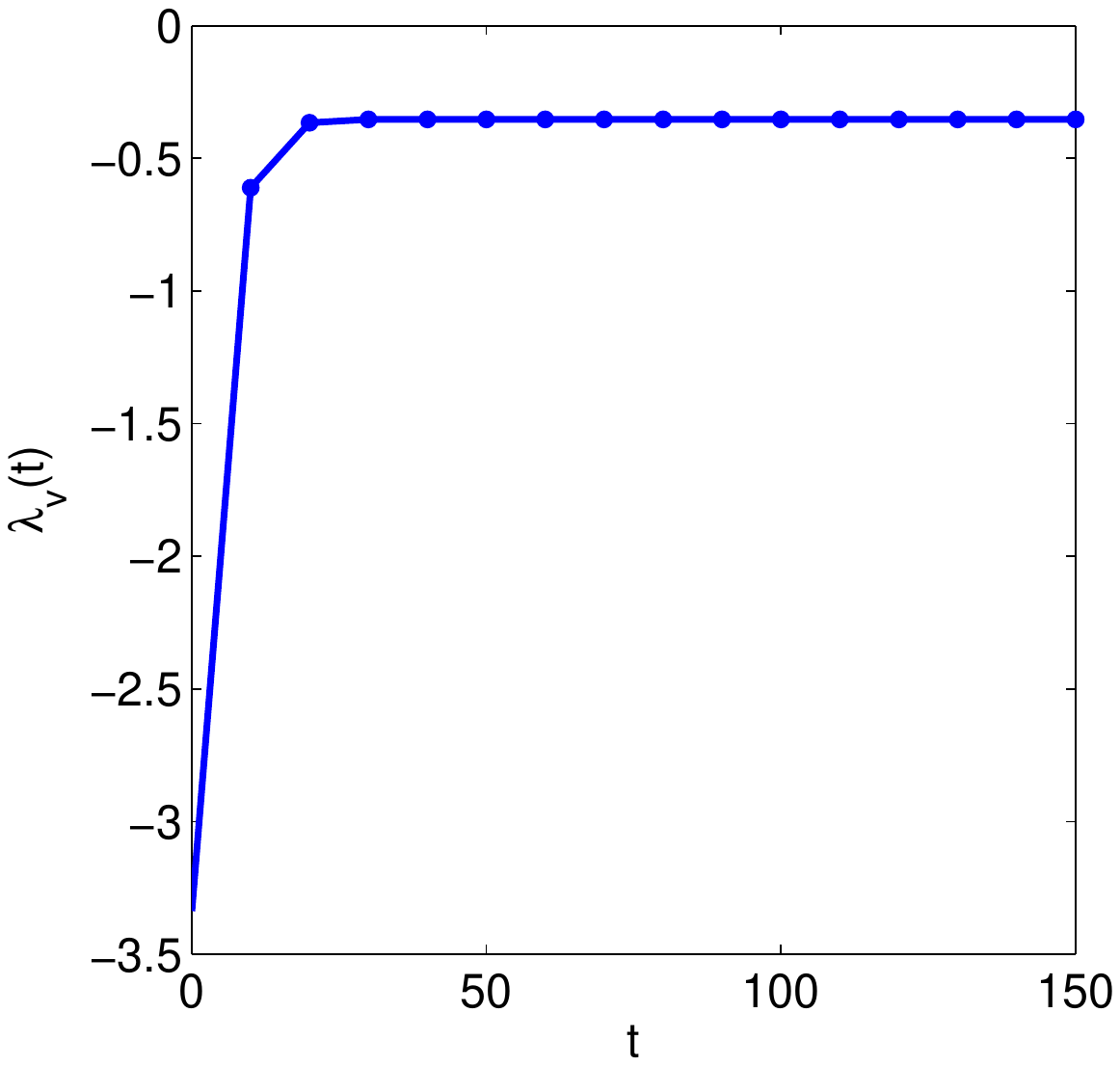}
  \caption{Solution for $u_0$ in Example~\ref{ex_linini} (left) and evolution of $\lambda_v$ (right)
for $J=40$ and $\Delta x = 0.1$. }
  \label{vlin}
\end{figure}

\begin{figure}
  \centering
     \includegraphics[viewport=35    29   372   358 ,width=0.45\textwidth]{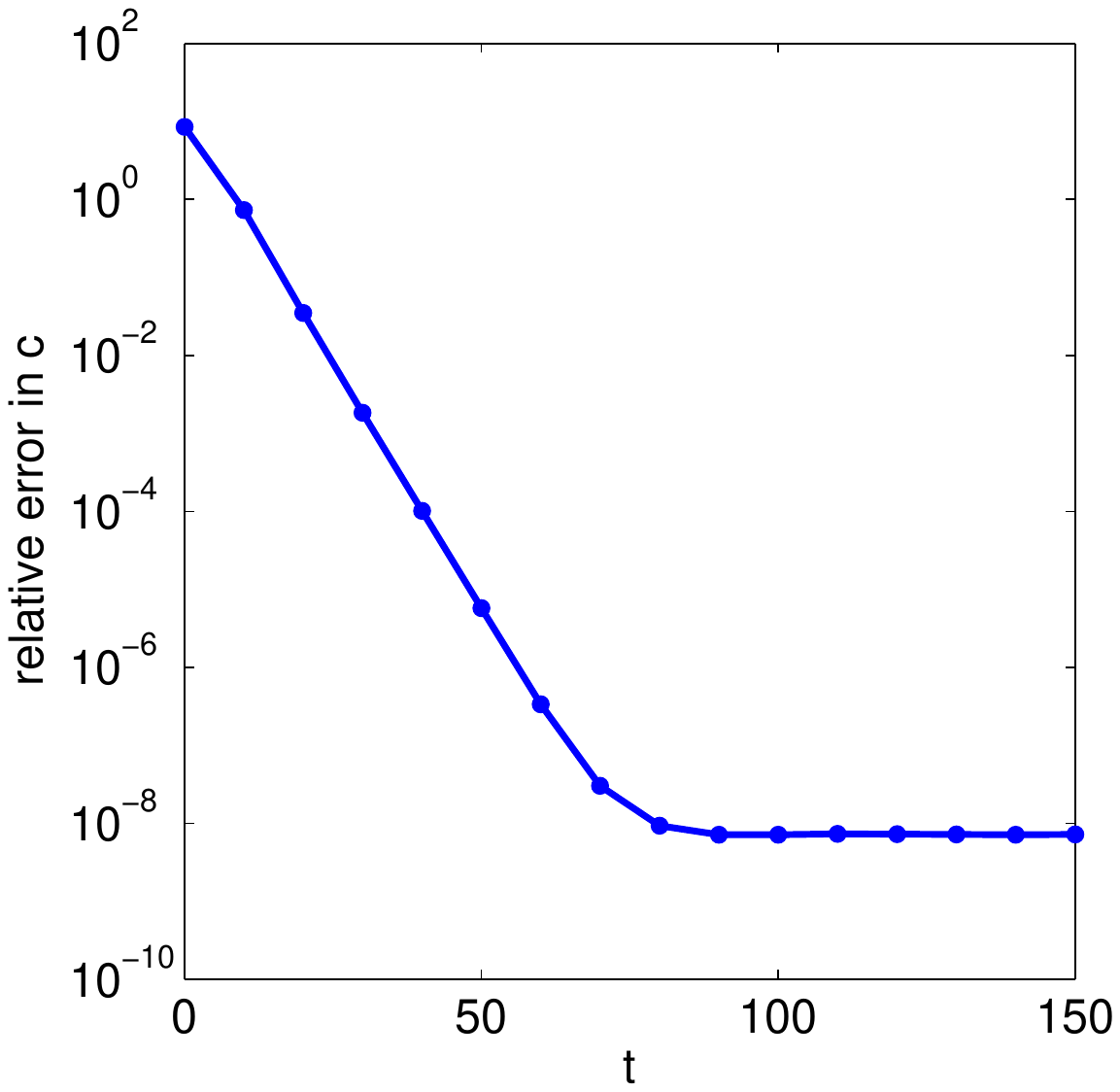}
     \includegraphics[viewport= 35    29   372   358,width=0.45\textwidth]{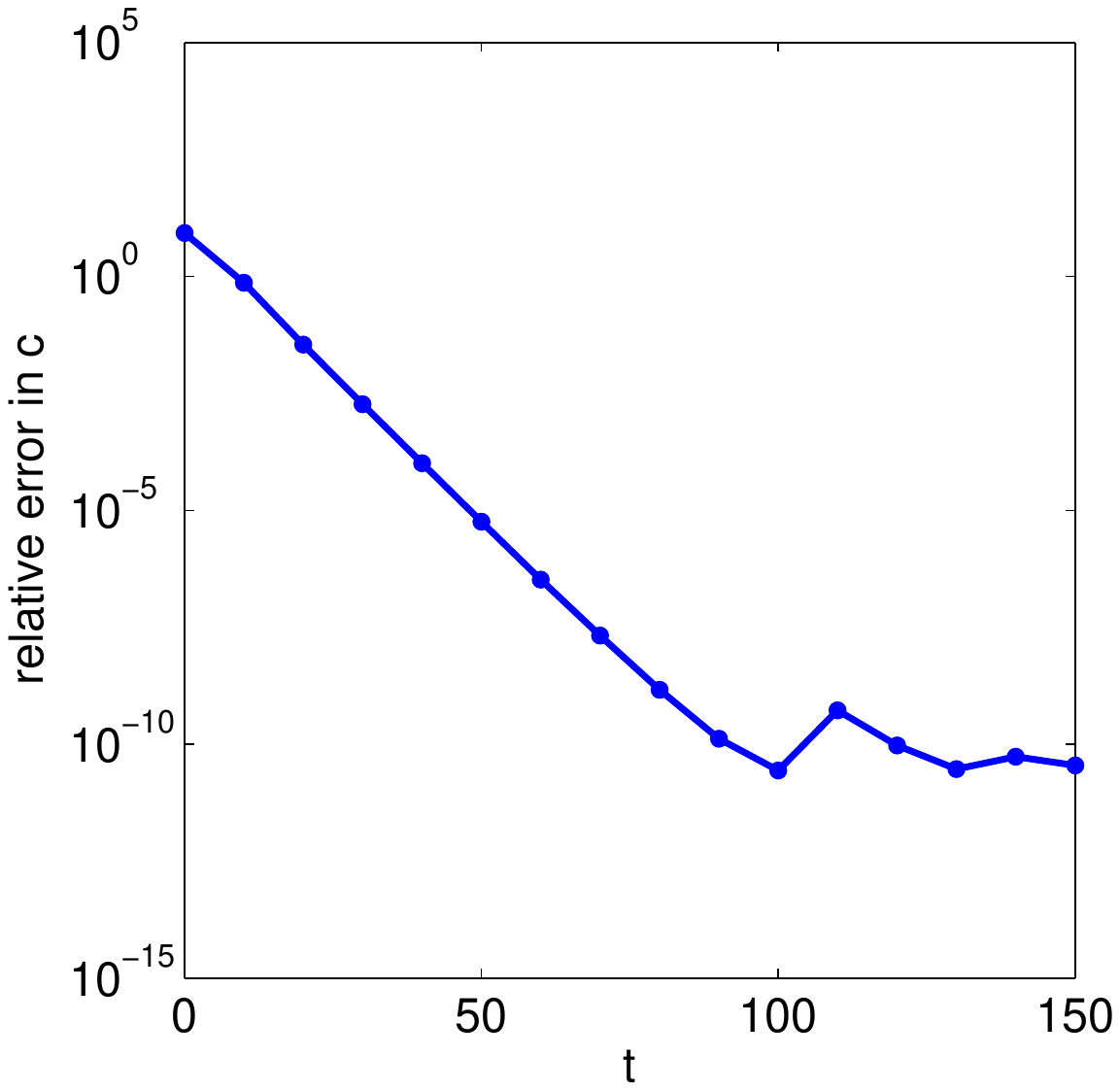}
  \caption{Error in the approximation of $c$ for Example~\ref{ex_linini}.
Left: $\Delta x=0.1$, Right: $\Delta x=0.025$.}
  \label{errclin}
\end{figure}

}

\end{ilustracion}

\begin{ilustracion} \label{ex_kt}
{\rm We consider the initial data
\begin{equation}
u_0(x) = \frac 12\Big(1 + 0.53 \frac{x}{J} + 0.47\sin
\Big(-\frac{3\pi x}{2J}\Big)\Big).
\end{equation}

\begin{figure}
  \centering
   \includegraphics[viewport= 44    44   372   346,width=0.45\textwidth]{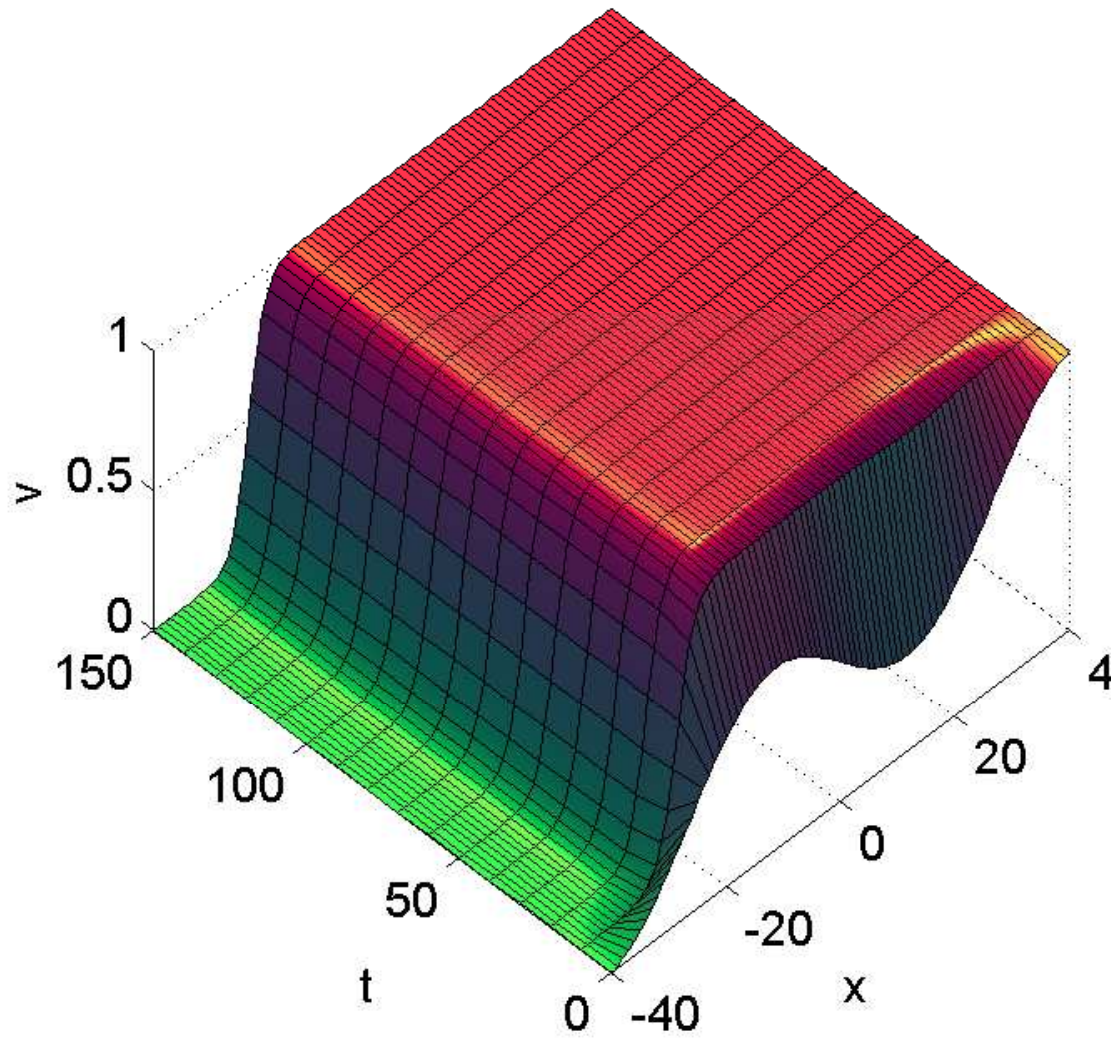}
  \includegraphics[viewport= 32    29   372   354,width=0.40\textwidth]{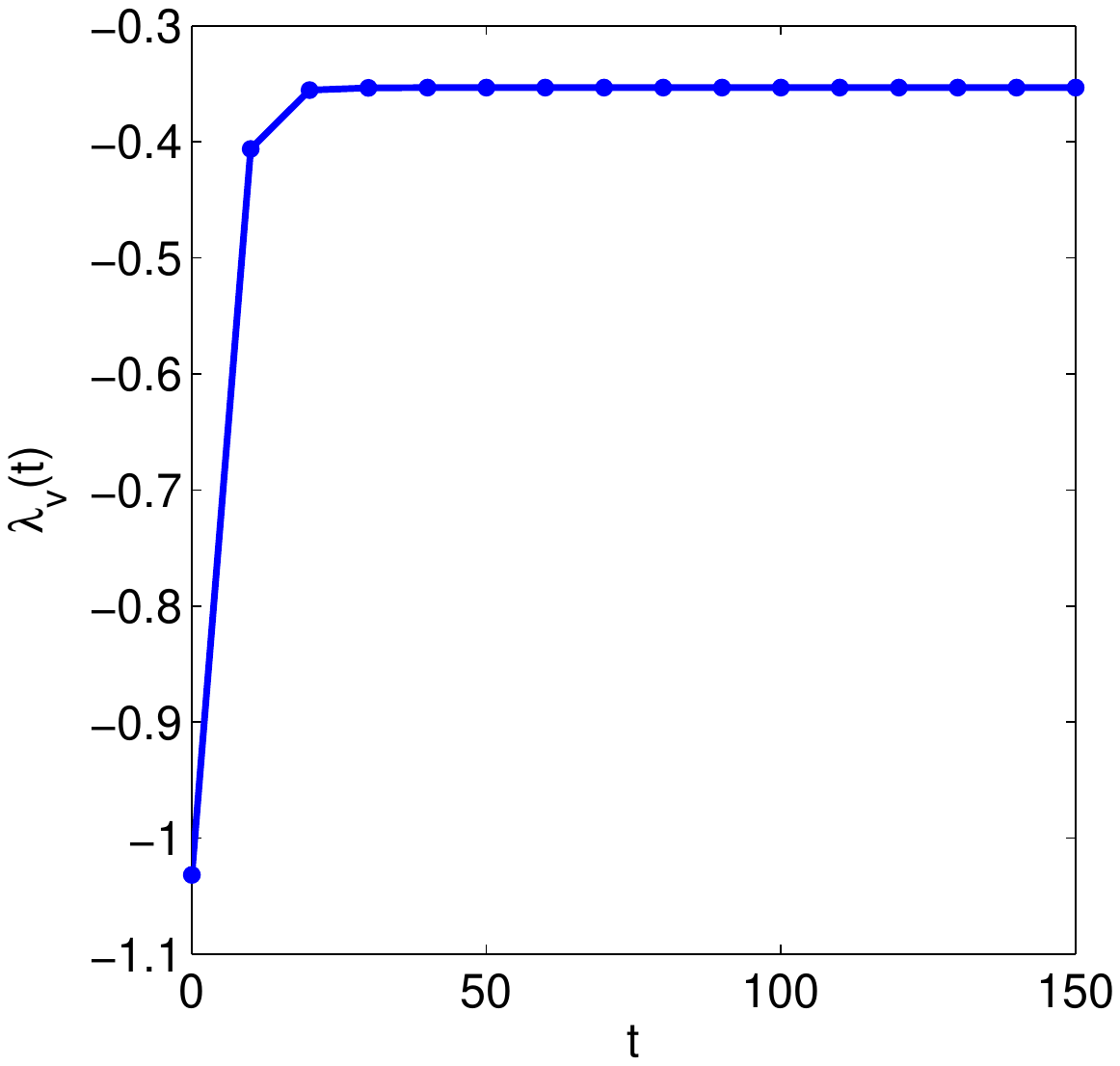}
  \caption{Solution for $u_0$ in Example~\ref{ex_kt}(left) and evolution of $\lambda_v$ (right)
for $J=40$ and $\Delta x = 0.1$.}
  \label{vkt}
\end{figure}

\begin{figure}
  \centering
    \includegraphics[viewport=41    29   372   358,width=0.45\textwidth]{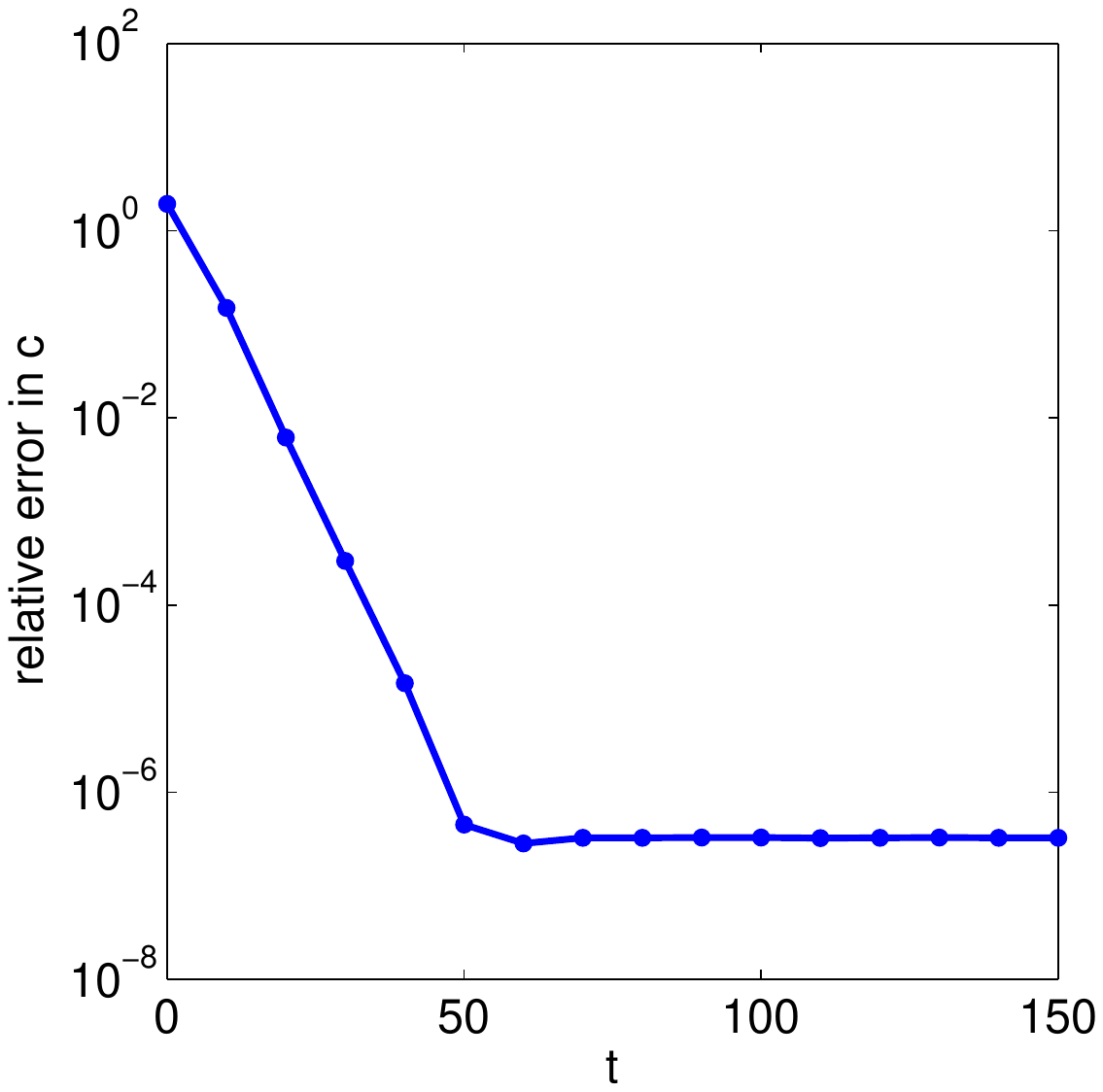}
    \includegraphics[viewport=41    29   372   358 ,width=0.45\textwidth]{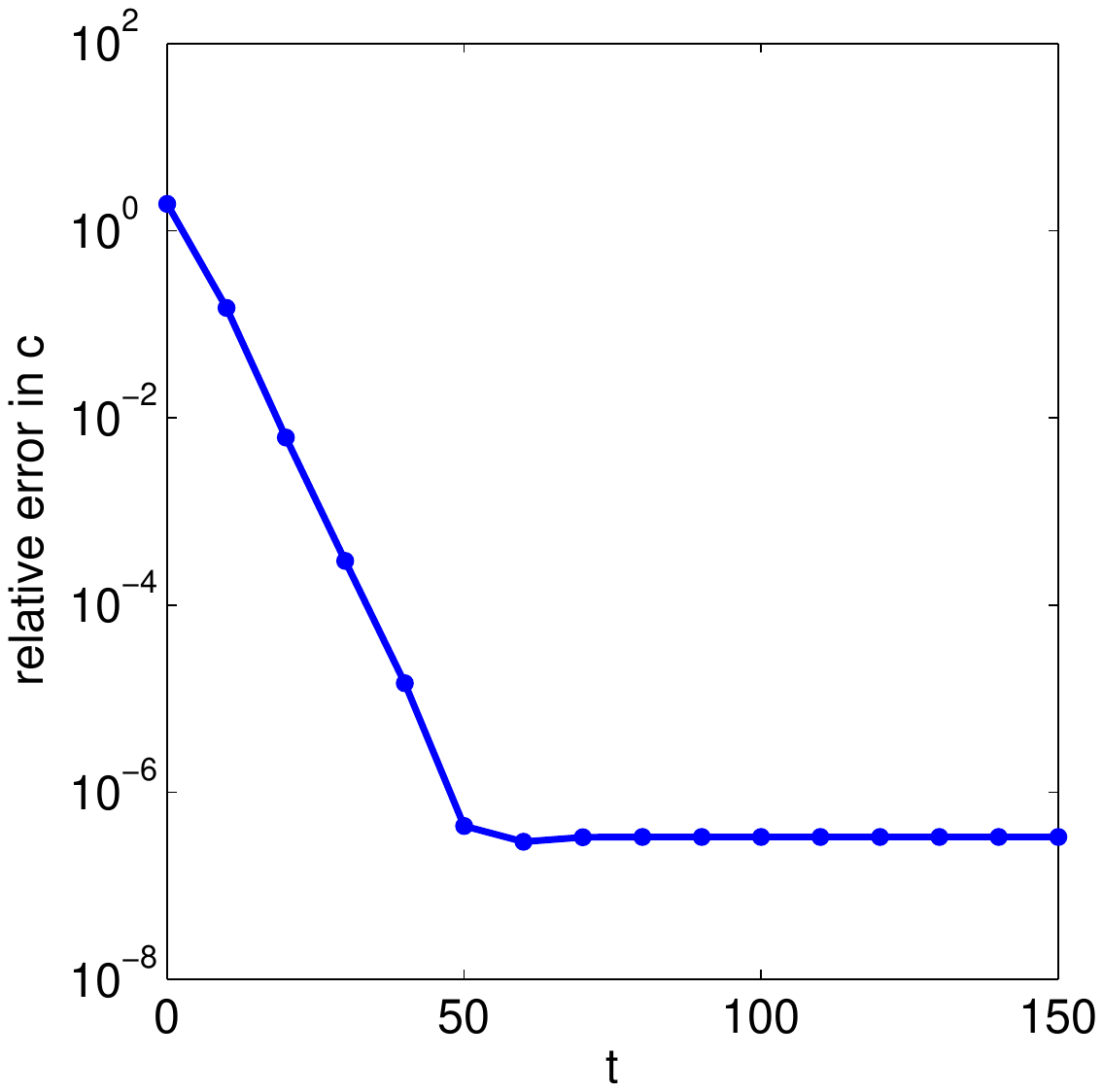}
  \caption{Error in the approximation of $c$ for Example~\ref{ex_kt}.
Left: $\Delta x=0.1$, Right: $\Delta x=0.025$.}
  \label{errckt}
\end{figure}

}
\end{ilustracion}

\begin{ilustracion} \label{ex_nbc}
{\rm We consider the initial data
\begin{equation}
u_0(x) = \left\{ \begin{array}{ll} 0.2, \qquad & \mbox{if }\ x < 0 \\
0.8, \qquad & \mbox{if }\ x > 0 .
\end{array}\right.
\end{equation}
In this case $u_0$ does not satisfy the boundary conditions. We chose this
example because if we consider the natural extension of $u_0$
to the whole line (by 0.8 to the right and 0.2 to the left), Theorem~\ref{th:v}
guarantees the convergence of the solution of \eqref{nagumo} to $\Phi$ in
\eqref{exactsol}.
%The results shown in Figures~\ref{vnbc} and \ref{errcnbc} are in good agreement
%with the expected behavior, although Theorem~\ref{th:nonloc_interval} is not
%enough to guarantee the convergence to the equilibrium of \eqref{modif_nagumo}.

\begin{figure}
  \centering
    \includegraphics[viewport=44    44   372   346,width=0.45\textwidth]{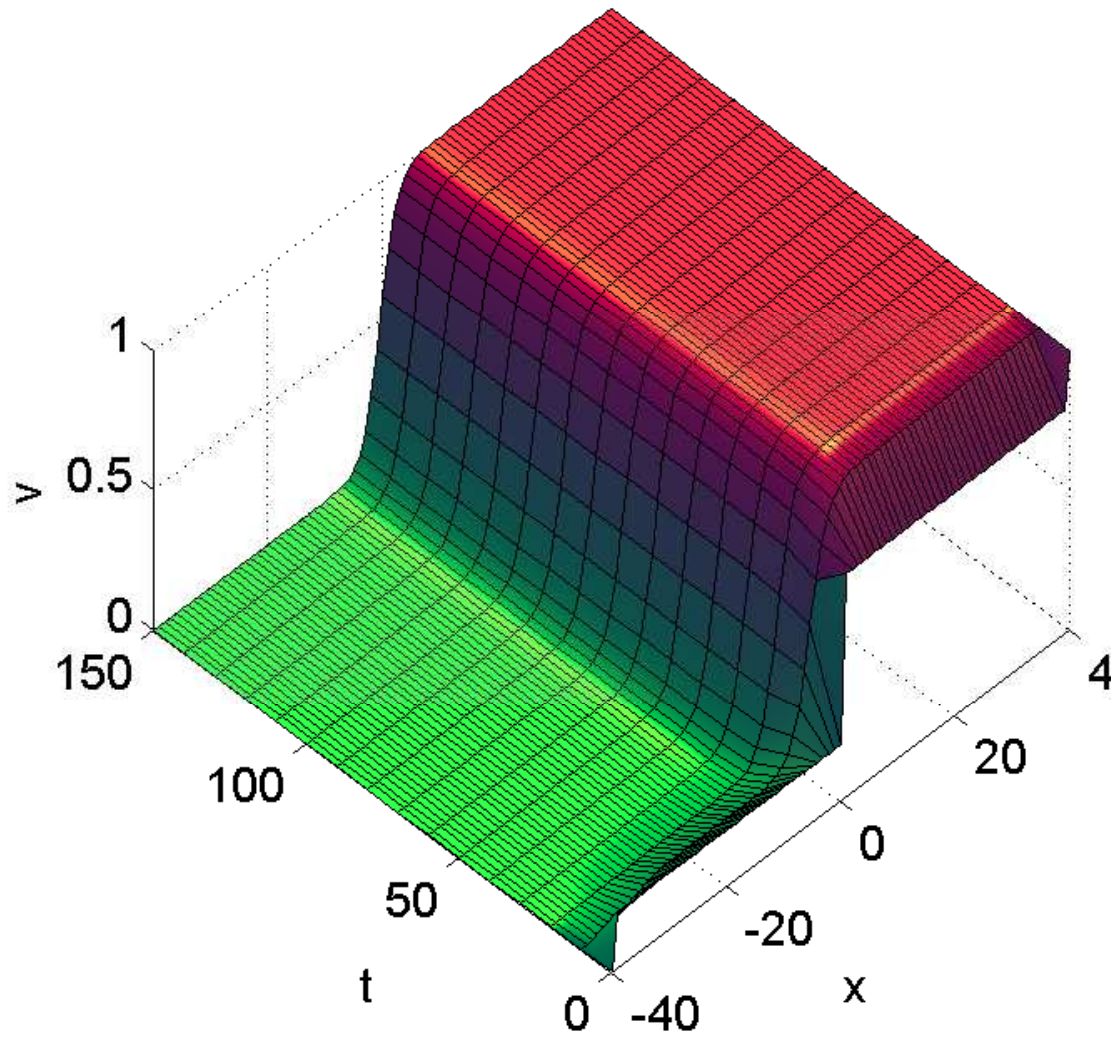}
   \includegraphics[viewport= 28    29   372   355,width=0.40\textwidth]{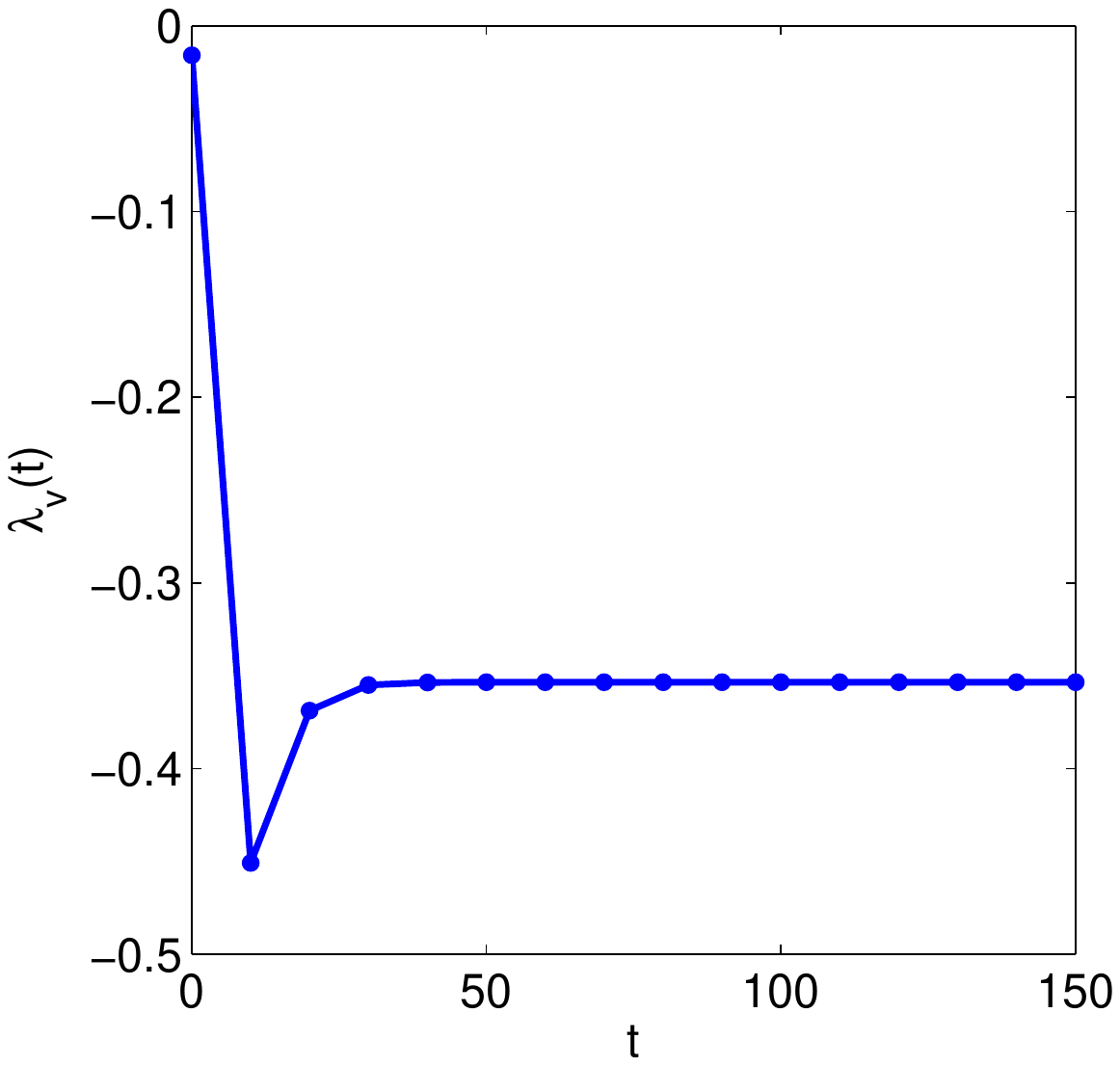}
  \caption{Solution for $u_0$ in Example~\ref{ex_kt} (left) and evolution of $\lambda_v$ (right)
for $J=40$ and $\Delta x = 0.1$.}
  \label{vnbc}
\end{figure}

\begin{figure}
  \centering
    \includegraphics[viewport=41    29   372   363 ,width=0.45\textwidth]{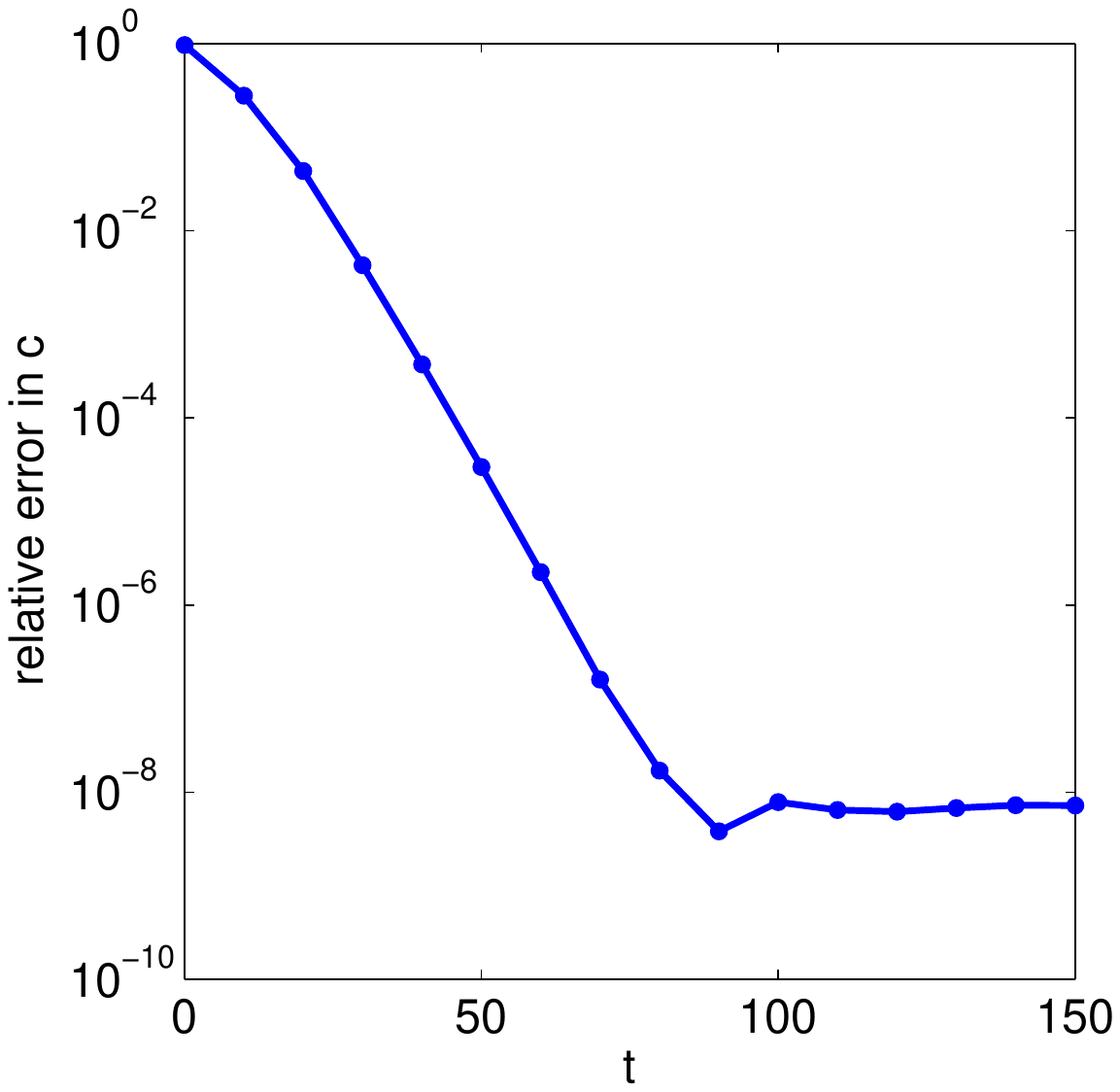}
    \includegraphics[viewport= 41    29   372   364,width=0.45\textwidth]{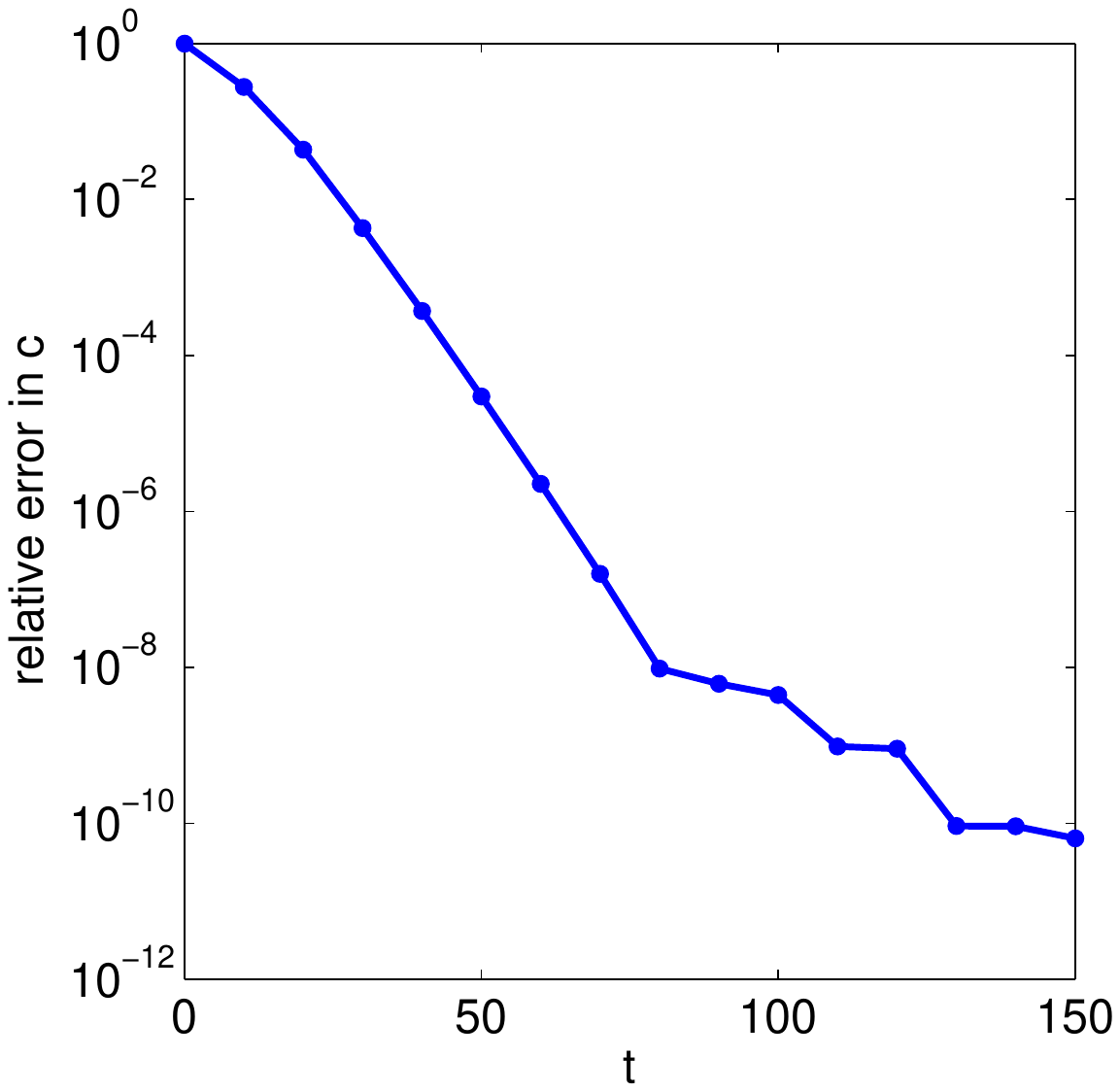}
  \caption{Error in the approximation of $c$ for Example~\ref{ex_kt}.
Left: $\Delta x=0.1$, Right: $\Delta x=0.025$.}
  \label{errcnbc}
\end{figure}

}
\end{ilustracion}

\appendix
\section{Technical lemma}\label{appendix-a}
We include in this appendix a technical result, where we study in detail the behavior at
$\pm \infty$ of the bounded solutions to a certain kind of second order
differential equations with variable coefficients.

\begin{lemma}\label{lema:asymp_ode}
Let $c\in \bR$ and let us consider the second order (non homogeneous) scalar
differential equation
\begin{equation}\label{ode_varcoef}
\psi''(x) + c \psi'(x) + a(x) \psi(x) = f(x), \qquad x_0 < x <
\infty,
\end{equation}
where
\par\noindent i)   $a(x)$ is a bounded piecewise continuous function satisfying
\begin{equation} \label{cota_b}
\limsup_{x\to+\infty}|a(x)-a|\leq M_1e^{-\theta x},
\end{equation}
\par\noindent ii) the function $f(x)$ satisfies
\begin{equation} \label{cota_f}
|f(x)|\leq M_2e^{-\tau x},
\end{equation}
where $\theta,\tau>0$. Then any bounded solution $\psi$ of \eqref{ode_varcoef}
tends to 0 as $x\to
\infty$. Moreover, if $0<\omega<\min \{\theta, -r_1, \tau\}$, with
$r_1 = \frac{-c-\sqrt{c^2-4a}}{2} < 0$, then
$$
|\psi(x)|, \, |\psi'(x)|, \, |\psi''(x)|  \le M e^{-\omega x}, \quad  \mbox{ for }\ x\ge x_2,
$$
for some constant $M>0$.
\end{lemma}
\begin{proof}
Let us rewrite the equation \ref{ode_varcoef} as
\begin{equation}\label{ode_diff}
\psi'' + c \psi' + a \psi = (a-a(x))\psi(x)+f(x) := b(x),
\end{equation}
and observe that from i) and ii) we have
$|b(x)|\leq Me^{-\gamma x}$ for some $M>0$ and with $\gamma=\min\{ \tau, \theta\}>0$.

%Let us consider the limit equation of \eqref{ode_varcoef} as $x \to \infty$
%\begin{equation}\label{ode_lim}
%\chi''(x) + c \chi'(x) + a \chi(x) = 0, \qquad x_0< x < \infty.
%\end{equation}
The roots of the characteristic equation associated to the homogeneous equation of  \eqref{ode_diff} are
precisely
\begin{equation}\label{roots}
r_1 = \frac{-c-\sqrt{c^2-4a}}{2} < 0 \qquad \mbox{ and } \qquad
r_2 = \frac{-c+\sqrt{c^2-4a}}{2} > 0.
\end{equation}

By the variation of constants formula, any solution $\psi$ of \eqref{ode_diff}
is of the form
\begin{equation}\label{vfi_general}
\psi(x) = Ce^{r_1x} + De^{r_2 x} + \frac{1}{r_1-r_2}
\int_{x_0}^x (e^{r_1(x-s)}-e^{r_2(x-s)})b(s)\,ds,
\end{equation}
for $C, D\in\bR$. If require further
that $\psi$ is bounded, the only possible choice for $D$ is
\begin{equation}\label{D}
D = \frac{1}{r_1-r_2}
\int_{x_0}^{\infty} e^{-r_2s} b(s)\,ds,
\end{equation}
leading to
\begin{equation}\label{vfi_bounded}
\psi(x) = Ce^{r_1x} + \frac{e^{r_1 x}}{r_1-r_2} \int_{x_0}^x
e^{-r_1 s} b(s)\,ds + \frac{e^{r_2 x}}{r_1-r_2}
\int_{x}^{\infty} e^{-r_2 s} b(s)\,ds.
\end{equation}
But, if $r_1+\gamma\ne 0$, then
$$
\left|\int_{x_0}^x
e^{-r_1 s} b(s)\,ds \right|\leq M \int_{x_0}^x
e^{-(r_1+\gamma) s} \leq M
\frac{e^{-(r_1+\gamma)x_0}-e^{-(r_1+\gamma)x}}{r_1+\gamma}
$$
and if $r_1+\gamma=0$, then
$$
\left|\int_{x_0}^x
e^{-r_1 s} b(s)\,ds \right|\leq M(x-x_0).
$$
Moreover, since $r_2+\gamma>0$, we have
$$
\left|\int_{x}^\infty
e^{-r_2 s} b(s)\,ds \right|\leq M \int_{x}^\infty
e^{-(r_2+\gamma) s} \leq M \frac{e^{-(r_2+\gamma)x}}{r_2+\gamma}.
$$
Plugging these estimates in \eqref{vfi_bounded}, and with some simple
computations we obtain that, if $r_1+\gamma\ne 0$ then
$$
|\psi(x)|\leq C_1 e^{r_1 x}+ C_2e^{-\gamma x}+C_3e^{-\gamma x}
$$
and if $r_1+\gamma=0$, then
$$
|\psi(x)|\leq C_1 e^{r_1 x}+ C_2e^{-\gamma x}(x-x_0)+C_3e^{-\gamma x},
$$
from where the conclusion for $\psi$ follows easily.

To obtain the bounds for $\psi'(x)$ we just take derivatives in
\eqref{vfi_bounded}. We obtain an extra term, $b(x)$, and the rest of the
terms are estimated similarly as in the case of $\psi(x)$. To estimate
$\psi''(x)$ we use the equation satisfied by $\psi$ and the bounds obtained for
$\psi(x)$ and $\psi'(x)$. \hfill \end{proof}

\begin{nota}\label{nota-apendice}
The same conclusions of the previous Lemma hold if we are dealing with the
interval $-\infty<x<x_1$. In this case we need to specify the behavior of the
functions $a(\cdot)$ and $f(\cdot)$ as $x\to -\infty$ and the conclusion is the
exponential decay of the solution as $x\to -\infty$.
\end{nota}

\par\bigskip

\section{Proof of Lemma \ref{operatorA}} \label{proof-lemma}
Finally, we include in this appendix a proof of Lemma
\ref{local-operator-behavior}.
\begin{proof}
$(i)$ and $(ii)$ follow from \cite[Section 5.4 and Appendix A]{Henry}.

$(iii)$ Observe first that we know the behavior of $\vfi_\infty(x),
\vfi'_\infty(x)$ as $x\to \pm \infty$. Notice that the orbit $x\to
(\vfi_\infty(x), \vfi'_\infty(x))$ is the heteroclinic orbit connecting $(0,0)$
(as $x\to -\infty$) with $(1,0)$ (as $x\to +\infty$) of the ODE,
\begin{equation}
\left\{
\begin{array}{l}
U'=V,     \\
V'=-\lambda V-f(U)
\end{array}
\right.
\end{equation}
and therefore the orbit lies in the unstable manifold of $(0,0)$ and the stable
manifold of $(1,0)$. Via linearization of the equation in $(0,0)$ and $(1,0)$
we can obtain that if we define
$$
r_1=\frac{-\lambda-\sqrt{\lambda^2-4f'(1)}}{2}, \qquad
r_2=\frac{-\lambda+\sqrt{\lambda^2-4f'(0)}}{2},
$$
we have
$$
|\vfi_\infty(x)|, |\vfi'_\infty(x)|\leq Ce^{r_2 x}, \hbox{ as }x\to - \infty,
$$
$$
|\vfi_\infty(x)-1|, |\vfi'_\infty(x)|\leq Ce^{r_1 x}, \hbox{ as }x\to +\infty
$$
and therefore
$$
|f(\vfi_\infty)|, |f'(\vfi_\infty(x))-f'(0)|\leq Ce^{r_2 x}, \hbox{ as }x\to -
\infty
$$
and
$$
|f(\vfi_\infty)|, |f(\vfi'_\infty(x))-f'(1)|\leq Ce^{r_1 x}, \hbox{ as }x\to
+\infty.
$$

Using the equation for $\vfi_\infty$, that is, $\vfi_\infty''=-\lambda
\vfi'_\infty-f(\vfi_\infty)=0$, we get also
$$
|\vfi''_\infty(x)|\leq Ce^{r_2 x}, \hbox{ as }x\to - \infty,\quad
|\vfi''_\infty(x)|\leq Ce^{r_1 x}, \hbox{ as }x\to +\infty.
$$

Applying Lemma \ref{lema:asymp_ode} and Remark \ref{nota-apendice}, we have
that if there exists a function $w$ such that $L_0^\infty w=\vfi'_\infty$, then
$$
|w(x)|, |w'(x)|, |w''(x)|\leq Ce^{r_2^- x}, \hbox{ as }x\to - \infty,
$$
and
$$
|w(x)|,|w'(x)|, |w''(x)|\leq Ce^{r_1^- x}, \hbox{ as }x\to +\infty.
$$
where $0<r_2^-<r_2$ and $r_1<r_1^- < 0$ but arbitrarily close to $r_2$ and
$r_1$, respectively.

Once this estimates have been obtained, we can perform the change of variables
$v(x)=e^{\frac{\lambda}{2}x}w(x)$ which will be a function in $H^2(\bR)$,
because of the estimates found above for $w,w',w''$. Therefore, $v$ will be a
solution of
\begin{equation}\label{selfadjoint-equation}
v''+ \left(f'(\vfi_\infty(x))+\frac{|\lambda|^2}{2} \right)v=\chi_\infty(x),
\end{equation}
where $\chi_\infty(x)=e^{\frac{\lambda}{2}x}\vfi'_\infty(x)$, which is a
function in $L^2(\bR)$ because of the exponential bounds obtained for
$\vfi'_\infty (x)$ and  it is an eigenfunction of the operator $v\to
v''+ \left(f'(\vfi(x))+\frac{|\lambda|^2}{2} \right)v$ associated to the
eigenvalue $0$. But
this operator is selfadjoint and therefore there cannot exist a solution of
equation \eqref{selfadjoint-equation}. \hfill
\end{proof}

\bibliographystyle{amsplain}

\end{document}